\documentclass[a4paper,oneside,11pt]{article}
\usepackage{a4wide, dsfont}
\usepackage{graphicx}
\usepackage{array}
\usepackage{booktabs}
\usepackage{tikz}
\usepackage{bbm}   

\usepackage[latin1]{inputenc}

\usepackage{amsmath,amssymb,amsthm}
\usepackage{url}
\usepackage[colorlinks, linkcolor = black, citecolor = black, filecolor = black, urlcolor = blue]{hyperref} 

\usepackage{ifthen}
\usepackage{natbib}
\usepackage{color}
\usepackage{caption}
\usepackage{subcaption}
\usepackage{arydshln}
\usepackage{graphicx}
\usepackage{longtable}
\usepackage{marvosym}
\usepackage{cases}
\usepackage{enumitem}
\usepackage{amsmath}
\newcommand{\leadingzero}[1]{\ifnum #1<10 0\the#1\else\the#1\fi}

\usepackage{marginnote}

\newcounter{temp1}

\newtheorem{satz}{Satz}

\theoremstyle{definition}

\newtheorem{example}{Example}
\newtheorem{remark}[example]{Remark}
\newtheorem{assumption}{Assumption}

\theoremstyle{plain}
\newtheorem{theorem}[satz]{Theorem}
\newtheorem{proposition}[satz]{Proposition}

\newtheorem{corollary}[satz]{Corollary}
\newtheorem{lemma}{Lemma}
\newcommand*{\defeq}{\mathrel{\vcenter{\baselineskip0.5ex \lineskiplimit0pt
			\hbox{\scriptsize.}\hbox{\scriptsize.}}}%
	=}

\newcommand{\f}[1]{{\mathbf{#1}}}

\newcommand{\ovec}{{\operatorname{vec}}}
\newcommand{\ov}{{\operatorname{v}}}

\newcommand{\tr}{{\operatorname{trace}}}

\newcommand{\var}{{\operatorname{\mathbb{V}\mathrm{ar}}}}
\newcommand{\cov}{{\operatorname{\mathbb{C}\mathrm{ov}}}}
\newcommand{\cor}{{\operatorname{\mathbb{C}\mathrm{or}}}}

\newcommand{\op}[1]{{\operatorname{o_\mathbb{P}}\left(#1\right)}}
\newcommand{\Op}[1]{{\operatorname{\mathcal{O}_\mathbb{P}}\left(#1\right)}}

\newcommand{\one}{\mathbbm{1}}

\newcommand{\normz}[1]{{\left\lVert#1\right\rVert_2}}
\newcommand{\normi}[1]{{\left\lVert#1\right\rVert_\infty}}
\newcommand{\norme}[1]{{\left\lVert#1\right\rVert_1}}

\newcommand{\normzq}[1]{{\left\lVert#1\right\rVert_2^2}}
\newcommand{\normzv}[1]{{\left\lVert#1\right\rVert_2^4}}

\newcommand{\normzM}[1]{{\left\lVert#1\right\rVert_{\mathrm{M},2}}}

\newcommand{\normFM}[1]{{\left\lVert#1\right\rVert_{\mathrm{F}}}}

\newcommand{\E}{{\mathbb{E}}}
\newcommand{\R}{{\mathbb{R}}}
\newcommand{\N}{{\mathbb{N}}}
\newcommand{\X}{{\mathbb{X}}}
\newcommand{\Y}{{\mathbb{Y}}}
\newcommand{\Prob}{{\mathbb{P}}}

\newcommand{\orthproj}[1]{{\mathrm{P_{#1^{\perp}}}}}
\newcommand{\proj}[1]{{\mathrm{P_{#1}}}}

\newcommand{\sign}{{\operatorname{sign}}}

\newcommand{\rank}{{\operatorname{rank}}}
\newcommand{\supp}{{\operatorname{supp}}}

\newcommand{\init}{{\mathrm{init}}}

\newcommand{\AL}{{\mathrm{AL}}}

\newcommand{\LS}{{\mathrm{LS}}}

\newcommand{\Ymu}{{\Y_n^{\mu}}}
\newcommand{\Xmu}{{\X_n^{\mu}}}
\newcommand{\XmuS}{{\X_{n,S_\mu}^{\mu}}}
\newcommand{\epsmu}{{\varepsilon_n^{\mu}}}

\newcommand{\Ysigma}{{\Y_n^{\sigma}}}
\newcommand{\Xsigma}{{\X_n^{\sigma}}}
\newcommand{\XsigmaS}{{\X_{n,S_\sigma}^{\sigma}}}
\newcommand{\epssigma}{{\varepsilon_n^{\sigma}}}

\newcommand{\bx}{\f{x}}

\newcommand{\bA}{\f{A}}

\newcommand{\bw}{\f{w}}
\newcommand{\bW}{\f{W}}
\newcommand{\bv}{\f{v}}
\newcommand{\bu}{\f{u}}

\newcommand{\Cmul}{{c_{\,\mathrm{C}^{\mu},\mathrm{l}}}}
\newcommand{\Cmuu}{{c_{\,\mathrm{C}^{\mu},\mathrm{u}}}}
\newcommand{\Bmuu}{{c_{\,\mathrm{B}^{\mu},\mathrm{u}}}}

\newcommand{\Csigmal}{{c_{\,\mathrm{C}^{\sigma},\mathrm{l}}}}
\newcommand{\Csigmau}{{c_{\,\mathrm{C}^{\sigma},\mathrm{u}}}}
\newcommand{\Bsigmau}{{c_{\,\mathrm{B}^{\sigma},\mathrm{u}}}}

\newcommand{\tauvX}{{\tau_{\ov(\f{X})}^2}}

\newcommand{\tauA}{{\tau_{\f{A}}^2}}

\begin{document}
	
%
%%%%%%%%%%%%%%%%%%%%%%%%%%%%%%%%%%%%%%%%%%%%%%%%%%%%%%%%%%%%%%%%%%%%%%%%%%%%%%%
%%%%%%%%%%%%%%%%       Titel   %%%%%%%%%%%%%%%%%%%%%%%%%%%%%%%%%%%%%%%%%%%%%%%%
%%%%%%%%%%%%%%%%%%%%%%%%%%%%%%%%%%%%%%%%%%%%%%%%%%%%%%%%%%%%%%%%%%%%%%%%%%%%%%%
%
{\title{Bounded support in linear random coefficient models: Identification and variable selection}%: Variational Problems, Compression, and Noise Removal}
\author{Philipp Hermann and Hajo Holzmann\footnote{Corresponding author. Prof.~Dr.~Hajo Holzmann, Department of Mathematics and Computer Science, Philipps-Universität Marburg, Hans-Meerweinstr., 35043 Marburg, Germany} \\
\small{Department of Mathematics and Computer Science}  \\
\small{Philipps-Universität Marburg} \\
\small{\{herm, holzmann\}@mathematik.uni-marburg.de}}

\maketitle
%
		
%\textbf{Hajo Holzmann\footnote{Address for correspondence: Prof.~Dr.~Hajo Holzmann,
%Philipps-Universität Marburg,
%Fachbereich Mathematik und Informatik,
%Hans-Meerwein-Straße.
%D-35032 Marburg, Germany,
%Email: holzmann@mathematik.uni-marburg.de,
%Fon: + 49 6421 2825454
%}
% }\\*[0.4cm]
%
%{\sl Fachbereich Mathematik und Informatik, Philipps-Universität Marburg, Germany.}
}

\begin{abstract}
We consider linear random coefficient regression models, where the regressors are allowed to have a finite support. First, we investigate identifiability, and show that the means and the variances and covariances of the random coefficients are identified from the first two conditional moments of the response given the covariates if the support of the covariates, excluding the intercept, contains a Cartesian product with at least three points in each coordinate. We also discuss ientification of higher-order mixed moments, as well as partial identification in the presence of a binary regressor. Next we show the variable selection consistency of the adaptive LASSO for the variances and covariances of the random coefficients in finite and moderately high dimensions. This implies that the estimated covariance matrix will actually be positive semidefinite and hence a valid covariance matrix, in contrast to the estimate arising from a simple least squares fit. We illustrate the proposed method in a simulation study.    		

\end{abstract}
		
%\textbf{Keywords:}\ Random Coefficients, Endogeneity, Consumer Demand,
		
	%
	%%%%%%%%%%%%%%%%%%%%%%%%%%%%%%%%%%%%%%%%%%%%%%%%%%%%%%%%%%%%%%%%%%%%%%%%%%%%
	%%%%%%%%%%%%%%%%%%                                           %%%%%%%%%%%%%%%
	%%%%%%%%%%%%%%%%%%    Introduction                           %%%%%%%%%%%%%%%
	%%%%%%%%%%%%%%%%%%                                           %%%%%%%%%%%%%%%
	%%%%%%%%%%%%%%%%%%%%%%%%%%%%%%%%%%%%%%%%%%%%%%%%%%%%%%%%%%%%%%%%%%%%%%%%%%%%
	%
	%\input{lest}
	%\bibliographystyle{elsarticle-harv}
	%

\section{Introduction}

In various statistical analyses in fields such as medicine and economics, there is a large extent of individual heterogeneity in the effect of observed covariates, which is routinely modeled by random coefficients  - also called random effects - models. For example, in contemporary microeconomic data sets with many observations and potentially a large number of explanatory variables, non-observed heterogeneity plays an important role \citep{lewbel2005modeling}.  
An important issue then is to select those coefficients which actually are random if there is a large set of potential variables which might have individual - specific effects. 

To this end, in this paper we shall consider the following random coefficients regression model 
\begin{align}
	Y =  B_0 + \f{W^\top} \f{B}  \,,  \label{random:coefficient:model:intercept}
\end{align}
where $\f{B}, \f{W} \in \R^{p-1}$ are independent random vectors, $B_0$ is a random variable and $\f{W}=(W_1,\dotsc,W_{p-1})^\top$ represents the random regressors.  

Model \eqref{random:coefficient:model:intercept}, which is related to random effects models from the literature on biostatistics \citep{schelldorfer2011estimation}, was introduced by \citet{Hildreth1968} and \citet{Swami1970}. They assumed that $B_0, \ldots, B_{p-1}$ are independent, and focused on estimating their means and variances by least squares in two stages. \citet{arellano2012identifying} studied a panel - version of the random coefficient model. 
 \citet{beran1992} initiated the nonparametric analysis of the distribution of the random coefficients.  For $p=2$, \citet{beran1996} used Fourier methods to construct an estimator of the joint density of $(B_0, B_1)^\top$. Their method was taken up again by \citet{Hoderlein2010}, who put it into the form of a more conventional kernel estimator and generalized it to arbitrary dimension $p$. Further related literature includes \citet{Gautier2013}, who analyze a binary choice version of the model,  \citet{lewbel2017unobserved} who study a generalization of \eqref{random:coefficient:model:intercept} in which the products $B_1 \, W_1, \ldots ,B_{p-1}\, W_{p-1}$ are related to $Y$ by some arbitrary (possibly non-linear) unknown function, as well as  \citet{hoderlein2017triangular}, \citet{breunig2018specification}, \citet{dunker2019tests} and \citet{holzmann2020rate}. Recently, \citet{gaillac2019adaptive} studied nonparametric identification and adaptive estimation in a random coefficient regression model, where covariates have bounded but continuous variation. 

The  above nonparametric approaches which target the full density of the random coefficients require a large or at least, as in \citet{gaillac2019adaptive}, continuous support of the covariates, which is often an unrealistic assumption in applications. In this paper, we shall focus on situations in which the covariates have bounded and in particular finite support. 
In this latter setting there is little hope to identify and estimate the density of the random coefficients nonparametrically. 
Therefore, we shall focus on the first and second moments, which are arguably of most interest in applications. Variable selection techniques for means, variances and covariances of the random coefficients then allow to determine which variables have an effect on average (non-zero mean of the coefficient), which variables have heterogenous effects (non-zero variances) and for which covariates the effects are correlated. In particular, we shall argue that it is important not to focus exclusively on the variances of the random coefficients, but to take the full variance-covariance matrix into account.
Further, estimating the first and second moments of the random coefficients then allows to predict the first and second moments of the response $Y$ conditional on the covariates. Finally, normality of the random coefficients is a common parametric assumption, under which their distribution is fully determined by means, variances and covariances.

\medskip

Model \eqref{random:coefficient:model:intercept} is related to random effects models from the literature on biostatistics \citep{schelldorfer2011estimation}. These are studied in a longitudinal framework, and the goal is then to estimate the fixed effects by using a quasi-likelihood approach and to predict the random effects. Papers which study these models in a high-dimensional setting are, among others, \citet{schelldorfer2011estimation} and  \citet{li2021inference}.

\medskip

The paper is organized as follows. In Section \ref{sec:model:random:coefficients} we clarify  under which assumptions on the support of the covariates,  first and second moments of the random coefficients  are identified. It turns out that identification holds if the support of the covariate vector contains a Cartesian product with at least three support points for each covariate. Conversely, identification generally fails if one covariate only has two support points. 
%that a necessary and sufficient condition for identification is . 
In Section \ref{ch:fixedp}  we turn to estimation and in particular to variable selection with a focus on the variances and covariances in model \eqref{random:coefficient:model:intercept}. We use the adaptive LASSO originally introduced in \citet{zou2006adaptive}, which may achieve variable selection consistency without additional restrictive assumptions such as the irrepresentable assumption required for the ordinary LASSO, and show the variable selection consistency in fixed and moderately high dimensions. The technical issues are to deal with the residuals when estimating centered second moments of the random coefficients as well as with the heteroscedasticity of the model.] Section \ref{simulations} contains some numerical illustrations. Proofs of the main results are given in Section \ref{sec:mainproofs}, while some further auxiliary results are deferred to the supplementary appendix. 

We shall use the following notation: For an $n \times p$ matrix $\mathbb{X}$ and a subset $S \subseteq \{1, \ldots, p\}$ of the index set, $\mathbb{X}_S$ denotes the $n \times |S|$ matrix containing those columns of $\mathbb{X}$ with indices in $S$. A similar notation is $\boldsymbol{v}_S$ for a vector $\boldsymbol v \in \R^p$. $\normzM{\mathbb{X}}$ denotes the operator norm of $\mathbb{X}$ for the Euclidean norm, and $\normFM{\mathbb{X}}$ the Frobenius norm, that is the Euclidean norm of the vectorization of $\mathbb{X}$. 

%\citet{Tibshirani1996} 
%
%\citet{Vandegeer2009} and \citet{Foucart2013} discuss these assumptions and their relationship to each other. Moreover, \citet{Zhao2006} showed that for sign-consistency of the LASSO an additional assumption, commonly called irrepresentable or incoherence condition, is required, and \citet{Wainwright2009} introduced the primal-dual witness characterization of the LASSO and gives sufficient and necessary conditions for sparsity recovery under independent sub-Gaussian errors. \citet{Zou2006} proposed the adaptive LASSO because then the irrepresentable condition can be relaxed for variable selection and the estimators enjoys the oracle properties under homoscedasticity and a fixed number of coefficients. For a growing number \citet{Huang2008} provide asymptotics, \citet{Wagener2012} and \citet{Wagener2012} extend these asymptotic results for heteroscedastic errors.
%

%{\color{red} Introduce some relevant notation, in particular $\mathbb{X}_A$ for a subset $A$ of the column index set.} 

\section{Identification of first and second moments} \label{sec:model:random:coefficients}

In model \eqref{random:coefficient:model:intercept} we also write
$\f{A} = (B_0,\f{B^\top})^\top \in \R^p$, so that $\f{A} = (A_1,\dotsc,A_p)^\top$ and $\f{W}$ are independent. 
%In this context and $\f{A}=(A_1,\dotsc,A_p)^\top$ the random regression coefficients.
%
 We assume that the first and second moments of the random coefficients $\f{A}$ exist and set 
%
%\begin{align}
%	\mu^* \defeq \E[\f{A}] \quad \in \R^{p} \label{def:mean:true}
%\end{align}
%and 
\begin{align}
	\mu^* \defeq \E[\f{A}] \quad \in \R^{p}\,, \qquad \text{and} \qquad 	\Sigma^* \defeq \cov(\f{A}) = \E\big[(\f{A}-\mu^*) (\f{A}-\mu^*)^\top] \quad \in \R^{p \times p}\,.  \label{def:cov:true}
\end{align}
In this section we consider conditions for identification and partial identification of the moments $\mu^*$ and $\Sigma^*$ in terms of the support of the covariates $\f{W}$. 

While one may argue that means and variances are of main applied interest, the joint variation of the random coefficients as described by the covariances and the correlations is also relevant. Further, we shall see that excluding covariances from the analysis and falsely assuming a diagonal covariance matrix a-priori can lead to wrong conclusions about the (non-)randomness of the coefficients. The proofs of the results in this section are collected in Section \ref{sec:identification:proofs}. 

%%%%%%%%%%%%%%%%%%%%%%%%%%%%%%%%%%%%%%%%%%%%%%%%%%%%%%%%%%%%%%%%%%%%%%%%%%%%%%%%%%%%%%%%%%%%%%%%%%%%%%
%%%%%%%%%%%%%%%%%%%%%%%%%%%%%%%%%%%%%%%%%%%%%%%%%%%%%%%%%%%%%%%%%%%%%%%%%%%%%%%%%%%%%%%%%%%%%%%%%%%%%%
%%%%%%%%%%%%%%%%%%%%%%%%%%%%%%%%%%%%%%%%%%%%%%%%%%%%%%%%%%%%%%%%%%%%%%%%%%%%%%%%%%%%%%%%%%%%%%%%%%%%%%
%%%%%%%%%%%%%%%%%%%%%%%%%%%%%%%%%%%%%%%%%%%%%%%%%%%%%%%%%%%%%%%%%%%%%%%%%%%%%%%%%%%%%%%%%%%%%%%%%%%%%%
%%%%%%%%%%%%%%%%%%%%%%%%%%%%%%%%%%%%%%%%%%%%%%%%%%%%%%%%%%%%%%%%%%%%%%%%%%%%%%%%%%%%%%%%%%%%%%%%%%%%%%

%\subsection{Illustration for a single regressor}
%
To illustrate, first consider the case of a single regressor, resulting in the model 
\begin{equation}\label{eq:randomcoefficients1}
	Y = B_0 + W_1 \, B_1 \,.
\end{equation}
For $W_1$ supported on $\{0,1\}$ we have the following result. 
\begin{proposition}\label{prop:twosupportpointsfirst}
	Suppose that in model (\ref{eq:randomcoefficients1}) the random variable $W_1 \in \{0,1\}$ is binary, and denote the identified standard deviations by 
	\[ s_1 = \sqrt{\var (B_0) } \,,\qquad s_2 = \sqrt{ \var (B_0+ B_1) } \,.\]
	Then each value 
	\[ \sqrt{\var( B_1) } \in \big[| s_1 - s_2|,s_1 + s_2 \big] \]
	is consistent with $s_1$ and $s_2$, provided the correlation $\rho = \cor(B_0,B_1)$ is chosen for $\sqrt{\var( B_1) } >0$ as 
	\begin{equation}\label{eq:boundcorrelation}
		\rho = \frac{s_2^2 - s_1^2 - \var( B_1) }{2\,  s_1 \sqrt{\var( B_1) }} \in 
		\begin{cases}
[-1,1], & \text{if } s_2 > s_1 \,,\\ 
\big[-1, -\sqrt{s_1^2 - s_2^2}/s_1\big] & \text{if } s_1 \geq s_2 \,.
		\end{cases}  
	\end{equation}
\end{proposition}
%
%\edithajothree[Extend this discussion, relate to Biostatistics literature.]
Thus, to conclude from $\var (B_0) = \var(B_0 + B_1)$ that 
$\var(B_1) = 0 $ fully relies on the assumption of a diagonal covariance matrix, without this assumption, $B_1$ can well be random.

On the other hand, the following proposition shows that three distinct support points of $W_1$ are enough to identify the means $\E [B_j]$, the variances $\var(B_j)$, $j=0,1$, and the covariance $\cov(B_0,B_1)$. From Proposition \ref{prop:twosupportpointsfirst} and not surprisingly, two support points are insufficient for this purpose.

\begin{proposition}\label{prop:firstsimpleprop}
	In model \eqref{eq:randomcoefficients1}, if $W_1$ has $n+1$ support points and $\E \big[ |B_0|^n \big], \E \big[|B_1|^n\big]<\infty$, then all mixed moments $\E \big[B_0^j\, B_1^k\big]$, $j,k \geq 0$, $j+k \leq n$, are identified. 
\end{proposition}

%\begin{proof} %\hfill\\
%Cf. Section \ref{sec:identification:proofs}.
%\end{proof}
%
%\begin{remark}
%\end{remark}
%

%%%%%%%%%%%%%%%%%%%%%%%%%%%%%%%%%%%%%%%%%%%%%%%%%%%%%%%%%%%%%%%%%%%%%%%%%%%%%%%%%%%%%%%%%%%%%%%%%%%%%%
%%%%%%%%%%%%%%%%%%%%%%%%%%%%%%%%%%%%%%%%%%%%%%%%%%%%%%%%%%%%%%%%%%%%%%%%%%%%%%%%%%%%%%%%%%%%%%%%%%%%%%
%%%%%%%%%%%%%%%%%%%%%%%%%%%%%%%%%%%%%%%%%%%%%%%%%%%%%%%%%%%%%%%%%%%%%%%%%%%%%%%%%%%%%%%%%%%%%%%%%%%%%%
%%%%%%%%%%%%%%%%%%%%%%%%%%%%%%%%%%%%%%%%%%%%%%%%%%%%%%%%%%%%%%%%%%%%%%%%%%%%%%%%%%%%%%%%%%%%%%%%%%%%%%
%%%%%%%%%%%%%%%%%%%%%%%%%%%%%%%%%%%%%%%%%%%%%%%%%%%%%%%%%%%%%%%%%%%%%%%%%%%%%%%%%%%%%%%%%%%%%%%%%%%%%%

\subsection{Identification of the covariance matrix} \label{sec:identification:fullpoint}
Now let us turn to the identification of $\mu^* $ and $\Sigma^*$  in \eqref{def:cov:true} in general dimensions. To this end, consider the half-vectorization of symmetric matrices of dimension $p \times p$,
\begin{align}\label{def:vec} 
	\ovec(M) =\big(M_{11},\dotsc,M_{pp},M_{12},\dotsc,M_{1p},M_{23},\dotsc,M_{2p},\dotsc, M_{(p-1)p} \big)^\top \in \R^{\frac{p(p+1)}{2}} 
\end{align} 
for $ M \in \R^{p \times p}$ with $M^\top = M$, and set 
\[ \sigma^* \defeq \ovec(\Sigma^*).\]
Note that the first $p$ entries of $\sigma^*$ are the variances and the remaining entries are the covariances.
In model \eqref{random:coefficient:model:intercept} we have that 
\begin{align}\label{linear:system:identifying:variances}
	\var\,\big(Y \, \big| \, \bW = \bw) = (1, \f{w^\top}) \, \Sigma^* \, (1, \f{w^\top} )^\top,
\end{align} 
so that the quadratic form in $\Sigma^*$ is identified over $(1, \f{w})$ with $\f{w}$ ranging over the support of $\f{W}$.
%
%{\color{red} Benoetigt mean identification, um identifiziert zu sein}
%
 Note that \eqref{linear:system:identifying:variances} can be written in vectorized form as 
\begin{align}\label{linear:system:identifying:variances1}
 \var\,\big(Y \, \big| \, \bW = \bw) & = \big(1, (\f{w}^2)^\top, 2 \f{w^\top}, 2 w_1 w_2,\dotsc, 2 w_1 w_{p-1}, 2 w_2 w_3,\dotsc, 2 w_{p-2} w_{p-1} \big)\, \sigma^*\,\nonumber\\
 &= \ov \big( (1, \f{w^\top })^\top \big)^\top\, \sigma^*\,,
\end{align}
where we recall that $\sigma^* = \ovec(\Sigma^*)$, and the vector transformation $\ov$ is defined by 
\begin{align}\label{def:v}
	\ov(\bx) = \big(x_1^2,\dotsc,x_p^2,2 x_1 x_2,\dotsc,2 x_1 x_p,2 x_2 x_3,\dotsc, 2 x_2 x_p,\dotsc, 2 x_{p-1} x_p \big)^\top \in \R^{\frac{p(p+1)}{2}} \,, \quad \bx \in \R^p.
\end{align}
Based on \eqref{linear:system:identifying:variances1} we can establish linear equations for the $p (p+1)/2$ entries of $\Sigma^*$ respectively $\sigma^*$. With the above notation, we may state the following basic result. %The proofs of this section are deferred to the supplement, Section \ref{sec:proofs:identification:fullpoint}.

\begin{theorem}\label{prop:necsuffcond}
	In model \eqref{random:coefficient:model:intercept} a sufficient condition for identification of the mean vector $\mu^*$ and the covariance matrix $\Sigma^*$ is the existence of $p(p+1)/2$ points $\f{w_1},\dotsc, \f{w_{p(p+1)/2}} \in \R^{p-1}$ in the support of $\bW$, for which the matrix 
\begin{equation}\label{eq:matrixident}
	 S = \bigg[ \ov \Big( (1, \f{w_1^\top} )^\top \Big) ,\dotsc, \ov \Big( (1, \f{w_{p(p+1)/2}^\top} )^\top \Big) \bigg]^\top %\quad \in \R^{\frac{p(p+1)}{2}  \times \frac{p(p+1)}{2}} 
	\end{equation} 
	of dimension $p(p+1)/2 \times p(p+1)/2$ is of full rank. This condition is also necessary for identification in the subset of full-rank covariance matrices.  
\end{theorem}
The theorem remains valid if one can show that for $m \geq p (p+1)/2$ support points, the resulting matrix $S_m$ has full rank $p (p+1)/2$. 
%\begin{proof} %\hfill\\
%	Cf. Section \ref{sec:identification:proofs}.
%\end{proof}
%
%Note that certain singular covariance matrices may also be identified from lower-rank matrices $S$. 
%
%
%\begin{remark} %\hfill\\
%	For which points is the matrix $S$ of full rank? Its determinant is a sum of monomials in the $p(p+1)/2$ coordinate variables of the vectors $\f{w_1},\dotsc, \f{w_{\frac{p(p+1)}{2}}} \in \R^{p-1}$ of degree $(p-1)^2+2(p-1)$, and characterizing its zero set, hence those points for which is requirement is not satisfied, is a formidable task even for $p=2$. Of course, given $p(p+1)/2$ support points, one may simply form the matrix $S$ and compute its determinant. More generally, given $m \geq p(p+1)/2$ support points, one may directly check whether
%	%
%	\[ S_m = \bigg[ \ov \Big( (1, \f{w_1^\top} )^\top \Big) ,\dotsc, \ov \Big( (1, \f{w_{m}^\top} )^\top \Big) \bigg]^\top \quad \in \R^{m  \times \frac{p(p+1)}{2}} \]
%	%
%	has full rank $p(p+1)/2$, which is equivalent to the $p(p+1)/2$ column vectors of $S_m$ being linearly independent in $\R^{m}$, or to 
%	%
%	$ S_m^\top S_m$ being invertible.
%\end{remark}  
%
In the following example, we show that the condition of the previous theorem can never be satisfied if one of the regressors only has two support points. 
\begin{example}\label{ex:twosupportpoints}
	Suppose that $W_1$ has only two support points $a$ and $b$ and that the joint support of $\bW$ is finite. Then the matrix $S_m$, where $m$ is the total number of support points, has rank at most $p (p+1)/2-1$. Thus, from Theorem \ref{prop:necsuffcond}, full-rank covariance matrices $\Sigma^*$ are not identified. Indeed, the matrix $S_m^\top$ contains the submatrix 
	\[ \begin{bmatrix} 1 & \dotsc & 1 & 1 & \dotsc&  1 \\ a^2 & \dotsc & a^2 & b^2 & \dotsc &  b^2 \\ 2a & \dotsc&  2a & 2b & \dotsc & 2b \end{bmatrix} \quad \in \R^{3 \times m}\,.\] 
	Evidently, this matrix is of column rank at most $2$, since there are only two distinct columns. Thus, its row rank is also at most two, which implies that the corresponding three columns in $S_m$ are linearly dependent. 
\end{example}
%
%It is thus of some interest to have a simple, sufficient condition for identifiability of the covariance matrix. We provide the following general result, which does not rely on Theorem \ref{prop:necsuffcond}, but rather uses directly the polarization formula for identification.   
%
In contrast, if each covariate has at least three support points and the joint support contains the corresponding Cartesian product, then we retain identification of $\Sigma^*$. 
\begin{theorem}\label{the:identcartprodt}
	Consider model \eqref{random:coefficient:model:intercept}. Suppose that the support of $\bW = (W_1, \dotsc, W_{p-1})^\top$ contains the Cartesian product of three points in each coordinate. Then there exist $p(p+1)/2$ support points such that the matrix $S$ in \eqref{eq:matrixident} has full rank $p(p+1)/2$ and consequently, the means and (co-)variances of the random coefficients $\bA$ are identified. Conversely, if there is a $W_j$ having only two support points, then in the full-rank covariance matrices identification fails.   
\end{theorem}
%

%%%%%%%%%%%%%%%%%%%%%%%%%%%%%%%%%%%%%%%%%%%%%%%%%%%%%%%%%%%%%%%%%%%%%%%%%%%%%%%%%%%%%%%%%%%%%%%%%%%%%%
%%%%%%%%%%%%%%%%%%%%%%%%%%%%%%%%%%%%%%%%%%%%%%%%%%%%%%%%%%%%%%%%%%%%%%%%%%%%%%%%%%%%%%%%%%%%%%%%%%%%%%
%%%%%%%%%%%%%%%%%%%%%%%%%%%%%%%%%%%%%%%%%%%%%%%%%%%%%%%%%%%%%%%%%%%%%%%%%%%%%%%%%%%%%%%%%%%%%%%%%%%%%%
%%%%%%%%%%%%%%%%%%%%%%%%%%%%%%%%%%%%%%%%%%%%%%%%%%%%%%%%%%%%%%%%%%%%%%%%%%%%%%%%%%%%%%%%%%%%%%%%%%%%%%
%%%%%%%%%%%%%%%%%%%%%%%%%%%%%%%%%%%%%%%%%%%%%%%%%%%%%%%%%%%%%%%%%%%%%%%%%%%%%%%%%%%%%%%%%%%%%%%%%%%%%%

%
%

%

%\begin{proof} %\hfill\\
%	Cf. Section \ref{sec:identification:proofs}.
%\end{proof}

%
\subsection{Partial identification} \label{sec:identification:partial}

What can be said about the covariance matrix of the random coefficients if there are binary regressors?
Assume a single binary regressor $Z$, and additional regressors $\bW \in \R^{p-2}$ (slightly modifying the notation in this section) for which the support contains a Cartesian product with at least three points in each coordinate. Our model is then written as  
\begin{equation*} %\label{eq:thercmodel}
	Y = B_0 + Z \, B_1 + \f{W^\top} \f{B_2} \,.
\end{equation*}
The set of covariance matrices of $\f{A} = (B_0, B_1, \f{B_2^\top})^\top \in \R^p$ consistent with the conditional second moments of $Y$ is
\begin{align}\label{eq:setofcov}
	\begin{split}
		\mathcal{S} \defeq & \big\{\Sigma \in \R^{p \times p} \, \big| \,  \Sigma \text{ positive semi-definite and} \\
		& \quad \quad (1,z,\f{w^\top}) \, \Sigma \, (1,z,\f{w^\top})^\top = \var(Y \, | \, Z=z,\bW=\bw)\ \forall \ (z,\bw) \in \text{supp}(Z,\bW) \big\} \,.
	\end{split}
\end{align}
Suppose that the support of $(Z,\f{W^\top})^\top \in \R^{p-1}$ has a product structure. From Theorem \ref{the:identcartprodt}, using $Z=0$ and $Z=1$ we identify the covariance matrices
\begin{equation*}%\label{eq:identified}
	\cov\big((B_0, \f{B_2^\top})^\top \big)\quad \text{and} \quad \cov \big((B_0+B_1, \f{B_2^\top})^\top \big) \, ,
\end{equation*}
or equivalently 
\begin{equation}\label{eq:identified}
	\cov\big((B_0, \f{B_2^\top})^\top \big)\,, \qquad \cov \big(B_1 ; \f{B_2} \big)\,,\qquad \var(B_0 + B_1) \,.
\end{equation}
Here for random vectors $\f{C}$ and $\f{D}$, $\cov(\f{C})$ is the covariance matrix of $\f{C}$, while $\cov(\f{C}; \f{D})$ contains the cross-covariances of $\f{C}$ and $\f{D}$. 
%
%
%\begin{proposition}\label{prop:partialident}
%	%
%	If the support of $(Z,\f{W^\top})^\top \in \R^{p-1}$ has a product structure, and only $Z$ is binary, then the set \eqref{eq:setofcov} consists of all $p \times p$-dimensional covariance matrices which have the required identified moments in \eqref{eq:identified}.
%	%
%\end{proposition}
%%
%\begin{proof} %\hfill\\
%	{\color{red} Proof is missing.}
%\end{proof}
%%
%
%From (\ref{eq:identified}), we identify the variances $\var(B_0)$ and $\var (B_0 + B_1)$, and obtain the bounds  from Proposition \ref{prop:twosupportpointsfirst}.
%%
%These are evidently no longer sharp in case of additional covariates, however, 
Sharp bounds for $\var(B_1)$ are given by
\begin{equation}\label{eq:boundscovariance}
	\inf_{\Sigma \in \mathcal{S}} \Sigma_{22} \leq \var (B_1)\, \leq \sup_{\Sigma \in \mathcal{S}} \Sigma_{22}\,,
\end{equation}
where the set $\mathcal{S}$ in \eqref{eq:setofcov} is characterized by the restrictions given by the identified parts \eqref{eq:identified} of the matrix $\Sigma^*$. These bounds can be obtained numerically by semi-definite programming. 
An interesting particular question is the potential randomness of $B_1$, which is addressed in the following proposition, which relies on the identified quantities in \eqref{eq:identified}.   
%
%In certain cases, we can even conclude non-randomness of $B_1$. The proof is deferred to the supplement, Section \ref{sec:proofs:identification:partial}.
%
\begin{proposition}\label{prop:identifiedzero}
	Suppose that the support of $(Z,\f{W^\top})^\top$ has a product structure, and that only $Z$ is binary. 
	\begin{enumerate}[label=\arabic*.]
		\item If $\var( B_0) \not = \var(B_0 + B_1) $, or if $\cov(B_1; \f{B_2} )$ is not the zero vector, then $\var (B_1) >0$. 
		\item Conversely, suppose that  $\var( B_0) = \var(B_0 + B_1)$ and that $\cov(B_1; \f{B_2}) = \f{0}_{p-2}$.
		\begin{enumerate}
			\item If $\cov\big((B_0, \f{B_2^\top})^\top \big)$ is degenerate, and its kernel contains a vector with non-zero first coordinate, then necessarily $\var(B_1)=0$.
			\item On the other hand, if $\cov\big((B_0, \f{B_2^\top})^\top \big)$ has full rank, then the upper bound in \eqref{eq:boundscovariance} for $\var(B_1)$ is strictly positive. 
		\end{enumerate}
	\end{enumerate}

\end{proposition}

\subsection{Identification of higher-order moments}

The $k^{\text{th}}$ -order mixed moments of the random vector $\f{A}$, $k \in \N$, are given by 
$$ m(k_1, \ldots, k_p) = \E\big[A_1^{k_1}\, \ldots \, A_p^{k_p} \big],\qquad k_j \in \N_0,\ k_1 + \ldots + k_p = k,$$ 
of which there are $\binom{p+k-1}{k}$ many. 
Information on the mixed moments in the linear random coefficient model $ Y = A_1 + A_2 \, W_1 + \ldots  + A_p\, W_{p-1}$ comes from the identified conditional $k^{\text{th}}$ moments  of $Y$ given $\f{W}$, 
\begin{equation}\label{eq:conditionalhighermoments}
\E[Y^k | \f{W} = \f{w}] = \E\big[\big((1,\f{w}^\top)\, \f{A} \big)^k\big].
\end{equation}
These can be represented as an inner product of $\binom{p+k-1}{k}$ - dimensional vectors, one consisting of the mixed moments $m(k_1, \ldots, k_p)$, the other with corresponding entry 
\begin{equation}\label{eq:entriesmatrix}
\binom{k}{k_1 \,  \ldots \, k_p}\, w_1^{k_2}\, \cdot \ldots \cdot w_{p-1}^{k_p},
\end{equation}
where 
$\f{w} = (w_1, \ldots, w_{p-1})$. Hence, we have analogously to the result in Theorem \ref{prop:necsuffcond} that if there are  $\binom{p+k-1}{k}$ support points $\f{w}_j$ of $\f{W}$ such that if we form the matrix with rows as in \eqref{eq:entriesmatrix} for the coordinates of the $\f{w}_j$, the resulting quadratic matrix has full rank, then the $k^{\text{th}}$ -order mixed moments of $\f{A}$ are identified. 

While we were not able to obtain a sufficient condition along the lines of Theorem \ref{the:identcartprodt}, we have the following result which guarantees identification. 
%requires that the support contains a Cartesian product with at least $k$ points in each coordinate, but

\begin{theorem}\label{th:identhigherorder}
	%
%	For $\lambda \in \R$, $k \in \N$, $i\in \{0,1\}$ let $S(\lambda,k,i)= \{\lambda\, i, \lambda\, (i+1)\ldots, \lambda\, (k+i) \}$. 
%
	If in model \eqref{random:coefficient:model:intercept}, the support of $\bW = (W_1, \dotsc, W_{p-1})^\top$ contains $p$ points $\f{w}_1, \ldots, \f{w}_p$ in general position, for which for each $j \in \{1, \ldots, k\}$ and $i_1, \ldots, i_j \in \{1, \ldots, p\}$, the vector $ (\f{w}_{i_1} + \ldots + \f{w}_{i_j})/j$ is also in the support of $\bW$.
	%a Cartesian product of the form 
%
%	$$S(\lambda_1,k,i_1)  \times \ldots \times S(\lambda_{p-1},k,i_{p-1})$$
%	
%	for non-zero numbers $\lambda_j$ and $i_j \in \{0,1\}$, 
	Then the mixed moments of $\f{A}$ up to order $k$ are identified. 
%	three points in each coordinate. Then there exist $p(p+1)/2$ support points such that the matrix $S$ in \eqref{eq:matrixident} has full rank $p(p+1)/2$ and consequently, the means and (co-)variances of the random coefficients $\bA$ are identified. Conversely, if there is a $W_j$ having only two support points, then in the full-rank covariance matrices identification fails.   
	%
\end{theorem}
%

%\section{Parts of proofs of Identification of the moments} \label{sec:identification}

%\input{identification}

\section{Sign-consistency of the adaptive LASSO estimator} \label{ch:fixedp}

%\subsection{Second central moments} \label{sec:second:moments}
%
%\subsubsection{Regression model and estimator} \label{sec:model:second:moments}

In this section we derive the asymptotic variable selection properties of the adaptive LASSO in the linear random coefficient regression model \eqref{random:coefficient:model:intercept}, where we focus on estimating and selecting the variances and covariances of the random coefficients. 
	First, in Section \ref{sec:fixedparasymp} we consider an asymptotic regime with a fixed number $p$ of regressors, before turning to the moderately high-dimensional setting in which $p \to \infty$ but at a slower rate than the sample size $n$. 

The adaptive LASSO and its variable selection properties, originally introduced in \citet{zou2006adaptive}, have already been investigated intensively in the literature. For example, \citet{zou2009adaptive} consider the adaptive LASSO and an adaptive version of the elestic net in moderately high dimensions, while \citet{huang2008adaptive} investigate the high-dimensional situation with strong assumptions on the first stage estimator, and  \citet{wagener2013adaptive}  extend their approach to a heteroscedastic framework. 
Here, our contributions mainly are to deal with the residuals when estimating centered second moments of the random coefficients, and to extend the analysis of \citet{zou2009adaptive} to our setting with random coefficients.   	
	
%\edithajothree[Relate contribution to high-dimensional variance component models.]

%Our main interest in this chapter is variable selection and estimation of the first and second moments of a fixed number $p$ of random coefficients in model \eqref{random:coefficient:model:intercept}.  Very common and well-studied estimators in these situations are the 

%LASSO and the adaptive LASSO, cf. \citet{Tibshirani1996}, \citet{Zhao2006} and \citet{Zou2006}. These estimators combine the quadratic loss with a (weighted) $\ell_1$-penalization. Note that the adaptive version requires an consistent initial estimator. 
%
%%For asymptotic results in the setting of a fixed number of coefficients usually the Least Squares Estimator is used for this purpose. Auxiliary characterizations of the aforementioned estimators are given in the supplement, Section \ref{sec:appendix:estimators}. \\
%
%{\color{red} Focus on the second moments. First moments similar and easier, summarize and give references}.

\bigskip

We observe independent random vectors $(Y_1,\f{W_1^\top})^\top,\dotsc,(Y_n,\f{W_n^\top})^\top$  distributed according to the random coefficient regression model \eqref{random:coefficient:model:intercept}, and write
\begin{align*}
	Y_i =  B_{i,0} + \f{W_i^\top} \f{B_{i}} = \f{X_i^\top} \f{A_i} \,, \quad i =1,\dotsc,n\,,   %\label{random:coefficient:model}
\end{align*}
where $\f{X_i} = ( 1, \f{W_i^\top} )^\top \in \R^p$ with $\f{W_i} \sim \f{W}$ and $\f{A_i} = (B_{i,0},\f{B_{i}^\top} )^\top \sim \f{A}$ are independent random vectors. Here $\f{X_i} = ( X_{i,1},\dotsc,X_{i,p} )^\top$ represents the observed covariates and $\f{A_i} = ( A_{i,1},\dotsc,A_{i,p})^\top$ the unobserved individual regression coefficients.\\

In the following we denote by
\begin{align*}
	%	S_{\mu} &\defeq \supp(\mu^*) = \Big\{ k \in \{1,\dotsc,p\} : \mu_k^* \neq 0 \Big\}\,,\\
	S_{\sigma} &\defeq \supp\big(\sigma^* \big) = \Big\{ k \in \big\{1,\dotsc,p(p+1)/2\big\} \, \big| \, \sigma^*_k \neq 0 \Big\}\,, \qquad s_{\sigma} \defeq | S_{\sigma} |\,,
\end{align*}
the support of %the mean vector $\mu^*$ and 
the half-vectorization $\sigma^*$ of the covariance matrix $\Sigma^*$. 
%In addition let
%\begin{align}
%	s_{\mu} \defeq |S_{\mu}| \,, \quad \quad 
%\end{align}
%be the cardinalities and
%\begin{align}
%	S_{\mu}^c \defeq \{1,\dotsc,p\} \setminus S_{\mu} \,, \quad \quad \quad S_{\sigma}^c \defeq \bigg\{1,\dotsc,\frac{p(p+1)}{2}\bigg\} \setminus S_{\sigma}
%\end{align}
$S_{\sigma}^c \defeq \big\{1,\dotsc,p(p+1)/2\big\} \setminus S_{\sigma}$ will denote 
the relative complement of this set. 

%\begin{itemize}
%	\item There can be deterministic coefficients in the linear random coefficient regression model \eqref{random:coefficient:model:intercept}. If $k \in \{1,\dotsc,p\}$ is an element of $S_{\sigma}^c$, then $\var(A_k)=0$ and hence $A_k$ is constant (almost surely). Note that this coefficient is then also uncorrelated with the other coefficients $A_l$, $l \in \{1,\dotsc,p\}, l \neq k$. In consequence the $k$-th column and row of $\Sigma^*$ and the corresponding entries of $\sigma^*$ are equal to zero as well, and hence elements of $S_{\sigma}^c$.
%	\item Some of the coefficients can be equal to zero. If $k \in \{1,\dotsc,p\}$ is an element of $S_{\mu}^c$ and $S_{\sigma}^c$, then $\E[A_k]=\var(A_k)=0$ and hence $A_k = 0$ (almost surely).
%	\item There can be uncorrelated coefficients. If $k \in \{p+1,\dotsc,p(p+1)/2\}$ is an element of $S_{\sigma}^c$, then $\cov(A_l,A_{l'}) = 0$ for some $l,l' \in \{1,\dotsc,p\}$.
%\end{itemize} 

%For the estimation of the variances and covariances of the random coefficients knowledge about their means is crucial. 
%
%Hence we have to proceed in a two-step procedure. At first we determine a consistent estimator $\widehat{\mu}_n$ of $\mu^*$ based on the observations $(Y_1,\f{X_1^\top})^\top,\dotsc,(Y_n,\f{X_n^\top})^\top$ and with the help of the linear regression model in \eqref{random:coefficient:model:first:moments}. 
%
For an estimator $\widehat{\mu}_n$ of $\mu^*$ we define the regression residuals 
%\begin{align*}
$\widetilde{Y}_i \defeq Y_i - \f{X_i^\top} \widehat{\mu}_n$, % = \f{X_i^\top} \big( \f{A_i} - \mu^* \big) + \f{X_i^\top} \big(\mu^* - \widehat{\mu}_n \big)\,, \quad i = 1,\dotsc,n \,,
%\end{align*}
%
and write the squared residuals as
%Note that these variables are also observable since they only depend on $(Y_1,\f{X_1^\top})^\top, \dotsc,$ $(Y_n,\f{X_n^\top})^\top$. Moreover, the linear regression model in \eqref{random:coefficient:model:first:moments} implies 
%\begin{align*}
%\widetilde{Y}_i = \f{X_i}^\top \mu^*  +  \f{X_i^\top} \big( \f{A_i} - \mu^* \big) - \f{X_i^\top} \widehat{\mu}_n = \f{X_i^\top} \big( \f{A_i} - \mu^* \big) + \f{X_i^\top} \big(\mu^* - \widehat{\mu}_n \big) \,.
%\end{align*}
\begin{align*}
Y_i^{\sigma} \defeq \widetilde{Y}_i^2 = \f{X_i^\top} \big( D_i - \Sigma^* + E_n + F_{n,i} \big) \f{X_i}\, ,
%\f{X_i^\top} \Big( \big( \f{A_i} - \mu^* \big) \big( \f{A_i} - \mu^* \big)^\top + \big(\mu^* - \widehat{\mu}_n \big) \big(\mu^* - \widehat{\mu}_n \big)^\top  \\
%& \quad \quad \quad \quad \quad \quad + \big( \f{A_i} - \mu^* \big) \big(\mu^* - \widehat{\mu}_n \big)^\top + \big(\mu^* - \widehat{\mu}_n\big) \big( \f{A_i} - \mu^* \big)^\top \Big) \f{X_i} \,,
\end{align*}
%These include among other terms the products of the centered coefficients and hence they are the response variables in the linear regression model of the second central moments. Expansion of the products and rearranging leads to
%\begin{align*}
%Y_i^{\sigma} &= \f{X_i^\top} \Big( \big( \f{A_i} - \mu^* \big) \big( \f{A_i} - \mu^* \big)^\top + \big(\mu^* - \widehat{\mu}_n \big) \big(\mu^* - \widehat{\mu}_n \big)^\top + 2 \,\big( \f{A_i} - \mu^* \big) \big(\mu^* - \widehat{\mu}_n \big)^\top \Big) \f{X_i} \\
%\end{align*}
%It is now obvious that the dependent variables are a quadratic form, which is determined by the symmetric matrix in brackets, of the explanatory variables. We define the symmetric matrices
%
where we set
\begin{align}
D_i &\defeq \big(\f{A_i} - \mu^* \big) \big(\f{A_i}- \mu^*\big)^\top, \qquad E_n \defeq \big( \mu^* - \widehat{\mu}_n \big) \big( \mu^* - \widehat{\mu}_n \big)^\top  , \label{def:Di}\\
% &&\in \R^{p \times p} \, , \label{def:En}\\
F_{n,i} &\defeq \big( \f{A_i} - \mu^* \big) \big( \mu^* - \widehat{\mu}_n \big)^\top + \big( \mu^* - \widehat{\mu}_n \big) \big( \f{A_i} - \mu^* \big)^\top. \label{def:Fni}
\end{align}
%The matrices $D_1,\dotsc,D_n$ contain the products of the centered individual coefficients. Hence $\E[D_i] = \Sigma^*$ holds since the coefficients $\f{A_i}$ are identically distributed with covariance matrix $\Sigma^*$. Furthermore, $E_n$ captures the (products of the) estimation error of the means $\mu^*$ and $F_{n,1},\dotsc,F_{n,n}$ contain the mixing products. Now we can express the dependent variables by
%\begin{align*}
%Y_i^{\sigma} = \f{X_i^\top} \Sigma^* \, \f{X_i} + \,, \quad i=1,\dotsc,n \,. 
%\end{align*}
%The first part of the sum on the right-hand side contains the second central moments in which we are interested.
%
Applying the half-vectorization $\ovec$ for symmetric matrices in \eqref{def:vec} and the corresponding vector transformation $\ov$ in \eqref{def:v} we obtain in vector-matrix form
\begin{align*}
	\Ysigma = \Xsigma \, \sigma^* + \epssigma = \XsigmaS \, \sigma_{S_\sigma}^* + \epssigma \,, %\label{random:coefficient:model:second:moments:matrix}
\end{align*}
%
%
%\begin{align}
%Y_i^{\sigma} = \ov\big(\f{X_i}\big)^\top \sigma^* +  \ov\big(\f{X_i}\big)^\top \ovec\big( D_i - \Sigma^* + E_n + F_i \big)\,, \quad i=1,\dotsc,n \,,\label{random:coefficient:model:second:moments}
%\end{align} 
%where $\sigma^* = \ovec(\Sigma^*)$. Note that the deterministic coefficient vector $\sigma^*$ is $s_\sigma$-sparse. The errors are heteroscedastic and, moreover, they aren't independent since they depend all on the estimate $\widehat{\mu}_n$. Let
where
\begin{align}
	\Ysigma %&\defeq \big( Y_1^{\sigma}, \dotsc, Y_n^{\sigma} \big)^\top  &&\in \R^n \, , \notag \\
	&\defeq \Big( \big( Y_1 - \f{X_1^\top} \widehat{\mu}_n \big)^2,\dotsc,\big( Y_n - \f{X_n^\top} \widehat{\mu}_n \big)^2 \Big)^\top ,\qquad 
	\Xsigma \defeq \Big[ \ov\big(\f{X_1}\big), \dotsc, \ov\big(\f{X_n}\big) \Big]^\top  , \notag \\
	\epssigma &\defeq \Big( \ov\big(\f{X_1}\big)^\top \ovec\big( D_1 - \Sigma^* + E_n + F_{n,1} \big),\dotsc, 
	\ov\big(\f{X_n}\big)^\top \ovec\big( D_n - \Sigma^* + E_n + F_{n,n} \big) \Big)^\top.  \label{def:epssigma}
\end{align}

Then the adaptive LASSO estimator with regularization parameter $\lambda_n^{\sigma}>0$ is given by
\begin{align}
	\widehat{\sigma}_n^{\,\AL} \in \rho_{\sigma,n,\lambda_n^{\sigma}}^{\,\AL} &\defeq \underset{\beta \in \R^{p(p+1)/2} }{\arg\min} ~ \Bigg( \frac{1}{n} \, \normzq{\Ysigma - \Xsigma \, \beta} +  2 \lambda_n^{\sigma}  \sum_{k=1}^{p(p+1)/2} \frac{|\beta_k|}{\big|\widehat{\sigma}_{n,k}^{\,\init}\big|} \Bigg) \,, \label{def:adaptive:LASSO:second:moments}
\end{align}
where $\widehat{\sigma}_n^{\,\init} \in \R^{p(p+1)/2}$ is an initial estimator of $\sigma^*$. 
%In addition the sequence $\big(\widehat{\sigma}_n^{\,\init}\big)_{n \in \N} \subset \R^p$ should be consistent with respect to $\sigma^*$. 
Note that if $\widehat{\sigma}_{n,k}^{\,\init} = 0$, we require $\beta_k=0$.

\subsection{Asymptotics for fixed dimension $p$}\label{sec:fixedparasymp}

The proofs of the results in this section are deferred to Section \ref{sec:fixedp:proofs} in the supplement. 

\begin{assumption}[Fixed $p$] \label{ass:moments:second:moments} %\hfill\\
	We assume that $\big(\f{X_i^\top}, \f{A_i^\top}\big)^\top$, $i=1, \ldots, n$, are identically distributed, and that 
	\begin{enumerate}[label=\normalfont{(A\arabic*)},leftmargin=9.9mm]
		\setcounter{enumi}{\value{temp1}}
		\item the random coefficients $\f{A}$ have finite forth moments, \label{moments:second:moments:B1}
		\item the covariates $\f{X}=(1,\f{W^\top})^\top$ (or rather $\f{W}$) have finite eighth moments, \label{moments:second:moments:B2}
		\item the symmetric matrix \begin{align*}
			\mathrm{C}^{\sigma} \defeq \E \Big[ \ov\big(\f{X}\big)\,\ov\big(\f{X}\big)^\top\Big] \,, %\quad \in \R^{\frac{p(p+1)}{2}  \times \frac{p(p+1)}{2}} 
		\end{align*}  
		which contains the fourth moments of the covariates, is positive definite. \label{moments:second:moments:B3}
		\setcounter{temp1}{\value{enumi}}
	\end{enumerate}
\end{assumption}

In the following proposition, we show that the critical third part of the assumptions follows from our identification results in Section \ref{sec:model:random:coefficients}.

\begin{proposition} \label{lemma:Csigma:positive:definite}%\hfill\\
	Under the assumption of Theorem \ref{the:identcartprodt}, that the support of the covariate vector $\f{W}$ contains a Cartesian product with three points in each coordinate, Assumption \ref{ass:moments:second:moments}, (A3), is satisfied, that is,  
	 $\mathrm{C}^{\sigma}$ is positive definite.
\end{proposition}
To formulate an asymptotic result on variable selection consistency and asymptotic normality in fixed dimensions, set
\begin{align}
	\mathrm{B}^{\sigma} \defeq \E \bigg[ \Big( \ov\big( \f{X} \big)^\top \Psi^* \, \ov\big( \f{X} \big) \Big) \, \ov\big( \f{X} \big) \, \ov\big( \f{X} \big)^\top \bigg] \,, \label{def:B:second:moments}% \quad \in \R^{\frac{p(p+1)}{2} \times \frac{p(p+1)}{2}}\,, 
\end{align}
	where
\begin{align*}
	\Psi^* &\defeq \Big[ \ovec\big(\mathcal{M}^{11}\big),\dotsc, \ovec\big(\mathcal{M}^{pp}\big),\ovec\big(\mathcal{M}^{12}\big),\dotsc,\ovec\big(\mathcal{M}^{1p}\big), \\ %\label{def:psi} \\
	&\quad \quad \quad \quad  \quad \quad  \quad \quad \ovec\big(\mathcal{M}^{23}\big), \dotsc,\ovec\big(\mathcal{M}^{2p} \big),\dotsc,\ovec\big(\mathcal{M}^{(p-1)p}\big) \Big]^\top \notag %\quad \in \R^{\frac{p(p+1)}{2} \times \frac{p(p+1)}{2}} 
\end{align*} 
with $\mathcal{M}^{kl} \in \R^{p \times p}$ and 
\begin{align}
	\big( \mathcal{M}^{kl} \big)_{uv} \defeq \cov\Big( \big( A_{k} -\mu_k^* \big) \big( A_{l} -\mu_l^* \big), \big( A_{u} -\mu_u^* \big) \big( A_{v} -\mu_v^* \big) \Big) \,. \label{def:Mkl}
\end{align}

\begin{theorem}[Variable selection and asymptotic normality for fixed $p$] \label{theorem:adaptive:LASSO:second:moments} %\hfill\\
Suppose that the estimator $\widehat{\mu}_{n} $ of $\mu^*$ used in the residuals $\widetilde{Y}_i$ is $\sqrt{n}$-consistent, that is $\sqrt{n} \, \big( \widehat{\mu}_{n} - \mu^* \big) = \Op{1}$. 
%
%	Consider the linear regression model \eqref{random:coefficient:model:second:moments} with $\widehat{\mu}_n = \widehat{\mu}_n^{\, \AL}$, where $\widehat{\mu}_n^{\, \AL}$ is a solution to \eqref{def:adaptive:LASSO:first:moments} with initial estimator $\widehat{\mu}_n^{\,\init}$ that satisfies  and regularization parameter $\lambda_n^{\mu}$ that satisfies $\lambda_n^{\mu} \to 0$, $\sqrt{n} \, \lambda_n^{\mu} \to 0$ and $n \, \lambda_n^{\mu} \to \infty$.
	 Further, let Assumption \ref{ass:moments:second:moments} be satisfied, and assume that for the initial estimator $\widehat{\sigma}_{n}^{\,\init}$ in the adaptive LASSO $\widehat{\sigma}_n^{\,\AL}$ in \eqref{def:adaptive:LASSO:second:moments} we also have that $\sqrt{n} \, \big( \widehat{\sigma}_{n}^{\,\init} - \sigma^* \big) = \Op{1}$. If the regularization parameter is chosen as $\lambda_n^{\sigma} \to 0$, $\sqrt{n} \, \lambda_n^{\sigma} \to 0$ and $n \, \lambda_n^{\sigma} \to \infty$, then it follows that $\widehat{\sigma}_n^{\,\AL}$ is sign-consistent, 
	\begin{align}
		\mathbb{P}\Big( \sign\big(\widehat{\sigma}_n^{\,\AL}\big) = \sign\big(\sigma^*\big)\Big) \to 1 \, , \label{sign:consistency:adaptive:LASSO:second:moments}
	\end{align}
	and satisfies 
	\begin{align}
		\sqrt{n} \, \big( \widehat{\sigma}_{n,S_{\sigma}}^{\,\AL} - \sigma_{S_\sigma}^* \big) ~ \stackrel{d} \longrightarrow ~ \mathcal{N}_{s_{\sigma}} \Big( \f{0}_{s_{\sigma}} ,\big(\mathrm{C}_{S_\sigma S_\sigma}^{\sigma}\big)^{-1} \, \mathrm{B}_{S_{\sigma} S_{\sigma}}^{\sigma} \, \big(\mathrm{C}_{S_\sigma S_\sigma}^{\sigma}\big)^{-1} \Big) \, . \label{asymptotic:normality:adaptive:LASSO:second:moments}
	\end{align} 
\end{theorem}            
We defer the proof of the theorem to the supplementary appendix, Section \ref{sec:fixedp:proofs}.% \todo[Noch eintragen].
%\begin{proof} %\hfill\\
%	Cf. .
%\end{proof}      

\begin{remark}[Guaranteeing a positive semi-definite matrix] %\hfill\\
Consider the positive semi-definite cone
\begin{align*}
	\mathbb{S}_p^+ \defeq \big\{ M \in \R^{p \times p} \mid M \text{ is symmetric and positive semi-definite} \big\} \quad \subset \R^{p \times p}\,,
\end{align*} 
and its image under the vectorization operator
	\begin{align*}
	\mathbb{V}_p^+ \defeq \big\{ \ovec(M) \mid M \in \mathbb{S}_p^+ \big\} \quad \subset \R^{\frac{p(p+1)}{2}}.
\end{align*}
It would be of interest to directly restrict the estimate of $\sigma^*$ to $\mathbb{V}_p^+$, resulting in 
\begin{align}
	\widehat{\sigma}_{n,\text{pos}}^{\,\AL} \in  \underset{\beta \in \mathbb{V}_p^+ }{\arg\min} ~ \Bigg( \frac{1}{n} \, \normzq{\Ysigma - \Xsigma \, \beta} +  2 \lambda_n^{\sigma}  \sum_{k=1}^{p(p+1)/2} \frac{|\beta_k|}{\big|\widehat{\sigma}_{n,k}^{\,\init}\big|} \Bigg) \,, \label{def:adaptive:LASSO:second:moments:posdef}
\end{align}
an actual covariance matrix. 
Computationally this estimate is feasible in principle by using methods from semidefinite programming as discussed e.g.~in \citet{vandenberghe1996semidefinite}, or by reparametrizing positive semidefinite matrices in terms of Cholesky factors and maximizing over these Cholesky factors.
However, technically it is hard to extend the primal-dual witness approach underlying the proof of Theorem \ref{theorem:adaptive:LASSO:second:moments} to this setting. Indeed, the primal-dual witness approach amounts to showing that a vector with the correct sparsity pattern asymptotically satisfies the necessary and sufficient KKT - conditions for a minimizer of \eqref{def:adaptive:LASSO:second:moments}. However, these KKT conditions become intractable  for the semindefinite problem in \eqref{def:adaptive:LASSO:second:moments:posdef}.
\end{remark}

Fortunately, we have the following result, in which some coefficients are non-random, while those which actually are random have a non-singular covariance matrix.
\begin{corollary} \label{cor:blockmatrix}
	Under the conditions of Theorem \ref{theorem:adaptive:LASSO:second:moments}, suppose that the covariance matrix of the random coefficients in \eqref{def:cov:true} has the form 
		\begin{align*}
		\Sigma^* = \begin{bmatrix}   
			\Sigma_1^* & \f{0}_{d \times (p-d)}  \\
			\f{0}_{(p-d) \times d} & \f{0}_{(p-d) \times (p-d)} 
		\end{bmatrix}
	\end{align*}
for a positive definite $d \times d$ - matrix $\Sigma_1^*$. Then $\mathbb{P}(\widehat{\sigma}_{n,\text{pos}}^{\,\AL} = \widehat{\sigma}_{n}^{\,\AL}) \to 1$, $n \to \infty$. 
\end{corollary}
This follows from Theorem \ref{theorem:adaptive:LASSO:second:moments} since the blocks of zeros in $\Sigma^*$ are estimated as zero with probability tending to one, and the estimate for $\Sigma_1^*$ will be positive definite asymptotically with full probability, since the positive definite matrices are open in $\R^{d \times d}$. Hence the unconstrained estimator $\widehat{\sigma}_{n}^{\,\AL}$ will correspond with probability tending to $1$ to a positive semi-definite matrix, which proves the corollary. Note that the corresponding statement would not be true for the ordinary least squares estimator.

\subsection{Diverging number $p$ of parameters}\label{sec:divparasymp}

Again we shall focus on the covariance matrix, for a discussion of estimating the means see the appendix, Section \ref{sec:appendix:meandiv}. 
Recall $\mathrm{C}^{\sigma}$ and $\mathrm{B}^{\sigma}$ which are given in \ref{moments:second:moments:B3} and \eqref{def:B:second:moments}.

\begin{assumption}[Growing $p$] \label{ass:moments:second:moments:growing} %\hfill\\
	We assume that $\big(\f{X_i^\top}, \f{A_i^\top}\big)^\top$, $i=1, \ldots, n$, are identically distributed, and that 
	\begin{enumerate}[label=\normalfont{(A\arabic*)},leftmargin=9.9mm]
		\setcounter{enumi}{\value{temp1}}
		\item the random coefficients $\f{A}$ have finite fourth moments, \label{moments:coefficients:growing}
		\item the vector transformation $\ov(\f{X})$ of the covariates $\f{X}$ is sub-Gaussian after centering,\label{ass:regressors:subgaussian}
		\item $\Csigmal \leq \lambda_{\min} \big( \mathrm{C}^{\sigma} \big) \leq \lambda_{\max} \big( \mathrm{C}^{\sigma} \big) \leq \Csigmau
		$ for some positive constants $ 0 < \Csigmal \leq \Csigmau < \infty$, where $\lambda_{\min}(A)$ and $\lambda_{\max}(A)$ denote the minimal and maximal eigenvalues of a symmetric matrix $A$, \label{ass:second:moments:C}
		\item  $\lambda_{\max} \big(\mathrm{B}^{\sigma} \big) \leq \Bsigmau$ for some positive constant $\Bsigmau > 0$, \label{ass:second:moments:B}
		\item $\lim_{n\to \infty} p^4 / n = 0$.  \label{ass:second:moments:limit:pn}
		\setcounter{temp1}{\value{enumi}}
	\end{enumerate}
\end{assumption}

The proof of the following result is provided in Section \ref{sec:divoar:proofs}.

\begin{theorem}[Variable selection for diverging $p$] \label{theorem:adaptive:LASSO:second:moments:growing} %\hfill\\
	Suppose that the estimator $\widehat{\mu}_{n} $ of $\mu^*$ used in the residuals $\widetilde{Y}_i$ is $\sqrt{n/p}$-consistent, that is $\sqrt{n/p} \, \normz{ \widehat{\mu}_{n} - \mu^* } = \Op{1}$. 
	Further, let Assumption \ref{ass:moments:second:moments:growing} be satisfied, and assume that for the initial estimator $\widehat{\sigma}_{n}^{\,\init}$ in the adaptive LASSO $\widehat{\sigma}_n^{\,\AL}$ in \eqref{def:adaptive:LASSO:second:moments} we have also $\sqrt{n}/p  \, \normz{ \widehat{\sigma}_{n}^{\,\init} - \sigma^* } = \Op{1}$. Moreover, if the regularization parameter is chosen as $\lambda_n^{\sigma} \to 0$, 
\begin{align*}
	\sqrt{s_\sigma \, n} \, \lambda_n^{\sigma} \, / (\sigma_{\min}^* \, p) \to 0 \,, \quad p/(\sigma_{\min}^* \, \sqrt{n}) \to 0\,,\quad n\,\lambda_n^\sigma/p^2 \to \infty
\end{align*}  
	with $\sigma_{\min}^* \defeq \min_{k \in S_{\sigma}} |\sigma_k^*|$, then it follows that $\widehat{\sigma}_n^{\,\AL}$ is sign-consistent, 
	\begin{align}
		\mathbb{P}\Big( \sign\big(\widehat{\sigma}_n^{\,\AL}\big) = \sign\big(\sigma^*\big)\Big) \to 1 \, . \label{sign:consistency:adaptive:LASSO:second:moments:growing}
	\end{align}
\end{theorem}  

\begin{remark}
Additional technical issues in the proof of Theorem \ref{theorem:adaptive:LASSO:second:moments:growing}, as compared to the analysis in \citet{Zou2006}, are to deal with the residuals when estimating centered second moments of the random coefficients as well as with the heteroscedasticity of the model.] Let us also point out that under the assumptions of the theorem, the least squares estimator satisfies the requirements made on the initial estimator, 

%Philipp: Maybe add, that the least squares estimator satisfies the conditions on the intial estimator in Theorem \ref{theorem:adaptive:LASSO:second:moments:growing}, explanation can be found in response.]
\end{remark}

\begin{remark}
	For fixed $p$ (and $S_\sigma$) we obtain the same conditions for the choice of the regularization parameter as in Theorem \ref{theorem:adaptive:LASSO:second:moments}. Moreover, if only $S_\sigma$ is fixed, but the number of coefficients grows, the first condition on the regularization parameter in Theorem \ref{theorem:adaptive:LASSO:second:moments:growing} simplifies to $\sqrt{n} \, \lambda_n^{\sigma} \, /p \to 0$ and the second one is satisfied by  \ref{ass:second:moments:limit:pn}.
	% In addition, the corresponding proof reveals that we have $\sqrt{n}/p \, \big\| \widehat{\sigma}_{n}^{\,\AL} - \sigma^*\big\|_2 = \Op{1}$ and that the adaptive LASSO enjoys on the support $S_\sigma$ the same asymptotic as the least squares estimator $\widehat{\sigma}_{n}^{\,\LS}$, which evidently satisfies $\sqrt{n}/p  \, \normz{ \widehat{\sigma}_{n}^{\,\LS} - \sigma^* } = \Op{1}$ as well.
\end{remark}

\begin{remark}
Assumption \ref{ass:regressors:subgaussian} is satisfied for bounded covariates which we mainly focus on in this paper. If we merely assume a sub-Gaussian distribution for the regressor vector $\f{X}$ instead of its vector transformation $\ov(\f{X})$, we would require a result for the rate of concentration of the sample fourth moment matrix of sub-Gaussian random vectors in the spectral norm.
% For that purpose a tail inequality for squares and products of sub-Exponential random variables would be needed. {\color{red} See maybe \url{https://arxiv.org/pdf/1804.02605.pdf}, nochmal schauen, was es da dazu gibt}. The rate is slower than $\operatorname{\mathcal{O}_P} \big( \sqrt{p/n} \big)$ (note that here the dimension $p$ is actually $p^2$) and hence the Assumption \ref{ass:second:moments:limit:pn} has to be adjusted appropriate.
\end{remark}

\begin{remark} \label{remark:assumption:p2n} Assumption \ref{ass:second:moments:limit:pn} can be relaxed to $\lim_{n\to \infty} p^2 / n = 0$, which is the minimal condition so that the assumptions of Theorem \ref{theorem:adaptive:LASSO:second:moments:growing} can be satisfied, if the centered coefficients $\f{A}-\mu^*$ are sub-Gaussian as well and $\lim_{n\to \infty} n \exp(-C_p \,p) = 0$ holds for some positive constant $C_p>0$.	 See Remark \ref{rem:technicalcondagain} after the proof of Lemma \ref{lemma:gradient:bounded:second:moments:3:growing}. 
\end{remark}

\begin{remark}[Elastic net] \label{remark:elastic:net} Our results in Theorem \ref{theorem:adaptive:LASSO:second:moments:growing} should extend to the adaptive elastic net estimator, see \citet{zou2009adaptive} for an analysis of the adaptive elastic net in moderately high dimensions. The asymptotic properties should be similar to those of the adaptive LASSO, but its numerical performance may be better since the covariates in the design matrix $\Xsigma$ may be highly correlated.
\end{remark}

\section{Simulations} \label{simulations}

In this section we investigate numerically the performance of the adaptive LASSO with respect to variable selection of the variances and covariances of the random coefficients in two settings. Moreover, we consider various combinations for the sample size $n$ and the number $p$ of coefficients to study the performance for growing $p$.   

\medskip

We consider the linear random coefficient regression model \eqref{random:coefficient:model:intercept} where the first four coefficients $\big( B_0,B_1,B_2,B_3 \big)^\top \sim \mathcal{N}_4 \big(\mu_1^*,\Sigma_1^*\big)$ are normally distributed with mean vector $\mu_1^* = \big( 40,15,0,-10 \big)^\top$ and covariance matrix
\begin{align*}
	\Sigma_1^* = \begin{bmatrix}   
		10 & 15.65 & -5.20 & 0 \\
		15.65 & 50 & 0 & 12.65 \\
		-5.20 & 0 & 30 & -12.25 \\
		0 & 12.65 & -12.25 & 20 \\    
	\end{bmatrix} \,.
\end{align*}  

The exact correlations of the coefficients are $\rho_{01} = \cor(B_0,B_1) = 0.7$, $\rho_{02} = -0.3$, $\rho_{13} = 0.4$, $\rho_{23} = -0.5$ and evidently $\rho_{03} = \rho_{12} = 0$. Furthermore, we set the fifth coefficient $B_4$ equal to $20$ and add deterministic zeros for the remaining $p-5$ coefficients in model \eqref{random:coefficient:model:intercept}. Hence we obtain in total the mean vector
\begin{align*}
\mu^* = \Big(\big(\mu_1^*\big)^\top, 20, \f{0}_{(p-5)}^\top \Big)^\top
\end{align*}
and the covariance matrix
\begin{align*}
	\Sigma^* = \begin{bmatrix}   
		\Sigma_1^* & \f{0}_{4 \times (p-4)}  \\
		\f{0}_{(p-4) \times 4} & \f{0}_{(p-4) \times (p-4)} 
	\end{bmatrix}
\end{align*} 
(which equals the setting in Corollary \ref{cor:blockmatrix}) for the random coefficient vector $\f{A}$. Obviously the number $s_\sigma$ of non-zero elements in the half-vectorization $\sigma^*$ of the covariance matrix $\Sigma^*$ is always equal to $8$ for each $p \geq 5$. Moreover, the covariates $W_1,\dotsc,W_{p-1}$ in model \eqref{random:coefficient:model:intercept} are assumed to be independent and identically uniform distributed on the interval $[-1,1]$ ($\mathcal{U}[-1,1]$) or on the set $\{-1,0,1\}$ ($\mathcal{U}\{-1,0,1\}$). 

\medskip

In our numerical study we simulate $n$ pairs $(Y_1,\f{W_1^\top})^\top,\dotsc,(Y_n,\f{W_n^\top})^\top$ of data according to one of the above specified models and use them for variable selection of the second central moments of the random coefficients. For that purpose we apply the adaptive LASSO $\widehat{\sigma}_n^{\,\AL}$, which is given in \eqref{def:adaptive:LASSO:second:moments}, with the ordinary LASSO estimator as well as the least squares estimator as initial estimators $\widehat{\sigma}_n^{\,\init}$.  %which leads to improved performance as compared to the least squares estimator as initial estimator.] 
%\begin{align*}
%\widehat{\sigma}_n^{\,\init} = \widehat{\sigma}_n^{\,\LS} \in \underset{\beta \in \R^{p(p+1)/2} }{\arg\min} ~ \frac{1}{n} \, \normzq{\Ysigma - \Xsigma \, \beta} 
%\end{align*} 
To determine the residuals of the first stage mean regression we use the ordinary least squares estimator $\widehat{\mu}_n^{\,\LS}$. The adaptive LASSO is computed in our simulation by using the function \texttt{glmnet} of the eponymous package. Note that the intercept of the regression model is not penalized by this function, which means that the variance of the random intercept $B_0$ is not penalized in our setting. This is plausible since the coefficient $B_0$ includes the deterministic intercept as well as a random error which is not affected by the covariates.

\medskip

In each of the following scenarios we perform a Monte Carlo simulation with $m=10.000$ iterations to illustrate the sign-consistency of the adaptive LASSO $\widehat{\sigma}_n^{\,\AL}$ for various sample sizes, numbers of coefficients and supports for the regressors. Its regularization parameter $\lambda$ is always chosen such that the sign-recovery rate is as high as possible. For this purpose we use $1000$ independent repetitions in each scenario, run through a grid for $\lambda$ in each data set and determine the regularization parameters with a correct number of degrees of freedom.

 The average percentage of correct sign-recoveries are displayed in the subsequent Figure \ref{figure:1} for $n=5000$ and Figure \ref{figure:6} for $n=10000$ for both the least squares estimator as well as the ordinary LASSO as initial estimators, and for both choices of covariates. As the LASSO as initial estimator leads to much better selection performance, we concentrate on it in the following, where we consider in more detail the number of false positives and false negatives.

\begin{itemize}
\item[(a)] \textbf{Findings for sample size $\boldsymbol{n=5.000}$.} \\
Let us discuss the findings from Figures \ref{figure:1} - \ref{figure:5}. Evidently for both kinds of regressors the sign-recovery rate decreases if the number of coefficients increases. Note that the number of parameters which are estimated grows quadratically with the number $p$ of random coefficients since the half-vectorization $\sigma^*$ of the covariance matrix $\Sigma^*$ has dimension $p(p+1)/2$. In particular, if we consider $p=60$ coefficients in our model, we obtain $1830$ variances and covariances. Hence the results look quite satisfying, however, if the support of the regressors consists only of the three points $\{-1,0,1\}$, the sign-recovery rate is somewhat lower and decreases also slightly faster, as seen in Figure \ref{figure:1}. Second, there are rarely false positives, so that discoveries actually correspond to signals. The error in the sign recovery mainly stems from false negatives, of which there are rarely more than one, as seen in Figures \ref{figure:3} and \ref{figure:5}.

\begin{figure}[!h]       
	\centering	
	\includegraphics[width=0.7\linewidth]{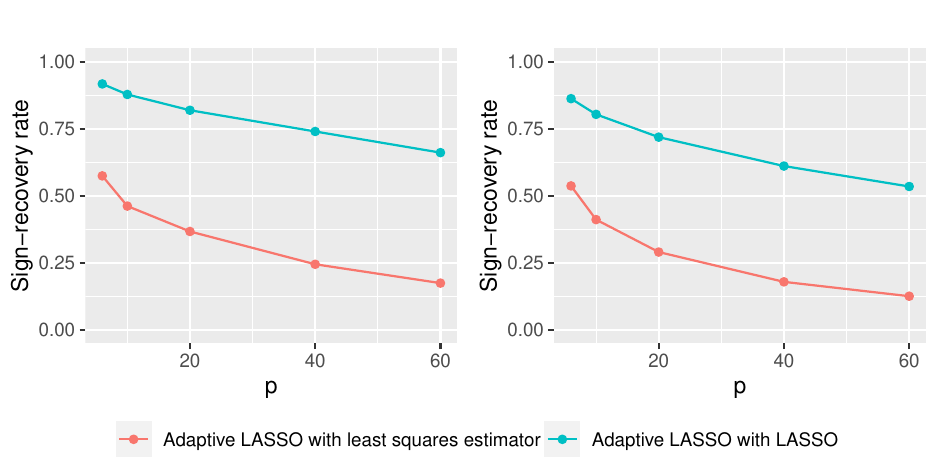}
	\caption{left chart shows the sign-recovery rate for $\mathcal{U}[-1,1]$ distributed regressors, right one for $\mathcal{U}\{-1,0,1\}$ distributed regressors. The sample size is always $n=5.000$.}
	\label{figure:1}
\end{figure}

%\begin{figure}[!h]       
%	\centering	
%	\includegraphics[width=0.7\linewidth]{sim_unif[-1,1]_n=5000-eps-converted-to}
%	\caption{frequency of false positives and false negatives for adaptive LASSO with least squares estimator as inital estimator, $\mathcal{U}[-1,1]$ distributed regressors and sample size $n=5.000$.}
%	\label{figure:2}
%\end{figure}

\begin{figure}[!h]       
	\centering	
	\includegraphics[width=0.7\linewidth]{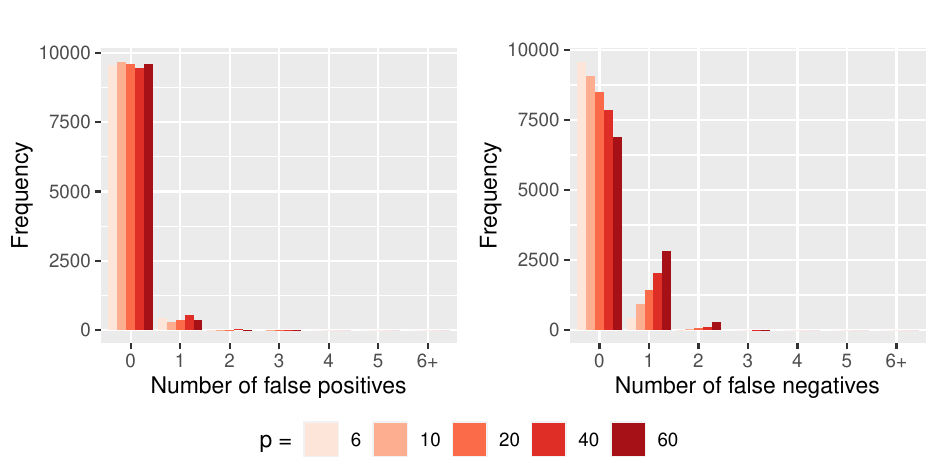}
	\caption{frequency of false positives and false negatives for adaptive LASSO with LASSO as inital estimator, $\mathcal{U}[-1,1]$ distributed regressors and sample size $n=5.000$.}
	\label{figure:3}
\end{figure}

%\begin{figure}[!h]       
%	\centering	
%	\includegraphics[width=0.7\linewidth]{sim_unif{-1,0,1}_n=5000-eps-converted-to}
%	\caption{frequency of false positives and false negatives for adaptive LASSO with least squares estimator as inital estimator, $\mathcal{U}\{-1,0,1\}$ distributed regressors and sample size $n=5.000$.}
%	\label{figure:4}  
%\end{figure}

\begin{figure}[!h]       
	\centering	
	\includegraphics[width=0.7\linewidth]{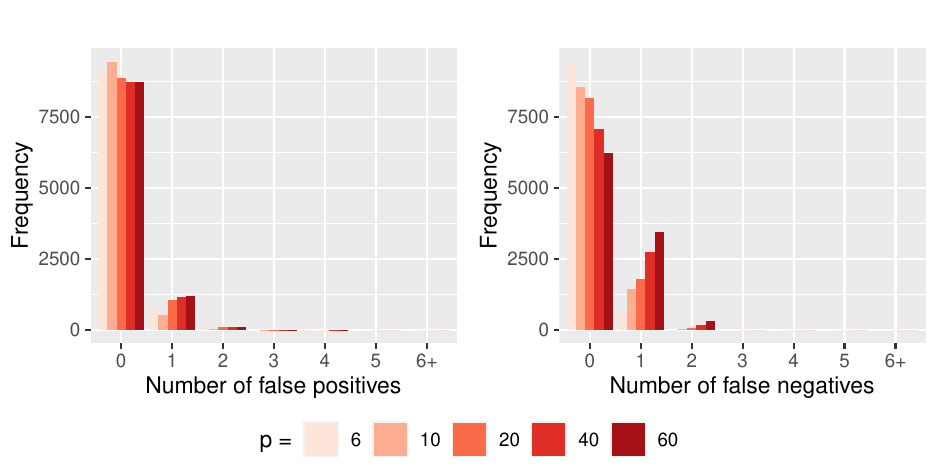}
	\caption{frequency of false positives and false negatives for adaptive LASSO with LASSO as inital estimator, $\mathcal{U}\{-1,0,1\}$ distributed regressors and sample size $n=5.000$.}
	\label{figure:5}  
\end{figure}

\item[(b)] \textbf{Sample size $\boldsymbol{n=10.000}$.} \\
In this setting the sign-recovery rate is in all scenarios much higher than in the first one with $n=5000$. In the bar charts of the false positives and false negatives there are no unexpected results to detect.

\begin{figure}[!h]       
	\centering	
	\includegraphics[width=0.7\linewidth]{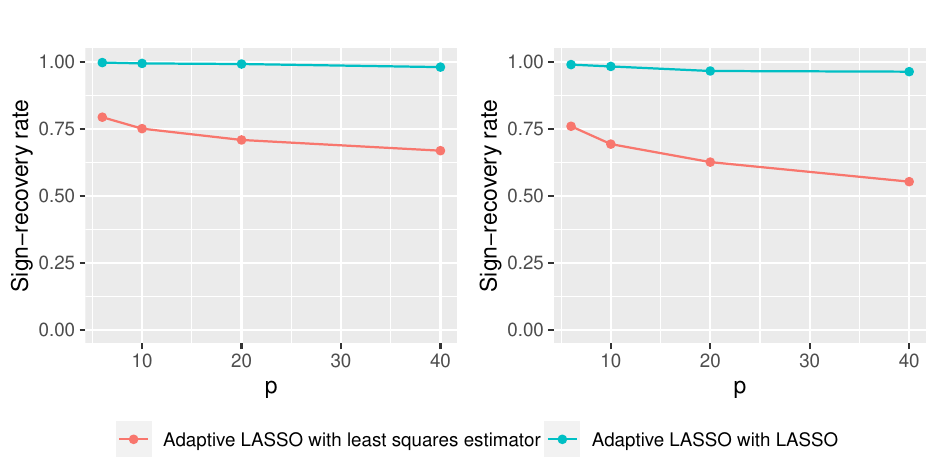}
	\caption{left chart shows the sign-recovery rate for $\mathcal{U}[-1,1]$ distributed regressors, right one for $\mathcal{U}\{-1,0,1\}$ distributed regressors. The sample size is always $n=10.000$.}
	 \label{figure:6} 
\end{figure}

%\begin{figure}[!h]       
%	\centering	
%	\includegraphics[width=0.7\linewidth]{sim_unif[-1,1]_n=10000-eps-converted-to}
%	\caption{frequency of false positives and false negatives for adaptive LASSO with least squares estimator as inital estimator, $\mathcal{U}[-1,1]$ distributed regressors and sample size $n=10.000$.}
%	\label{figure:7}  
%\end{figure}

\begin{figure}[!h]       
	\centering	
	\includegraphics[width=0.7\linewidth]{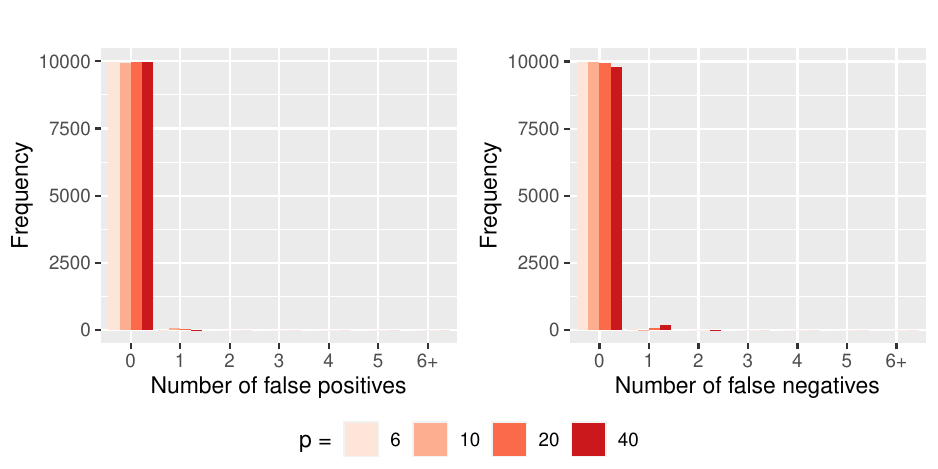}
	\caption{frequency of false positives and false negatives for adaptive LASSO with LASSO as inital estimator, $\mathcal{U}[-1,1]$ distributed regressors and sample size $n=10.000$.}
	\label{figure:8}  
\end{figure}

%\begin{figure}[!h]        
%	\centering	
%	\includegraphics[width=0.7\linewidth]{sim_unif{-1,0,1}_n=10000-eps-converted-to}
%	\caption{frequency of false positives and false negatives for adaptive LASSO with least squares estimator as inital estimator, $\mathcal{U}\{-1,0,1\}$ distributed regressors and sample size $n=10.000$.}
%	\label{figure:9}  
%\end{figure}

\begin{figure}[!h]        
	\centering	
	\includegraphics[width=0.7\linewidth]{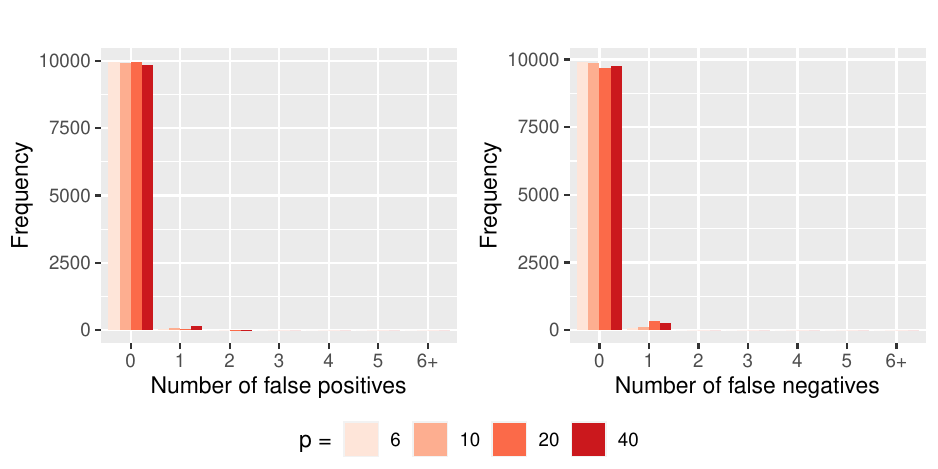}
	\caption{frequency of false positives and false negatives for adaptive LASSO with LASSO as inital estimator, $\mathcal{U}\{-1,0,1\}$ distributed regressors and sample size $n=10.000$.}
	\label{figure:10}  
\end{figure}

\end{itemize}

For comparison, we also present a figure for the sign-recovery rate for the mean in Figure  \ref{figure:11}. Here, even with the simple least squares estimator as initial estimator, the sign-recovery rate is already very high for sample size $n=5000$. 
%\textbf{Variable selection of the mean vector}

\begin{figure}[!h]       
	\centering	
	\includegraphics[width=0.7\linewidth]{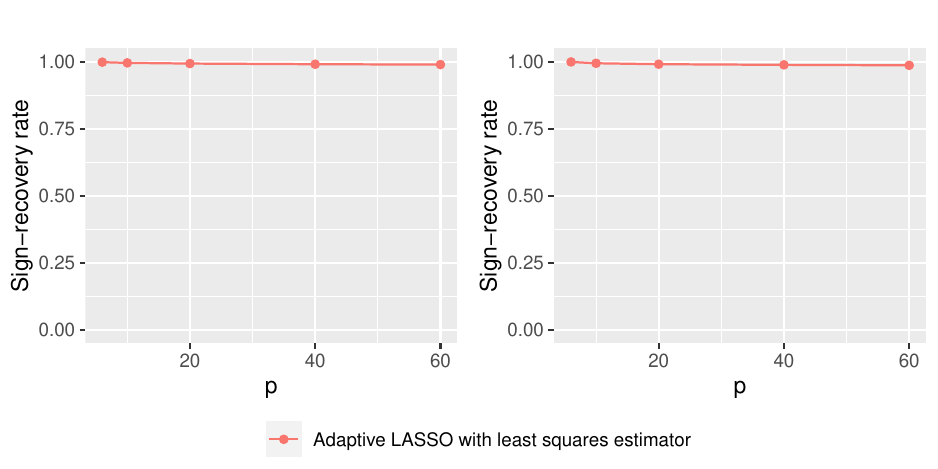}
	\caption{left chart shows the sign-recovery rate for $\mathcal{U}[-1,1]$ distributed regressors, right one for $\mathcal{U}\{-1,0,1\}$ distributed regressors. The sample size is always $n=5.000$.}
	\label{figure:11}
\end{figure}

\section{Proofs of the main results}\label{sec:mainproofs}

\subsection{Proofs for Section \ref{sec:model:random:coefficients}} \label{sec:identification:proofs}

\subsubsection*{Proofs of Propositions \ref{prop:twosupportpointsfirst} and \ref{prop:firstsimpleprop}} %\label{sec:proof:prop:twosupportpointsfirst}

\begin{proof}[Proof of Proposition \ref{prop:twosupportpointsfirst}]
	%
	%Letting $\widetilde{B}_1 = w B_1$, we may assume $w=1$. 
	%
	Set $u = \sqrt{\var(B_1)}$. From
	$ s_2^2 = s_1^2 + u^2 + 2\, \rho \, s_1 \, u$ 
	and $| \rho| \leq 1$ we obtain the inequalities
	\[ (u-s_1)^2 \leq s_2^2 \leq (u+s_1)^2.\]
	By equating $s_2^2 = (u+s_1)^2$ we obtain the solutions $\pm s_2 - s_1$ for $u$, which yields 
	$ u \geq s_2 - s_1$ if $s_2 >s_1.$ If $s_2 \leq s_1$ we obviously have only the bound $u \geq 0$.
	Equating $s_2^2 = (u-s_1)^2$ gives the solutions $\pm s_2 + s_1$ for $u$, which yields the bounds
	\[ u \in \big[| s_1 - s_2|,s_1 + s_2 \big]\]
	for the standard deviation $u=\sqrt{\var (B_1)}$. Solving the equation at the beginning for the correlation gives $ \rho = (s_2^2 - s_1^2 - u^2)/(2\, s_1 u)$, which ranges over the whole interval $[-1,1]$ if $s_2 > s_1$. If $s_1 \geq s_2$, the correlation must be negative, and 
	maximizing the above expression for $\rho$ over $u$ yields $u = \sqrt{s_1^2 - s_2^2}$, and finally the upper bound in (\ref{eq:boundcorrelation}). 
\end{proof}

%\subsubsection{Proofs for Section \ref{sec:identification:prelimary}} \label{sec:proofs:identification:prelimary}

\begin{proof}[Proof of Proposition \ref{prop:firstsimpleprop}] %\hfill\\
	It is enough to show that all mixed moments of order $n$ are identified from $n+1$ support points, the claim then follows by induction. By model \eqref{eq:randomcoefficients1} we obtain
	\begin{align*}
		\E\big[Y^n \, \big| \, W_1=w \big] = \E\big[ (B_0 + w \,B_1 )^n \big] = \sum_{k=0}^n \binom{n}{k}\, w^k \,\E\big[ B_0^{n-k}\, B_1^{k} \big].
	\end{align*}

	If $W$ has distinct support points $w_1, \dotsc, w_{n+1}$, we obtain a linear system for the moments $\E\big[ B_0^{n-k}\, B_1^{k} \big]$, $k=0, \dotsc, n$. Its design matrix satisfies
	\begin{align*}
		\det \Bigg(\binom{n}{k-1} w_j^{k-1} \Bigg)_{j,k \in \{1, \dotsc, n+1 \}} & = \prod_{l=0}^n \binom{n}{l}\ \det \Big(w_j^{k-1} \Big)_{j,k \in \{1, \dotsc, n+1\}}\\
		& = \prod_{l=0}^n \binom{n}{l}\  \prod_{1 \leq j < k \leq n+1} (w_k - w_j) \not=0 \,,
	\end{align*}
	so that the solution is unique. In the last equation we used the determinant of the Vandermonde matrix.
\end{proof}
%
%Consider the linear random coefficient regression model \eqref{random:coefficient:model:intercept}.
%%
%%\begin{equation}\label{eq:randomcoefficients}
%% Y = A_0 + \f{X^\top} \f{A}\,.
%%\end{equation}
%%
%Identification of the distribution or density of $\f{A}$ requires a large support of the covariates $\f{W}$, see \citet{Hoderlein2010} and \citet{Holzmann2019}. However, often covariates only have compact or even finite support. Hence our main interest is the identification of the covariance matrix $\Sigma^*$ (and of course of the means $\mu^*$) of the (potentially) random coefficients $\f{A}$ from covariates $\bW$ with possibly only finite support. 
%%
%
%

\subsubsection*{Proof of Theorem \ref{prop:necsuffcond}} %\label{sec:proof:prop:twosupportpointsfirst}

The proof needs some preparations. Recall that points $\f{w_1}, \dotsc, \f{w_{d}} \in \R^{d-1}$ are said to be in general position if $\sum_{k=1}^{d} \alpha_k \f{w_k} = \f{0}_{d-1}$ for $\alpha_k \in \R$, $\sum_{k=1}^{d} \alpha_k = 0$, implies that $\alpha_1 = \dotsc = \alpha_{d}=0$. The following result is well-known. 

\begin{lemma}\label{lem:generalposition}
	Points $\f{w_1}, \dotsc, \f{w_{d}} \in \R^{d-1}$ are in general position if and only if one of the following conditions holds.
	\begin{enumerate}[label=\arabic*.]
		\item $\f{w_2} - \f{w_1},\dotsc,  \f{w_{d}} - \f{w_1}$ are linearly independent.
		\item For each $ j \in \{1, \dotsc, d\}$ the point $\f{w_j}$ is not contained in $\big\{ \sum_{k=1,k \not= j}^{d} \alpha_k \f{w_k} \mid \sum_{k=1,k \not= j}^{d} \alpha_k = 1 \big\}$, the hyperplane generated by $\f{w_k}, k \not=j$. 
	\end{enumerate}
\end{lemma}
%
%\begin{proof} %\hfill\\
%Cf. {\color{red} Reference}.
%\end{proof}

%Before going to the general case, let us briefly consider diagonal covariance matrices. For a vector $\bw = (w_1, \dotsc, w_d)^\top \in \R^d$ we write $\bw^2 = (w_1^2, \dotsc, w_d^2)^\top \in \R^d$. 

\begin{lemma}\label{prop:onlyvariances}
	If the support of $\bW$ contains $p$ points $\f{w_1}, \dotsc, \f{w_p} \in \R^{p-1}$ in general position, then the means $\mu^*=\E [\bA ]$ are identified. 
%	If, moreover, the support also contains possibly distinct points $\f{\bar w_1}, \dotsc, \f{\bar w_p} \in \R^{p-1}$ for which $\f{\bar w_1}^2, \dotsc, \f{\bar w_p}^2 $ are in general position, then the variances are also identified. 
%	%
\end{lemma}
\begin{proof}[Proof of Lemma  \ref{prop:onlyvariances}] %\hfill\\
	The design matrix of the linear system $\E[ Y \, | \, \bW = \f{w_j}] = \E [B_0] + \f{w_j^\top} \E [\f{B}]$, $j=1,\dotsc,p$, has the same rank as the matrix  
	%
%	\[ \begin{bmatrix} 
%		1 & \f{w_1^\top} \\ \vdots & \vdots  \\ 1 & \f{w_p^\top} 
%	\end{bmatrix} \,.\] 
%	%
%	The matrix is of full rank, since its rank is the same as that of 
	%
	\[ \begin{bmatrix} 
		1 & \f{w_1^\top} \\ 0 & \f{w_2^\top} - \f{w_1^\top}\\ \vdots & \vdots  \\ 0 & \f{w_p^\top} - \f{w_1^\top}
	\end{bmatrix}\,,\] 
	which is invertible by Lemma \ref{lem:generalposition}. 
	%
%	Similarly, after identifying the means $ \E[ \bA]$, we have
%	%
%	\begin{align*}
%		\var \big(Y \, \big| \, \bW = \f{\bar \bw_j}\big) &= \E \Big[ \big(Y - \E[ Y \, | \, \bW = \f{\bar w_j}] \big)^2 \, \Big| \, \bW = \f{\bar \bw_j} \Big] = \E\Big[ \big( B_0 + \f{\bar\bw_j^\top} \f{B} - \E [B_0] - \f{\bar w_j^\top} \E [\f{B}] \big)^2 \Big] \\
%		&= \sigma_1^2 + (\f{\bar w_j}^2)^\top\, \big(\sigma_2^2, \dotsc, \sigma_p^2\big)^\top
%	\end{align*}
%	%
%	for $j=1,\dotsc,p$ since \eqref{eq:covarianceszero} holds. The linear system has design matrix 
%	%
%	\[ \begin{bmatrix} 
%		1 & (\f{\bar w_1}^2)^\top \\ \vdots & \vdots  \\ 1 & (\f{\bar w_p}^2)^\top 
%	\end{bmatrix} \,,\] 
%	%
%	which is also of full rank by Lemma \ref{lem:generalposition}.
%	%
\end{proof}
\begin{proof}[Proof of Theorem \ref{prop:necsuffcond}] %\hfill\\
	Suppose that $S$ is of full rank. Since $S$ contains the matrix
	\[ \begin{bmatrix} 1 & 2 \f{w_1^\top} \\ \vdots & \vdots  \\ 1 & 2 \f{w_{p(p+1)/2}^\top} \end{bmatrix} \quad \in \R^{\frac{p(p+1)}{2} \times p}\] 
	as a submatrix, in order for $S$ to have full rank, it is necessary that this submatrix has rank $p$. This implies that there are $p$ points among the support points $\f{w_1},\dotsc, \f{w_{p(p+1)/2}}$ in general position, thus identifying the means by Lemma \ref{prop:onlyvariances}. Then, the linear system which determines $\var(Y \, | \, \bW = \f{w_j})$ in terms of the entries of $\Sigma^*$ has full-rank design matrix $S$, see \eqref{linear:system:identifying:variances1}, thus identifying $\Sigma^*$ from the conditional variances. 

	Conversely, let $m=p(p+1)/2$. Suppose that the condition is not satisfied, then all support points $\bw$ of $\bW$ are such that the vectors $\ov \big( (1, \f{w^\top} )^\top \big)$ are contained in an $(m-1)$-dimensional linear subspace $V$ of $\R^m$. The $p\times p$-dimensional positive semi-definite matrices form a convex cone with interior consisting of positive definite matrices in the space of all $p \times p$-dimensional symmetric matrices. The image under the map $\ovec$ is thus a convex cone $\mathcal{C} \subset \R^m$ with non-empty interior in $\R^m$. 
	
	Let $\f{z}$ be a unit vector orthogonal to $V$, and let $Z$ be the $p \times p$-dimensional symmetric matrix for which $\ovec(Z)=\f{z}$. Since the positive definite matrices are open in the space of all $p \times p$-dimensional symmetric matrices, given a positive definite matrix $\Sigma^*$, for small $\epsilon>0$ the matrix $\Sigma_1 = \Sigma^* + \epsilon \,Z$ will still be positive definite, and hence a covariance matrix. Moreover, it is $\ovec(\Sigma_1) = \ovec(\Sigma^*) + \epsilon \, \ovec(Z) = \ovec(\Sigma^*) + \epsilon \,\f{z}$ and $(1, \f{w^\top})  \, Z \, (1, \f{w^\top})^\top  = \ov \big( (1, \f{w^\top} )^\top \big)^\top\, \f{z} = 0$ for $\bw$ in the support of $\bW$ by construction. Hence the conditional variances $(1, \f{w^\top}) \, \Sigma^* \, (1, \f{w^\top})^\top$ and $(1, \f{w^\top}) \, \Sigma_1 \, (1, \f{w^\top})^\top$ will be the same over the support of $\bW$. Thus, for normally distributed $\bA \sim \mathcal{N}_p(\f{0}_p,\Sigma^*)$ or $\bA \sim \mathcal{N}_p(\f{0}_p,\Sigma_1)$, the conditional normal distributions of $Y \, | \, \bW = \bw$ will coincide, showing nonidentifiability.
\end{proof}
\subsubsection*{Proof of Theorem \ref{the:identcartprodt}} %\label{sec:proof:prop:twosupportpointsfirst}
%
%\begin{proof} %\hfill\\
%Cf. Section \ref{sec:identification:proofs}.
%\end{proof}

%\begin{remark} %\hfill\\
%	Consider model \eqref{random:coefficient:model:intercept} and let $p=3$, that is $Y = B_0 + W_1 \, B_1 + W_2 \, B_2$. Consider the support points $(1,1)^\top$, $(2,4)^\top$, $(\sqrt{5},\sqrt{21})^\top$, which are in general position, while the squares $(1,1)^\top$, $(4,16)^\top$, $(5,21)^\top$ are not. However, if we have an additional support point $(w_1,w_2)^\top$ such that $(w_1^2, w_2^2)^\top$ is not on the line $(1,1)^\top + \lambda (1,5)^\top$, $\lambda \in \R$, then $(w_1,w_2)^\top$ will be in general position with at least one pair of the original three points, and the squares of these three points will also be in general position.
%\end{remark} 
%
%
For the proof of the theorem, we require the following lemma. 
\begin{lemma}\label{the:identificationsufficient}
	Suppose that the support of $\bW$ in \eqref{random:coefficient:model:intercept} contains points satisfying the following properties.
	\begin{enumerate}[label=\arabic*.]
		\item  The $p$ points $\f{w_1}, \dotsc,\f{w_p} \in \R^{p-1}$ are in general position.
		\item  For each $j \in \{1, \dotsc, p \}$ there exist points $\f{w_{j,1}}, \dotsc, \f{w_{j,p-1}} \in \R^{p-1}$, possibly equal to those in 1., such that 
		\begin{itemize}
			\item $\f{w_j},\f{w_{j,1}}, \dotsc, \f{w_{j,p-1}}$ are in general position,
			\item for each $j \in \{1, \dotsc, p\}$, $k \in \{1, \dotsc, p-1\}$ there is a $\f{z_{j,k}} \in \R^{p-1}$ for which $\f{w_j},\f{w_{j,k}},\f{z_{j,k}}$ are all distinct but generate only a one-dimensional affine space, i.e. are all contained in a line. 
		\end{itemize}
	\end{enumerate}
	Then the design matrix $S$ in \eqref{eq:matrixident} formed from all the points $\f{w_j},\f{w_{j,k}},\f{z_{j,k}}$ has full rank $p(p+1)/2$ and hence,  the mean vector $\mu^*$ and the covariance matrix $\Sigma^*$ of the random coefficients $\f{A}$ are identified. 
\end{lemma}
%
%\begin{proof} %\hfill\\
%Cf. Section \ref{sec:identification:proofs}.
%\end{proof}

The minimal number of support points required in this lemma is $p + p(p-1)/2 = p(p+1)/2$, which corresponds to the number of free parameters in $\Sigma^*$.
For the proof of Lemma \ref{the:identificationsufficient} we first need the following two preliminary lemmas. 

%\begin{proof} %\hfill\\
%	Cf. Section \ref{sec:identification:proofs}.
%\end{proof}
%\subsubsection{Proofs for Section \ref{sec:identification:fullpoint}} \label{sec:proofs:identification:fullpoint}
%%
\begin{lemma}\label{lem:first}
	Suppose that $\Sigma$ is a $p \times p$-dimensional symmetric matrix and $\f{v_1}, \dotsc, \f{v_p} \in \R^{p}$ is a known basis of $\R^{p}$. If $\bv \in \R^{p}$ and $\f{v^\top} \Sigma \, \f{v_j}$, $1 \leq j \leq p$, is identified, then $\f{v^\top} \Sigma \, \f{u}$ is identified for any vector $\f{u} \in \R^{p}$. In particular, $\Sigma$ is identified from the values $\f{v_j^\top} \Sigma \, \f{v_k}$, $1 \leq j \leq k \leq p$. 
\end{lemma}
\begin{proof}[Proof of Lemma \ref{lem:first}] %\hfill\\
	Given $\bu \in \R^{p}$ we may write $\bu = \sum_{j=1}^{p} \lambda_j \f{v_j}$ with $\lambda_1,\dotsc,\lambda_p \in \R$. Then 
	\[ \f{v^\top} \Sigma \, \f{u} = \sum_{j=1}^{p} \lambda_j \f{v^\top} \Sigma \,\f{v_j} \,,\]
	showing the first claim. For the second, let $e_k = (0, \dotsc, 0,1,0,\dotsc, 0)^\top$ denote the $k^{\text{th}}$ unit vector in $\R^{p}$. By assumption, one may write $e_k = \sum_{j=1}^{p} \lambda_{k,j} \f{v_j}$, where $\lambda_{k,j} \in \R$ and $1 \leq k \leq p$. Then 
	\[ \Sigma_{kl} = e_k^\top \Sigma \, e_l = \sum_{j_1, j_2 = 1}^{p} \lambda_{k,j_1} \lambda_{l,j_2}  \f{v_{j_1}^\top} \Sigma \, \f{v_{j_2}}\,.\]
	The result follows from the assumptions and the symmetry of $\Sigma$. 
\end{proof}
\begin{lemma}\label{lem:second}
	Let $\f{v_1}, \f{v_2}, \f{v_3} \in \R^{p}$ be such that each pair is linearly independent, but all three are linearly dependent, so that $\f{v_3} = \lambda_1 \f{v_1} + \lambda_2 \f{v_2}$, where $\lambda_1, \lambda_2 \not=0$. Then for a $p \times p$-dimensional symmetric matrix $\Sigma$ it holds that
	\[ \f{v_1^\top} \Sigma \, \f{v_2} = \frac{1}{2\, \lambda_1 \lambda_2} \Big( \f{v_3^\top} \Sigma \, \f{v_3} - \lambda_1^2 \,\f{v_1^\top} \Sigma \, \f{v_1} - \lambda_2^2 \, \f{v_2^\top} \Sigma \, \f{v_2} \Big) \,.\]
\end{lemma}   
\begin{proof}[Proof of Lemma \ref{lem:second}] %\hfill\\
	Plug in the expression for $\f{v_3}$ and compute the right side of the equation. 
\end{proof}

\begin{proof}[Proof of Lemma \ref{the:identificationsufficient}] %\hfill\\
	By Lemma \ref{prop:onlyvariances} and the first assumption the means $\mu^*$ are identified. 
	Hence we obtain the equations \eqref{linear:system:identifying:variances}
	or equivalently \eqref{linear:system:identifying:variances1} with $\bw$ ranging over the support points mentioned in the statement of the lemma. 
	To show that the design matrix $S$ in \eqref{eq:matrixident} has full rank $p (p+1)/2$, it suffices to show that from  these equations one can uniquely solve for $\sigma^*$. To this end, from the second assumption, for $j \in \{ 1, \dotsc, p \}$ and $k \in \{1, \dotsc, p-1\}$, letting $\f{v_1} = (1,\f{w_j^\top})^\top$, $\f{v_2} = (1,\f{w_{j,k}^\top})^\top$ and $\f{v_3} = (1,\f{z_{j,k}^\top})^\top$ in Lemma \ref{lem:second} we identify $(1,\f{w_j^\top}) \, \Sigma^* \, (1,\f{w_{j,k}^\top})^\top$. Since $(1,\f{w_j^\top}) \, \Sigma^* \, (1,\f{w_j^\top})^\top$ are also identified, from the first part in Lemma \ref{lem:first} we identify $(1,\f{w_j^\top}) \, \Sigma^* \, (1,\f{w_l^\top})^\top$, $j,l \in \{1, \dotsc, p \}$. Hence from the second part of that lemma and the first assumption $\Sigma^*$ itself is identified.
\end{proof}
\begin{proof}[Proof of Theorem \ref{the:identcartprodt}] %\hfill\\
	For the sufficiency, suppose that the support of $W_j$ contains $\{w_{j,k},\ k=1,2,3\}$, $j=1, \dotsc, p-1$. We apply Lemma \ref{the:identificationsufficient} with
	\begin{itemize}
		\item $\f{w_j} = (w_{1,1}, \dotsc, w_{j-1,1}, w_{j,2}, w_{j+1,1},\dotsc, w_{p-1,1})^\top$, $j = 1, \dotsc, p-1$, and \linebreak $\f{w_p} = (w_{1,1}, \dotsc, w_{p-1,1})^\top$,
		\item for $j \in \{1, \dotsc, p-1\}$, let $\f{w_{j,k}}$, $k \in \{1, \dotsc, p-1\}$, $k \not=j$, enumerate the points having $k^{\text{th}}$ coordinate $w_{k,2}$ and $j^{\text{th}}$ coordinate  $w_{j,2}$, otherwise coordinates $w_{i,1}$, the corresponding $\f{z_{j,k}}$ having $k^{\text{th}}$ coordinate $w_{k,3}$, $j^{\text{th}}$ coordinate  $w_{j,2}$, otherwise coordinates $w_{i,1}$. Furthermore, let $\f{w_{j,j}} = \f{w_{p}}$ and let $\f{z_{j,j}}$ have $j^{\text{th}}$ coordinate $w_{j,3}$, otherwise $w_{i,1}$,
		\item let $\f{w_{p,k}} = \f{w_k}$, $k \in \{1, \dotsc, p-1\}$, and $\f{z_{p,k}} =(w_{1,1}, \dotsc, w_{k-1,1}, w_{k,3}, w_{k+1,1},\dotsc, w_{p-1,1})^\top$. 
	\end{itemize}
	The requirements of the lemma are then easily checked by applying Lemma \ref{lem:generalposition}, 1. The necessity of at least three support points in each coordinate, if $\Sigma^*$ has full rank, is clear from Example \ref{ex:twosupportpoints}.
\end{proof}

\subsubsection{Proofs for Section \ref{sec:identification:partial}} \label{sec:proofs:identification:partial}

%\subsubsection*{Proof of Proposition \ref{prop:identifiedzero}} %\label{sec:proof:prop:twosupportpointsfirst}
%
\begin{proof}[Proof of Proposition \ref{prop:identifiedzero}] %\hfill\\
The claims in 1.~are clear. 
For 2., in both cases, from $\var(B_0) = \var(B_0 + B_1)$ we get that 
\[\var(B_1) = -2\,\cov(B_0,B_1)\,.\]
Since the covariance matrix $\cov\big((B_0,B_1,\f{B_2^\top})^\top\big)$ is positive semi-definite, setting $ s = \var(B_1)$, we hence require
\begin{equation}\label{eq:formpositive}
	s\,(z^2 - z v_1) + \f{v^\top}  \cov\big((B_0,\f{B_2^\top})^\top \big) \, \f{v} \geq 0
\end{equation}
for any $z \in \R$, $\bv = (v_1, \dotsc, v_{p-1})^\top \in \R^{p-1}$. 
\begin{enumerate}
\item[(a)] Choose $\bv\in \R^{p-1}$ in the kernel of $\cov\big((B_0,\f{B_2^\top})^\top \big)$ for which $v_1 >0$. If we assume $s>0$, then for $ 0 < z < v_1$ the form in (\ref{eq:formpositive}) would be negative. Hence $s=0$ in this case. 
\item[(b)] If $\cov\big((B_0,\f{B_2^\top})^\top \big)$ has full rank, then for the minimal eigenvalue $\lambda_{\min}>0$ of \linebreak $\cov\big((B_0,\f{B_2^\top})^\top \big)$ we have that
\[ \f{v^\top}  \cov\big((B_0,\f{B_2^\top})^\top \big) \, \f{v} \geq \lambda_{\min} \, \normzq{v} \,.\]
Therefore, the form in (\ref{eq:formpositive}) is positive definite for $0 \leq s \leq 4\, \lambda_{\min}$.
\end{enumerate}

\end{proof}
\begin{proof}[Proof of Theorem \ref{th:identhigherorder}]
By \eqref{eq:conditionalhighermoments}, the symmetric multilinear form
$$ u(\f{v}_1, \ldots, \f{v}_k) = \E\big[(\f{A}^\top \f{v}_1) \cdot \ldots \cdot (\f{A}^\top \f{v}_k)\big], \qquad \f{v}_j \in \R^p,\quad j=1, \ldots, k,$$
is identified over the diagonal 
$$\tilde u(\f{v}) = u(\f{v}, \ldots, \f{v})$$
 for $\f{v}^\top  = (1, \f{w}^\top)$ with $\f{w}$ in the support of $\f{W}$. 

We shall show that the symmetric multilinear form $u$ is identified. Then, inserting unit vectors $(0, \ldots, 0, 1, 0, \ldots, 0)$ yields the $k^{\text{th}}$ -order mixed moments.

By multilinearity, it sufficies to show that $u$ is identified over a basis of $\R^p$, that is, there exists a basis $\f{v_1}, \ldots, \f{v_p}$ of $\R^p$ such that $u(\f{v}_{i_1}, \ldots, \f{v}_{i_k})$ is identified for all choices $i_j \in \{1, \ldots, p\}$. To this end, we use the polarization formula for symmetric multilinear forms \citep[formula (7) ]{thomas2014polarization}, which we write as  
\begin{equation}\label{eq:polarizationmultilinear}
 u(\f{v}_1, \ldots, \f{v}_k) = \frac{1}{k!}\, \sum_{j=1}^k\, (-1)^{k-j}\, \sum_{\{i_1, \ldots i_j \} \subseteq \{1, \ldots, k\}} \,  j^k\, \tilde u \big((\f{v}_{i_1} +  \ldots +  \f{v}_{i_j})/j \big).
\end{equation}
Now for $\f{w}_1, \ldots, \f{w}_p$ as in the assumption of the theorem, the vectors $(1,\f{w}_i^\top)$, $i=1, \ldots, p$ are linearly independent by the proof of Lemma \ref{prop:onlyvariances}, and for $j \in \{1, \ldots, k\}$ and $i_1, \ldots, i_j \in \{1, \ldots, p\}$ we have  
$$ \big( (1,\f{w}_{i_1}^\top) + \ldots + (1,\f{w}_{i_j}^\top)\big)/j = \big(1,(\f{w}_{i_1} + \ldots + \f{w}_{i_j})^\top/j\big)$$
with $(\f{w}_{i_1} + \ldots + \f{w}_{i_j})^\top/j$ in the support of $\f{W}$. Hence the terms on the right of \eqref{eq:polarizationmultilinear} are identified for $k$ (not necessarily distinct) vectors $(1,\f{w}_i^\top)$, thus also the form $u$.  
\end{proof}

%\subsection{Main steps of the proofs} \label{sec:fixedp:proofs:main:steps}

\subsection{Proofs for Section \ref{sec:divparasymp}} \label{sec:divoar:proofs}

\subsubsection*{Proof of Theorem \ref{theorem:adaptive:LASSO:second:moments:growing}}

Consider the following decomposition of the error term \eqref{def:epssigma}:
	\begin{align}\label{eq:errordecomp}
	\epssigma = \delta_n  + \zeta_n + \xi_n 
\end{align}
with 
\begin{align}
	\delta_n &\defeq \Big( \ov\big(\f{X_1}\big)^\top \ovec\big( D_1 - \Sigma^* \big),\dotsc, \ov\big(\f{X_n}\big)^\top \ovec\big( D_n - \Sigma^* \big) \Big)^\top \, ,\label{def:deltan} \\
	%&= \Big( \f{X_1^\top} \big( D_1 - \Sigma^* \big) \f{X_1},\dotsc, \f{X_n^\top} \big( D_n - \Sigma^* \big) \f{X_n} \Big)^\top \\
	\zeta_n &\defeq \Big( \ov\big(\f{X_1}\big)^\top \ovec\big( E_n \big),\dotsc,\ov\big(\f{X_n}\big)^\top \ovec\big( E_n \big) \Big)^\top \, , \notag \\  
	%&= \Big( \f{X_1^\top} E_n \f{X_1},\dotsc, \f{X_n^\top} E_n \f{X_n} \Big)^\top \, , \qquad %\label{def:zetan} 
	\xi_n &\defeq \Big( \ov\big(\f{X_1}\big)^\top \ovec\big( F_{n,1} \big),\dotsc,\ov\big(\f{X_n}\big)^\top \ovec\big( F_{n,n} \big) \Big)^\top \,. \notag 
	%= \Big( \f{X_1^\top} F_{n,1} \f{X_1},\dotsc, \f{X_n^\top} F_{n,n} \f{X_n} \Big)^\top \, . \notag %\label{def:xin} 
\end{align}

The matrices $D_1,\dotsc,D_n,E_n,F_{n,1},\dotsc,F_{n,n}$ are defined in \eqref{def:Di} and \eqref{def:Fni}. 

For the proof of Theorem \ref{theorem:adaptive:LASSO:second:moments:growing} we need the following auxiliary lemmas.

\begin{lemma} \label{lemma:gradient:bounded:second:moments:1:growing} %\hfill\\
	%	Suppose that Assumption \ref{moments:second:moments:B1} holds. Then the random vectors $Z_n^{\sigma,1}$ in \eqref{def:parts:gradient:second:moments} satisfy 
	%
	Set $Z_n^{\sigma,1} = \frac{1}{n} \, \big(\Xsigma\big)^\top \delta_n$, then $\big\| Z_n^{\sigma,1} \big\|_2 = \Op{p/\sqrt{n}} $. 
\end{lemma}

\begin{proof}[Proof of Lemma \ref{lemma:gradient:bounded:second:moments:1:growing}]
	It is
	\begin{align*}
		\E \Big[ \normzq{Z_n^{\sigma,1}} \Big] &= \frac{1}{n^2} \, \E \Big[ \delta_n^\top \Xsigma \, (\Xsigma)^\top \delta_n \Big] = \frac{1}{n^2} \, \E \Big[ \tr \big( (\Xsigma)^\top \delta_n \, \delta_n^\top \Xsigma \big) \Big] \\
		&= \frac{1}{n^2} \, \E \Big[ \tr \big( (\Xsigma)^\top \E\big[ \delta_n \, \delta_n^\top \, \big| \,\Xsigma \big] \, \Xsigma \big) \Big] =   \frac{1}{n} \, \tr \Bigg( \E \bigg[ \frac{1}{n} \, (\Xsigma)^\top \Omega_n^{\sigma} \,\Xsigma \bigg] \Bigg) \,,
	\end{align*}
	where $\Omega_n^{\sigma} = \cov\big( \delta_n \, \big| \,\Xsigma \big)$ is given in Lemma \ref{lemma:gradient:bounded:second:moments:1}. It is obvious that
	\begin{align*}
		\E \bigg[ \frac{1}{n} \, (\Xsigma)^\top \Omega_n^{\sigma} \,\Xsigma \bigg] = \mathrm{B}^\sigma\,,
	\end{align*}
	and hence we obtain by Assumption \ref{ass:second:moments:B} the estimate
	\begin{align*}
		\E \Big[ \normzq{Z_n^{\sigma,1}} \Big] &= \frac{\tr \big( \mathrm{B}^\sigma \big)}{n} \leq \frac{\lambda_{\max} \big(\mathrm{B}^{\sigma} \big) \, p(p+1)}{2n} \leq \frac{\Bsigmau \, p(p+1)}{2n} \,. 
	\end{align*}
	Markov's inequality implies the assertion.
\end{proof}

\begin{lemma} \label{lemma:gradient:bounded:second:moments:2:growing} %\hfill\\
	%	Suppose that Assumption \ref{moments:second:moments:B1} holds. Then the random vectors $Z_n^{\sigma,1}$ in \eqref{def:parts:gradient:second:moments} satisfy 
	%
	Set $Z_n^{\sigma,2} = \frac{1}{n} \, \big(\Xsigma\big)^\top \zeta_n$, then $\big\| Z_n^{\sigma,2} \big\|_2 = \Op{p/n} $. 
	%In addition, it is clear that we obtain also $\big\| \frac{1}{n}\big(\XmuS\big)^\top \epsmu \big\|_2 = \Op{\sqrt{s_\mu/n}}$.
	%	\end{align*}
\end{lemma}

\begin{proof}[Proof of Lemma \ref{lemma:gradient:bounded:second:moments:2:growing}]
	It is 
	\begin{align*}
		\normz{Z_n^{\sigma,2}} &= \normz{\frac{1}{n}\, \sum_{i=1}^n \big( e_i^\top \zeta_n \big) \ov \big(\f{X_i}\big)} = \normz{\frac{1}{n}\, \sum_{i=1}^n \Big( \ov\big(\f{X_i}\big)^\top \ovec\big( E_n \big) \Big)  \ov \big(\f{X_i}\big)} \\
	%	&= \normz{\frac{1}{n}\, \sum_{i=1}^n \ov \big(\f{X_i}\big) \, \ov\big(\f{X_i}\big)^\top \ovec\big( E_n \big) } 
		& \leq \normzM{\frac{1}{n}\, \sum_{i=1}^n \ov \big(\f{X_i}\big) \, \ov\big(\f{X_i}\big)^\top} \, \normz{\ovec\big( E_n \big)} \,.
	\end{align*}
	We multiply each entry of $E_n$, which is not on the diagonal, with $\sqrt{2}$ and denote the resulting matrix by $\widetilde{E}_n$. Then it is clear that $\normz{\ovec\big( E_n \big)} \leq \big\| \ovec\big( \widetilde{E}_n \big) \big\|_2$ and $\big\| \ovec\big( \widetilde{E}_n \big) \big\|_2 = \normFM{E_n}$. Moreover, recall that $E_n = \big( \mu^* - \widehat{\mu}_n \big) \big( \mu^* - \widehat{\mu}_n \big)^\top$ is a rank-one matrix and hence $\normFM{E_n} \leq \normzq{\widehat{\mu}_n - \mu^*}$. Hence we obtain
	\begin{align*}
		\frac{n}{p} \, \normz{Z_n^{\sigma,2}} \leq \normzM{\frac{1}{n}\, \big(\Xsigma\big)^\top \Xsigma} \, \frac{n}{p} \, \normzq{\widehat{\mu}_n - \mu^*} = \Op{1} \, \Op{1} = \Op{1}
	\end{align*}
	since $\sqrt{n/p} \, \normz{ \widehat{\mu}_{n} - \mu^* } = \Op{1}$.
\end{proof}

\begin{lemma} \label{lemma:gradient:bounded:second:moments:3:growing} %\hfill\\
	%	Suppose that Assumption \ref{moments:second:moments:B1} holds. Then the random vectors $Z_n^{\sigma,1}$ in \eqref{def:parts:gradient:second:moments} satisfy 
	%
	Set $Z_n^{\sigma,3} = \frac{1}{n} \, \big(\Xsigma\big)^\top \xi_n $, then $\big\| Z_n^{\sigma,3}\big\|_2 = \Op{p^{3/2} / n^{5/8} + p^{2} / n^{3/4}} $. In particular, we obtain by Assumption \ref{ass:second:moments:limit:pn} the convergence $\sqrt{n}/p \, \big\| Z_n^{\sigma,3} \big\|_2 = \Op{p^{1/2} / n^{1/8} + p / n^{1/4}} = \op{1}$.
	%In addition, it is clear that we obtain also $\big\| \frac{1}{n}\big(\XmuS\big)^\top \epsmu \big\|_2 = \Op{\sqrt{s_\mu/n}}$.
	%	\end{align*}
\end{lemma}

\begin{proof}[Proof of Lemma \ref{lemma:gradient:bounded:second:moments:3:growing}]
	It is
	\begin{align}
		\normz{Z_n^{\sigma,3}} &= \normz{\frac{1}{n}\, \sum_{i=1}^n \big( e_i^\top \xi_n \big) \ov \big(\f{X_i}\big)} %= \normz{\frac{1}{n}\, \sum_{i=1}^n \Big( \f{X_i^\top} F_{n,i} \f{X_i} \Big)  \ov \big(\f{X_i}\big)} \notag \\
		= 2\, \normz{\frac{1}{n}\, \sum_{i=1}^n \Big( \f{X_i^\top} \big( \f{A_i} - \mu^* \big) \, \f{X_i^\top} \big( \mu^* - \widehat{\mu}_n \big) \Big)  \ov \big(\f{X_i}\big)} \notag \\
%		&= 2\, \normz{\frac{1}{n}\, \sum_{i=1}^n \Big( \f{X_i^\top} \big( \f{A_i} - \mu^* \big) \Big)  \ov \big(\f{X_i}\big) \, \f{X_i^\top} \big( \mu^* - \widehat{\mu}_n \big)} \notag \\
		&\leq 2\,\normzM{\frac{1}{n}\, \sum_{i=1}^n \Big( \f{X_i^\top} \big( \f{A_i} - \mu^* \big) \Big)  \ov \big(\f{X_i}\big) \, \f{X_i^\top}} \, \normz{\mu^* - \widehat{\mu}_n} \,, \label{proof:gradient:bounded:second:moments:3:growing:2}
	\end{align}
	and we have again $\sqrt{n/p} \, \normz{ \widehat{\mu}_{n} - \mu^* } = \Op{1}$. Moreover, let \begin{align*}
	\mathcal{T}_n (\tau_n) =  \bigcap_{i=1}^n \Big\{ \normz{\f{A_i} - \mu^*} \leq \tau_n \Big\}
	\end{align*}
	with $\tau_n > 0$, then we obtain 
	\allowdisplaybreaks
	{\small
	\begin{align}
		& \bigg\| \frac{1}{n}\, \sum_{i=1}^n \Big( \f{X_i^\top} \big( \f{A_i} - \mu^* \big) \Big)  \ov \big(\f{X_i}\big) \, \f{X_i^\top} \bigg\|_{\mathrm{M},2} \notag \\	
		&\leq \bigg\| \frac{1}{n}\, \sum_{i=1}^n \Big( \f{X_i^\top} \big( \f{A_i} - \mu^* \big) \Big)  \ov \big(\f{X_i}\big) \, \f{X_i^\top} \, \one_{  \mathcal{T}_n (\tau_n) } \bigg\|_{\mathrm{M},2} 
	+ \bigg\| \frac{1}{n}\, \sum_{i=1}^n \Big( \f{X_i^\top} \big( \f{A_i} - \mu^* \big) \Big)  \ov \big(\f{X_i}\big) \, \f{X_i^\top} \, \one_{  \mathcal{T}_n^c (\tau_n) } \bigg\|_{\mathrm{M},2} \notag \\
		&\leq \Bigg\| \frac{1}{n}\, \sum_{i=1}^n \Bigg( \Big( \f{X_i^\top} \big( \f{A_i} - \mu^* \big) \Big)  \ov \big(\f{X_i}\big) \, \f{X_i^\top} \, \one_{  \mathcal{T}_n (\tau_n) } - \E\bigg[ \Big( \f{X_i^\top} \big( \f{A_i} - \mu^* \big) \Big)  \ov \big(\f{X_i}\big) \, \f{X_i^\top} \, \one_{  \mathcal{T}_n (\tau_n) } \bigg] \Bigg) \Bigg\|_{\mathrm{M},2} \notag \\
		&  + \Bigg\| \frac{1}{n}\, \sum_{i=1}^n \E\bigg[ \Big( \f{X_i^\top} \big( \f{A_i} - \mu^* \big) \Big)  \ov \big(\f{X_i}\big) \, \f{X_i^\top} \, \one_{  \mathcal{T}_n (\tau_n) } \bigg] \Bigg\|_{\mathrm{M},2} + \bigg\| \frac{1}{n}\, \sum_{i=1}^n \Big( \f{X_i^\top} \big( \f{A_i} - \mu^* \big) \Big)  \ov \big(\f{X_i}\big) \, \f{X_i^\top} \, \one_{  \mathcal{T}_n^c (\tau_n) } \bigg\|_{\mathrm{M},2} \,. \label{proof:gradient:bounded:second:moments:3:growing:3}
	\end{align}
}
	For the first term of the sum we get
	\begin{align}
		\Bigg\| \frac{1}{n}\, &\sum_{i=1}^n \Bigg( \Big( \f{X_i^\top} \big( \f{A_i} - \mu^* \big) \Big)  \ov \big(\f{X_i}\big) \, \f{X_i^\top} \, \one_{  \mathcal{T}_n (\tau_n) } - \E\bigg[ \Big( \f{X_i^\top} \big( \f{A_i} - \mu^* \big) \Big)  \ov \big(\f{X_i}\big) \, \f{X_i^\top} \, \one_{  \mathcal{T}_n (\tau_n) } \bigg] \Bigg) \Bigg\|_{\mathrm{M},2} \notag \\
		 &= \sup_{\substack{u_1 \in \R^{p(p+1)/2},u_2 \in \R^{p},\\ \normz{u_1}, \normz{u_2} \leq 1}} \frac{1}{n} \, \sum_{i=1}^n \Bigg( u_1^\top \ov \big(\f{X_i}\big) \Big( \f{X_i^\top} \big( \f{A_i} - \mu^* \big) \, \f{X_i^\top} u_2  \Big) \, \one_{  \mathcal{T}_n (\tau_n) } \notag \\
		 & \qquad \qquad \qquad \qquad \qquad \qquad \qquad \qquad - \E\bigg[ u_1^\top \ov \big(\f{X_i}\big) \Big( \f{X_i^\top} \big( \f{A_i} - \mu^* \big) \, \f{X_i^\top} u_2  \Big) \, \one_{  \mathcal{T}_n (\tau_n) } \bigg]  \Bigg) \,. \label{proof:gradient:bounded:second:moments:3:growing:1}
	\end{align}
	For the second factors in brackets we obtain by the definition of the half-vectorization $\ovec$ in \eqref{def:vec} and the vector transformation $\ov$ in \eqref{def:v} the equation
	\begin{align*}
		\f{X_i^\top} \big( \f{A_i} - \mu^* \big) \, \f{X_i^\top} u_2  &= \f{X_i^\top} \big( \f{A_i} - \mu^* \big) \, u_2^\top \f{X_i} = \frac{1}{2} \, \f{X_i^\top} \Big( \big( \f{A_i} - \mu^* \big) \, u_2^\top +   u_2 \, \big( \f{A_i} - \mu^* \big)^\top \Big) \, \f{X_i} \\
		&=  \frac{1}{2} \, \ov\big(\f{X_i}\big)^\top \ovec \Big(  \big( \f{A_i} - \mu^* \big) \, u_2^\top +   u_2 \, \big( \f{A_i} - \mu^* \big)^\top \Big) 
		%&=  \ov\big(\f{X_i}\big)^\top  \Bigg( \frac{1}{2} \,\bigg( \ovec \Big( \big( \f{A_i} - \mu^* \big) \, u_2^\top \Big)  +   \ovec \Big(u_2 \, \big( \f{A_i} - \mu^* \big)^\top \Big) \bigg) \Bigg) \,.
	\end{align*}
	%The matrices in the last line are rank-one matrices and hence we can use the same arguments as in Lemma \ref{lemma:gradient:bounded:second:moments:2:growing} to bound the Euclidean norm, that is
	For the half-vectorization we can argue analogously as in Lemma \ref{lemma:gradient:bounded:second:moments:2:growing} and bound its Euclidean norm by
	\begin{align*}
		\frac{1}{2} \,\normz{\ovec \Big( \big( \f{A_i} - \mu^* \big) \, u_2^\top + u_2 \, \big( \f{A_i} - \mu^* \big)^\top \Big) } &\leq \frac{1}{2} \,\normFM{  \big( \f{A_i} - \mu^* \big) \, u_2^\top + u_2 \, \big( \f{A_i} - \mu^* \big)^\top } \\
		& \leq \normFM{  \big( \f{A_i} - \mu^* \big) \, u_2^\top }  \leq \normz{\f{A_i} - \mu^*} \, \normz{u_2} \,.
	\end{align*}
Suppose that $\ov(\f{X})$ is sub-Gaussian with variance proxy $\tauvX$, then conditionally on the coefficients $\f{A_1},\dotsc,\f{A_n}$, which are independent of the regressors $\f{X_1},\dotsc,\f{X_n}$, we obtain in \eqref{proof:gradient:bounded:second:moments:3:growing:1} for each $u_1,u_2$ a sum of centered and independent products consisting of two sub-Gaussian random variables with variance proxies $\tauvX \normzq{u_1} \leq \tauvX $ and $\tauvX \normzq{\f{A_i} - \mu^*} \, \normzq{u_2} \leq \tauvX \, \tau_n^2 $. Hence, in particular, the products are sub-Exponential with parameter bounded by $C_1 \, \tauvX \, \tau_n$ 
for a universal constant $C_1 > 0 $, see \citet[Lemma 2.7.7]{Vershynin2018}. Following the covering argument and applying the tail bound of sub-Exponential random variables as in \citet[Theorem 6.5]{Wainwright2019} leads to
\begin{align*}
	\Prob\Bigg( \, \Bigg\| \frac{1}{n}\, &\sum_{i=1}^n \Bigg( \Big( \f{X_i^\top} \big( \f{A_i} - \mu^* \big) \Big)  \ov \big(\f{X_i}\big) \, \f{X_i^\top} \, \one_{  \mathcal{T}_n (\tau_n) } \\
	&  \qquad - \E\bigg[ \Big( \f{X_i^\top} \big( \f{A_i} - \mu^* \big) \Big)  \ov \big(\f{X_i}\big) \, \f{X_i^\top} \, \one_{  \mathcal{T}_n (\tau_n) } \bigg] \Bigg) \Bigg\|_{\mathrm{M},2} \geq C_2\,\tau_n \, \frac{p}{\sqrt{n}} \Bigg) \leq C_3 \exp\big(-C_4 \, p^2 \big)
\end{align*}
for universal constants $C_2,C_3,C_4 >0$. Furthermore, we obtain for the third term in the sum in \eqref{proof:gradient:bounded:second:moments:3:growing:3} the estimate
\begin{align*}
	\Prob\Bigg( \bigg\| \frac{1}{n}\, \sum_{i=1}^n \Big( \f{X_i^\top} \big( \f{A_i} - \mu^* \big) \Big)  \ov \big(\f{X_i}\big) \, \f{X_i^\top} \, \one_{  \mathcal{T}_n^c (\tau_n) } \bigg\|_{\mathrm{M},2} \geq t \Bigg) \leq \Prob \big(\mathcal{T}_n^c (\tau_n)\big) \leq C_5 \, \frac{p^2\,n}{\tau_n^4} 
\end{align*}
for $t > 0$, since 
\begin{align}
	\Prob\big( \mathcal{T}_n^c (\tau_n) \big) &= \Prob \bigg( \bigcup_{i=1}^n \Big\{ \normz{\f{A_i} - \mu^*} > \tau_n \Big\} \bigg) \leq \sum_{i=1}^n \Prob\Big( \normz{\f{A_i} - \mu^*} > \tau_n \Big) \notag \\
	&= n \, \Prob\Big( \normz{\f{A} - \mu^*} > \tau_n \Big) \leq \frac{n\,\E\big[ \normzv{\f{A} - \mu^*} \big]}{\tau_n^4} \notag \\
	%&= \frac{n}{\tau_n^4} \,\E\Bigg[ \bigg( \sum_{k=1}^p \big(A_k - \mu_k^*\big)^2 \bigg)^2 \Bigg] 
	&= \frac{n}{\tau_n^4} \,\sum_{k,l=1}^p \E\Big[  \big(A_k - \mu_k^*\big)^2 \big(A_l - \mu_l^*\big)^2 \Big] 
	\leq C_5 \, \frac{p^2\,n}{\tau_n^4} \label{proof:gradient:bounded:second:moments:3:growing:4}
\end{align}
holds for a positive constant $C_5>0$ by Assumption \ref{moments:coefficients:growing}. Moreover, we obtain
\begin{align*}
	\E\bigg[ \Big( \f{X_i^\top} \big( \f{A_i} - \mu^* \big) \Big)  \ov \big(\f{X_i}\big) \, \f{X_i^\top} \, \one_{  \mathcal{T}_n (\tau_n) } \bigg] =  \E\bigg[ \Big( \f{X_i^\top} \big( \f{A_i} - \mu^* \big) \Big)  \ov \big(\f{X_i}\big) \, \f{X_i^\top} \, \big(  - \one_{  \mathcal{T}_n^c (\tau_n) } \big) \bigg]
\end{align*}
because 
\begin{align*}
	\E\bigg[ \Big( \f{X_i^\top} \big( \f{A_i} - \mu^* \big) \Big)  \ov \big(\f{X_i}\big) \, \f{X_i^\top} \bigg] = \E\bigg[ \Big( \f{X_i^\top} \E \big[ \f{A_i} - \mu^* \, \big| \, \f{X_i} \big] \Big)  \ov \big(\f{X_i}\big) \, \f{X_i^\top} \bigg] = \f{0}_{\frac{p(p+1)}{2} \times p} 
\end{align*}
is satisfied by the independence of $\f{X_i}$ and $\f{A_i}$. Hence the Cauchy Schwarz inequality implies for the second term in \eqref{proof:gradient:bounded:second:moments:3:growing:3} the estimate
\begin{align*}
	\Bigg\| \frac{1}{n}\, &\sum_{i=1}^n \E\bigg[ \Big( \f{X_i^\top} \big( \f{A_i} - \mu^* \big) \Big)  \ov \big(\f{X_i}\big) \, \f{X_i^\top} \, \one_{  \mathcal{T}_n (\tau_n) } \bigg] \Bigg\|_{\mathrm{M},2} \\
	&= \sup_{\substack{u_1 \in \R^{p(p+1)/2},u_2 \in \R^{p},\\ \normz{u_1}, \normz{u_2} \leq 1}} \frac{1}{n}\, \sum_{i=1}^n \E\bigg[ \Big( \f{X_i^\top} \big( \f{A_i} - \mu^* \big) \Big)  u_1^\top \ov \big(\f{X_i}\big) \, \f{X_i^\top} u_2 \, \big(  - \one_{  \mathcal{T}_n^c (\tau_n) } \big) \bigg]\\
	&\leq \sup_{\substack{u_1 \in \R^{p(p+1)/2},u_2 \in \R^{p},\\ \normz{u_1} , \normz{u_2} \leq 1}} \Bigg( \E\bigg[ \Big( \f{X^\top} \big( \f{A} - \mu^* \big) \Big)^2  \Big( u_1^\top \ov \big(\f{X}\big) \Big)^2 \big(\f{X^\top} u_2\big)^2 \bigg] \,  \Prob \big(\mathcal{T}_n^c (\tau_n)\big) \Bigg)^\frac{1}{2} \,.
\end{align*}
Further,
\begin{align*}
	&\sup_{\substack{u_1 \in \R^{p(p+1)/2},u_2 \in \R^{p},\\ \normz{u_1} , \normz{u_2} \leq 1}}  \E\bigg[ \Big( \f{X^\top} \big( \f{A} - \mu^* \big) \Big)^2  \Big( u_1^\top \ov \big(\f{X}\big) \Big)^2 \big(\f{X^\top} u_2\big)^2 \bigg]^\frac{1}{2} \\
	&\qquad \qquad \leq \sup_{\substack{u_1 \in \R^{p(p+1)/2},u_2 \in \R^{p},\\ \normz{u_1} , \normz{u_2} \leq 1}} \E\bigg[ \Big( \f{X^\top} \big( \f{A} - \mu^* \big) \Big)^4\bigg]^\frac{1}{4} \, \E\bigg[ \Big( u_1^\top \ov \big(\f{X}\big) \Big)^8 \bigg]^\frac{1}{8} \, \E\Big[ \big(\f{X^\top} u_2\big)^8 \Big]^\frac{1}{8}\\
	&\qquad \qquad \leq C_6 \, \E\Bigg[ \E\bigg[ \Big( \f{X^\top} \big( \f{A} - \mu^* \big) \Big)^4 \, \bigg| \, \f{A} \bigg] \Bigg]^\frac{1}{4}  \leq C_7 \, \E\Big[ \normzv{\f{A} - \mu^*}\Big]^\frac{1}{4} \leq C_8 \, \sqrt{p}\,,
\end{align*}
where $C_6,C_7,C_8>0$ are positive constants, since $\ov \big(\f{X}\big)$ and $\f{X}$ are sub-Gaussian and hence their moments exist, see \citet[Theorem 2.6]{Wainwright2019}. This implies together with \eqref{proof:gradient:bounded:second:moments:3:growing:4} the upper bound
\begin{align*}
	\Bigg\| \frac{1}{n}\, \sum_{i=1}^n \E\bigg[ \Big( \f{X_i^\top} \big( \f{A_i} - \mu^* \big) \Big)  \ov \big(\f{X_i}\big) \, \f{X_i^\top} \, \one_{  \mathcal{T}_n (\tau_n) } \bigg] \Bigg\|_{\mathrm{M},2} \leq C_8 \, \sqrt{p}\, \Big( \Prob \big(\mathcal{T}_n^c (\tau_n)\big) \Big)^\frac{1}{2} \leq C_8 \,\sqrt{C_5}\,\frac{p^{3/2} \, \sqrt{n}}{\tau_n^2}\,.
\end{align*}
So all in all collecting the terms leads to
\begin{align*}
	\Prob \Bigg( \bigg\| \frac{1}{n}\, &\sum_{i=1}^n \Big( \f{X_i^\top} \big( \f{A_i} - \mu^* \big) \Big)  \ov \big(\f{X_i}\big) \, \f{X_i^\top} \bigg\|_{\mathrm{M},2} \geq 2\, C_2\, \frac{\tau_n \,p}{\sqrt{n}} + 2 \, C_8 \,\sqrt{C_5}\,\frac{p^{3/2} \, \sqrt{n}}{\tau_n^2}  \Bigg) \\
	& \qquad \qquad \qquad \qquad \leq C_5 \, \frac{p^2\,n}{\tau_n^4} + C_3 \exp\big(-C_4 \, p^2 \big) \,.
\end{align*}
If $p^2 n /\tau_n^4 \to 0$ is satisfied, we obtain by \eqref{proof:gradient:bounded:second:moments:3:growing:2} the rate 
\begin{align*}
	\normz{Z_n^{\sigma,3}} = \Op{ \frac{\tau_n\, p^{3/2}}{n} + \frac{p^2}{\tau_n^2}} \,.
\end{align*}
Let $\tau_n = n^{3/8}$, then $p^2 n /\tau_n^4 = p^2/\sqrt{n} \to 0$ by Assumption \ref{ass:second:moments:limit:pn}, and 
\begin{align*}
	\normz{Z_n^{\sigma,3}} = \Op{ \frac{p^{3/2}}{n^{5/8}} + \frac{p^2}{n^{3/4}}} \,.
\end{align*}
%In particular, it follows that
%\begin{align*}
%	\sqrt{n}/p \, \normz{Z_n^{\sigma,3}} = \Op{\frac{p^{1/2}}{n^{1/8}}+ \frac{p}{n^{1/4}}} = \op{1}
%\end{align*}
%by Assumption \ref{ass:second:moments:limit:pn}.
\end{proof}

\begin{remark}\label{rem:technicalcondagain}	
	Suppose that the vector $\f{A}-\mu^*$ is sub-Gaussian with variance proxy $\tauA$, then we can use in the proof of Lemma \ref{lemma:gradient:bounded:second:moments:3:growing} the estimate 
	\begin{align*}
		\Prob \Big( \normz{\f{A}-\mu^*} > \tau_n \Big) = \Big( \sup_{v \in \R^p, \normz{v} \leq 1 } v^\top \big( \f{A}-\mu^* \big) > \tau_n \Big) \leq 6^p \exp\bigg(-\frac{\tau_n^2 }{8 \, \tauA}\bigg)  \,,
	\end{align*}
	see \citet[Theorem 1.19]{Rigollet2019}. Let $\tau_n = \sqrt{(C_p + \log(6) ) 8 \, \tauA\, p}$ with $C_p>0$, then
	\begin{align*}
		n \, \Prob \Big( \normz{\f{A}-\mu^*} > \tau_n \Big) \leq n \,\exp\bigg(-\frac{(C_p + \log(6) ) 8 \, \tauA\, p }{8 \, \tauA} + p \log(6) \bigg) = n \, \exp\big(-C_p \,p\big) ~ \to ~ 0 \,.
	\end{align*}
	Hence $\big\| Z_n^{\sigma,3}\big\|_2 = \Op{p^2/n} $ and, in particular, $\sqrt{n}/p \, \big\| Z_n^{\sigma,3}\big\|_2 = \Op{p / \sqrt{n}} = \op{1}$.
\end{remark}

\begin{proof}[Proof of Theorem \ref{theorem:adaptive:LASSO:second:moments:growing}]
	We shall use the primal-dual witness characterization of the adaptive LASSO in Lemma \ref{lemma:primal:dual:witness:adaptive:lasso} in the supplement, Section \ref{sec:appendix:estimators}, to prove the sign-consistency \eqref{sign:consistency:adaptive:LASSO:second:moments:growing}. We obtain by Assumption \ref{ass:regressors:subgaussian} and \citet[Theorem 6.5]{Wainwright2019} that
	\begin{align*}
		\normzM{\frac{1}{n} (\Xsigma)^\top \Xsigma - \mathrm{C}^\sigma} =  \Op{\sqrt{p(p+1)/n}} = \Op{p/\sqrt{n}}\,,
	\end{align*}
	which implies together with the Assumptions \ref{ass:second:moments:C} and \ref{ass:second:moments:limit:pn} the invertibility of the Gram matrix for large $n$, and hence by \citet[Lemma 11]{Loh2017} we get also
	\begin{align*}
		\normzM{ \bigg( \frac{1}{n} (\Xsigma)^\top \Xsigma \bigg)^{-1} - \big(\mathrm{C}^\sigma\big)^{-1}} = \Op{p/\sqrt{n}}\,.
	\end{align*}
	Furthermore, basic properties of the $\ell_2$ operator norm and Assumption \ref{ass:second:moments:C} lead to
	\begin{align*}
		&\normzM{\big(\X_{n,S_\sigma^c}^{\sigma}\big)^\top  \XsigmaS \Big( \big(\XsigmaS\big)^\top  \XsigmaS \Big)^{-1} - \mathrm{C}_{S_\sigma^c S_\sigma}^\sigma \big(\mathrm{C}_{S_\sigma S_\sigma}^\sigma\big)^{-1}} \\
		&  \quad \quad \quad =  \normzM{\frac{1}{n} \, \big(\X_{n,S_\sigma^c}^{\sigma}\big)^\top  \XsigmaS \bigg(\frac{1}{n} \, \big(\XsigmaS\big)^\top  \XsigmaS \bigg)^{-1} - \mathrm{C}_{S_\sigma^c S_\sigma}^{\sigma}  \big(\mathrm{C}_{S_\sigma S_\sigma}^{\sigma} \big)^{-1} } \\
		& \quad \quad \quad \leq  \Bigg( \normzM{\mathrm{C}_{S_\sigma^c S_\sigma}^{\sigma}} + \normzM{\frac{1}{n} \, \big(\X_{n,S_\sigma^c}^{\sigma}\big)^\top  \XsigmaS - \mathrm{C}_{S_\sigma^c S_\sigma}^{\sigma}} \Bigg) \\
		& \quad \quad \quad \quad \quad \quad \quad \quad \quad \quad \quad \quad \quad \quad \quad \cdot \normzM{\bigg(\frac{1}{n} \, \big(\XsigmaS\big)^\top  \XsigmaS \bigg)^{-1} - \big(\mathrm{C}_{S_\sigma S_\sigma}^{\sigma} \big)^{-1}} \\
		& \quad \quad \quad \quad \quad \quad \quad \quad + \normzM{\frac{1}{n} \, \big(\X_{n,S_\sigma^c}^{\sigma}\big)^\top  \XsigmaS - \mathrm{C}_{S_\sigma^c S_\sigma}^{\sigma}} \, \normzM{\big(\mathrm{C}_{S_\sigma S_\sigma}^{\sigma} \big)^{-1}}  = \Op{p/\sqrt{n}} \,.
	\end{align*} 
	In particular, this implies
	\begin{align}
		\normzM{ \bigg( \frac{1}{n} (\Xsigma)^\top \Xsigma \bigg)^{-1}} = \Op{1}\,, \qquad \normzM{\big(\X_{n,S_\sigma^c}^{\sigma}\big)^\top  \XsigmaS \Big( \big(\XsigmaS\big)^\top  \XsigmaS \Big)^{-1}} = \Op{1}\,. \label{proof:growing:1}
	\end{align}
	
	Moreover, let $\widehat{\sigma}_{n,\min}^{\,\init} \defeq \min_{k \in S_{\sigma}} |\widehat{\sigma}_{n,k}^{\,\init}|$, then
	\begin{align*}
		\bigg| \frac{\widehat{\sigma}_{n,\min}^{\,\init} - \sigma_{\min}^*}{\sigma_{\min}^*} \bigg| \leq \frac{1}{\sigma_{\min}^*} \, \normz{\widehat{\sigma}_{n}^{\,\init} - \sigma^*} = \Op{\frac{p}{\sigma_{\min}^* \, \sqrt{n}}} = \op{1}
	\end{align*}
	since $\sqrt{n}/p  \, \normz{ \widehat{\sigma}_{n}^{\,\init} - \sigma^* } = \Op{1}$ and $p/(\sigma_{\min}^* \, \sqrt{n}) \to 0$. This implies
	\begin{align*}
		\bigg( 1 + \frac{\widehat{\sigma}_{n,\min}^{\,\init} - \sigma_{\min}^*}{\sigma_{\min}^*}\bigg)^{-1} = \Op{1}\,,
	\end{align*}
	see \citet[Section 2.2]{Vaart1998}. Hence we obtain
	\begin{align}
		\frac{\sqrt{n}}{p}\, \normz{\lambda_n^{\sigma} \, \bigg( \frac{1}{ |\widehat{\sigma}_{n,S_{\sigma}}^{\,\init}|} \odot \sign\big(\sigma_{S_\sigma}^*\big) \bigg)} &\leq \frac{\sqrt{n} \, \lambda_n^{\sigma}}{p} \, \normz{\frac{1}{ |\widehat{\sigma}_{n,S_{\sigma}}^{\,\init}|}} \leq \frac{\sqrt{s_\sigma \, n} \, \lambda_n^{\sigma}}{p} \, \normi{\frac{1}{ |\widehat{\sigma}_{n,S_{\sigma}}^{\,\init}|}} \notag \\
		&= \frac{\sqrt{s_\sigma \, n} \, \lambda_n^{\sigma}}{p} \, \big(\widehat{\sigma}_{n,\min}^{\,\init}\big)^{-1} 
%		&= \frac{\sqrt{s_\sigma \, n} \, \lambda_n^{\sigma}}{p} \, \big(\sigma_{\min}^*\big)^{-1} \, \bigg(1 + \frac{\widehat{\sigma}_{n,\min}^{\,\init} - \sigma_{\min}^*}{\sigma_{\min}^*}\bigg)^{-1} \notag  \\
%		&= \frac{\sqrt{s_\sigma \, n} \, \lambda_n^{\sigma}}{\sigma_{\min}^* \,p} \, \Op{1} \notag \\
		= \op{1} \label{proof:growing:3}
	\end{align}
	since $\sqrt{s_\sigma \, n} \, \lambda_n^{\sigma} / (\sigma_{\min}^* \,p) \to 0 $ by assumption. It follows that
	\begin{align}
		&\frac{\sqrt{n}}{p} \, \normz{ \big(\X_{n,S_\sigma^c}^{\sigma}\big)^\top  \XsigmaS \Big( \big(\XsigmaS\big)^\top  \XsigmaS \Big)^{-1} \Bigg( \lambda_n^{\sigma} \, \bigg( \frac{1}{|\widehat{\sigma}_{n,S_{\sigma}}^{\,\init}| } \odot \sign\big(\sigma_{S_\sigma}^*\big) \bigg) \Bigg) + \frac{1}{n} \, \big(\X_{n,S_\sigma^c}^{\sigma}\big)^\top \proj{\XsigmaS} \, \epssigma } \notag \\
		& \quad \leq \normzM{ \big(\X_{n,S_\sigma^c}^{\sigma}\big)^\top  \XsigmaS \Big( \big(\XsigmaS\big)^\top  \XsigmaS \Big)^{-1}} \, \frac{\sqrt{n}}{p} \, \normz{ \lambda_n^{\sigma} \, \bigg( \frac{1}{|\widehat{\sigma}_{n,S_{\sigma}}^{\,\init}| } \odot \sign\big(\sigma_{S_\sigma}^*\big) \bigg) }\notag \\
		& \quad \quad  + \frac{\sqrt{n}}{p} \, \normz{\frac{1}{n} \, \big(\X_{n,S_\sigma^c}^{\sigma}\big)^\top  \epssigma} + \normzM{ \big(\X_{n,S_\sigma^c}^{\sigma}\big)^\top  \XsigmaS \Big( \big(\XsigmaS\big)^\top  \XsigmaS \Big)^{-1}} \, \frac{\sqrt{n}}{p} \, \normz{\frac{1}{n} \, \big(\X_{n,S_\sigma}^{\sigma}\big)^\top  \epssigma} \notag \\
		& \quad = \Op{1} \,\op{1} +  \Op{1} + \Op{1}  = \Op{1}\,, \label{proof:growing:2}
	\end{align}         
	by Lemmas \ref{lemma:gradient:bounded:second:moments:1:growing} - \ref{lemma:gradient:bounded:second:moments:3:growing} and \eqref{proof:growing:1}, where 
	\begin{align*}
		\proj{\XsigmaS} = \mathrm{I}_n - \XsigmaS \Big( \big(\XsigmaS\big)^\top \XsigmaS \Big)^{-1} \big(\XsigmaS\big)^\top \,.
	\end{align*}
	Furthermore, it is 
\begin{align*}
	\frac{\big|\widehat{\sigma}_{n,k}^{\,\init}\big|}{\lambda_{n}^{\sigma}} \leq \frac{\big\|\widehat{\sigma}_{n,S_\sigma^c}^{\,\init}\big\|_2 }{\lambda_{n}^{\sigma}} = \frac{\big\|\widehat{\sigma}_{n,S_\sigma^c}^{\,\init} - \sigma_{S_\sigma^c}^*\big\|_2 }{\lambda_{n}^{\sigma}} \leq \frac{\big\|\widehat{\sigma}_{n}^{\,\init} - \sigma^*\big\|_2 }{\lambda_{n}^{\sigma}} = \frac{ \sqrt{n}/p \, \big\|\widehat{\sigma}_{n}^{\,\init} - \sigma^*\big\|_2 }{ (\sqrt{n}/p) \, \lambda_{n}^{\sigma}}
\end{align*}
	for all $k \in S^c$. The condition $\sqrt{n}/p \, \normz{ \widehat{\sigma}_{n}^{\,\init} - \sigma^* } = \Op{1}$ together with $n\,\lambda_n^\sigma/p^2 \to \infty$ implies the convergence
\begin{align*}
	\frac{\big|\widehat{\sigma}_{n,k}^{\,\init}\big|}{ (\sqrt{n}/p) \, \lambda_{n}^{\sigma}} = \frac{1}{n\,\lambda_n^\sigma/p^2} \, \Op{1} = \op{1} \,. 
\end{align*}
	Hence it follows by \eqref{proof:growing:2} that the first condition \eqref{primal:dual:witness:adaptive:lasso:1} of Lemma \ref{lemma:primal:dual:witness:adaptive:lasso} is satisfied with high probability for a sufficient large sample size $n$. Furthermore, let
	\begin{align*}
		\widetilde{\sigma}_{n,S_\sigma} =  \sigma_{S_\sigma}^* + \bigg(\frac{1}{n} \, \big(\XsigmaS\big)^\top  \XsigmaS \bigg)^{-1} \Bigg( \frac{1}{n} \, \big(\XsigmaS\big)^\top \epssigma -  \lambda_n^{\sigma} \, \bigg( \frac{1}{ |\widehat{\sigma}_{n,S_{\sigma}}^{\,\init}|} \odot \sign\big(\sigma_{S_\sigma}^*\big) \bigg) \Bigg) \,.
	\end{align*}
	Then we obtain
	\begin{align*}
		\frac{\sqrt{n}}{p} \, \normz{ \widetilde{\sigma}_{n,S_\sigma} -  \sigma_{S_\sigma}^* } &\leq  \normzM{\bigg(\frac{1}{n} \, \big(\XsigmaS\big)^\top  \XsigmaS \bigg)^{-1}} \Bigg( \frac{\sqrt{n}}{p} \, \normz{\frac{1}{n} \, \big(\X_{n,S_\sigma}^{\sigma}\big)^\top  \epssigma}  \notag \\
		& \quad \quad \quad \quad \quad \quad + \frac{\sqrt{n}}{p} \, \normz{\lambda_n^{\sigma} \, \bigg( \frac{1}{ |\widehat{\sigma}_{n,S_{\sigma}}^{\,\init}|} \odot \sign\big(\sigma_{S_\sigma}^*\big) \bigg) } \Bigg) \notag \\
		&= \Op{1} \big( \Op{1} + \op{1} \big) = \Op{1}
	\end{align*}
	by \eqref{proof:growing:1}, \eqref{proof:growing:3} and Lemmas \ref{lemma:gradient:bounded:second:moments:1:growing} - \ref{lemma:gradient:bounded:second:moments:3:growing}. In particular, this implies 
	\begin{align*}
		\normz{ \widetilde{\sigma}_{n,S_\sigma} -  \sigma_{S_\sigma}^* } = \Op{p/\sqrt{n}}=\op{1} 
	\end{align*}
	by Assumption \ref{ass:second:moments:limit:pn}, and hence the second condition, $\sign\big(\widetilde{\sigma}_{n,S_\sigma}\big)= \sign\big(\sigma_{S_\sigma}^*\big)$, of Lemma \ref{lemma:primal:dual:witness:adaptive:lasso} is also satisfied with high probability for large sample sizes $n$. Sign-consistency of the adaptive LASSO and $\widehat{\sigma}_{n,S_\sigma}^{\,\AL} = \widetilde{\sigma}_{n,S_\sigma}$ is the consequence.
\end{proof}

\newpage

\appendix

\section{Supplement: Proofs for Section \ref{sec:fixedparasymp}} \label{sec:fixedp:proofs}

\subsubsection*{Proof of Proposition \ref{lemma:Csigma:positive:definite}}

\begin{proof}[Proof of Proposition \ref{lemma:Csigma:positive:definite}] %\hfill\\
From Theorem \ref{the:identcartprodt}, under the assumptions of the proposition the matrix
\[ 	 S = \bigg[ \ov \Big( \big(1, \f{W_1^\top} \big)^\top \Big) ,\dotsc, \ov \Big( \big(1, \f{W_{p(p+1)/2}^\top} \big)^\top\Big) \bigg]^\top %\quad \in \R^{\frac{p(p+1)}{2}  \times \frac{p(p+1)}{2}} 
\]
is of full rank with positive probability.
Therefore, the random positive semi-definite matrix 
\begin{align*}
	\frac{1}{n} \, \big(\Xsigma\big)^\top  \Xsigma = \frac{1}{n} \sum_{i=1}^{n} \ov\Big( \big(1,\f{W_i^\top} \big)^\top \Big)\,\ov\Big( \big(1,\f{W_i^\top} \big)^\top \Big)^\top 
\end{align*}
for $n \geq p(p+1)/2$ is positive definite with positive probability. Hence its expected value, which equals $\mathrm{C}^{\sigma} $, is positive definite.  
\end{proof}

%
%The results for the adaptive LASSO in Theorem \ref{theorem:adaptive:LASSO:first:moments} and \ref{theorem:adaptive:LASSO:second:moments} are very similar. Hence also the main steps of the proofs are identically. Thus, we mainly prove the results for the estimation of the variances in Section \ref{sec:fixedp:second:moments} in detail, and merely outline the ones for the estimation of the means in Section \ref{sec:fixedp:first:moments}. \\
%
%At first we consider the equivalence of the convergence of a matrix and the convergence of its entries. The proof is deferred to the supplement, Section \ref{sec:proof:lemma:convergence:matrices}.

%\begin{remark} \label{remark:eps:sigma:decomposition} %\hfill\\

\subsubsection*{Proof of Theorem \ref{theorem:adaptive:LASSO:second:moments}}

	Turning to the proof of Theorem \ref{theorem:adaptive:LASSO:second:moments}, recall the decomposition \eqref{eq:errordecomp} of  the error term \eqref{def:epssigma}. %can be decomposed in three different terms, namely

\begin{lemma} \label{lemma:gradient:bounded:second:moments:2} %\hfill\\
	%Suppose that Assumption \ref{moments:second:moments:B1} and \ref{moments:second:moments:B2} hold. Then the random vectors $Z_n^{\sigma,2}$ in 
	Under the conditions of Theorem \ref{theorem:adaptive:LASSO:second:moments}, we have that 
	\begin{align*}
		\frac{1}{\sqrt{n}} \, \big(\Xsigma\big)^\top \big( \zeta_n + \xi_n \big) = \op{1} \, .
	\end{align*}
\end{lemma}

The proofs of the previous as well as the following lemma are deferred to the end of this section.

\begin{lemma} \label{lemma:gradient:bounded:second:moments:1} %\hfill\\
%	Suppose that Assumption \ref{moments:second:moments:B1} holds. Then the random vectors $Z_n^{\sigma,1}$ in \eqref{def:parts:gradient:second:moments} satisfy 
%
Set $Z_n^{\sigma,1} = \frac{1}{\sqrt{n}} \, \big(\Xsigma\big)^\top \delta_n$, then 
	\begin{align*}
		\E\big[ Z_n^{\sigma,1} \, \big| \, \Xsigma \big] = \f{0}_{p(p+1)/2} \quad \text{and} \quad \cov\big( Z_n^{\sigma,1} \, \big| \, \Xsigma \big) = \frac{1}{n}\, \big(\Xsigma\big)^\top \Omega_n^{\sigma} \, \Xsigma \,  ,
	\end{align*}
where $\Omega_n^{\sigma}$ is a diagonal matrix with entries $\ov( \f{X_1} )^\top \Psi^* \, \ov ( \f{X_1} ),\dotsc,\ov( \f{X_n} )^\top \Psi^* \, \ov ( \f{X_n} )$. In particular, $\cov(Z_n^{\sigma,1})= \mathrm{B}^\sigma$ and $Z_n^{\sigma,1} = \Op{1} $.
%	\end{align*}
\end{lemma}
%

%The asymptotic normality in Theorem \ref{theorem:adaptive:LASSO:first:moments} and \ref{theorem:adaptive:LASSO:second:moments} is based on the following central limit theorem of random vectors.
%
%\begin{proposition}[Lindeberg-Feller central limit theorem] \label{theorem:central:limit:random:vectors} \hfill\\
%	For each $n \in \N$ let $Q_{n,1},\dotsc,Q_{n,k_n}$ be independent random vectors with finite variances such that
%	\begin{align*}
%		\lim_{n \to \infty} \normiM{ \sum_{i=1}^{k_n} \cov \big( Q_{n,i} \big) - \Gamma } = 0
%	\end{align*}
%	and for every $\delta > 0$
%	\begin{align}
%		\lim_{n \to \infty} \sum_{i=1}^{k_n} \E \Big[ \normzq{Q_{n,i}} ~ \mathbbm{1} \big\{ \normz{Q_{n,i}} > \delta  \big\} \Big] = 0 \, .  \label{lindeberg:central:limit:random:vectors}
%	\end{align}
%	Then the sequence $\sum_{i=1}^{k_n} \big( Q_{n,i} - \E[Q_{n,i}] \big)$ converges in distribution to a normal $\mathcal{N}(0, \Gamma)$ distribution.
%\end{proposition}
%
%\begin{proof} %\hfill\\
%	Cf. \citet[Proposition 2.27]{Vaart1998}.
%\end{proof}

%Finally we prove Theorem \ref{theorem:adaptive:LASSO:first:moments} and \ref{theorem:adaptive:LASSO:second:moments}.

\begin{proof}[Proof of Theorem \ref{theorem:adaptive:LASSO:second:moments}] %\hfill\\
	We shall use the primal-dual witness characterization of the adaptive LASSO in Lemma \ref{lemma:primal:dual:witness:adaptive:lasso} in the supplement, Section \ref{sec:appendix:estimators}, to prove the sign-consistency \eqref{sign:consistency:adaptive:LASSO:second:moments}, and the Lindeberg-Feller central limit theorem for random vectors, see \citet[Proposition 2.27]{Vaart1998}, to prove the asymptotic normality \eqref{asymptotic:normality:adaptive:LASSO:second:moments}. For more details see also the proof of Theorem \ref{theorem:adaptive:LASSO:second:moments:growing} if necessary.
	By Lemmas \ref{lemma:gradient:bounded:second:moments:2} and \ref{lemma:gradient:bounded:second:moments:1}, setting
		\begin{align*}
		\proj{\XsigmaS} = \mathrm{I}_n - \XsigmaS \Big( \big(\XsigmaS\big)^\top \XsigmaS \Big)^{-1} \big(\XsigmaS\big)^\top \,,
	\end{align*}
	%
	% together with Remark \ref{remark:decomposition:gradient:second:moments} it follows that
	%
	we have that
	\begin{align*}
		\frac{1}{\sqrt{n}}\, \big(\X_{n,S_\sigma^c}^{\sigma}\big)^\top \proj{\XsigmaS} \, \epssigma = \Op{1}\,. % + \Op{1} = \Op{1}\,.  \label{proof:theorem:LASSO:second:moments:4}
	\end{align*}
	In addition, the requirements $\sqrt{n} \, \lambda_n^{\sigma} \to 0 $ and $\sqrt{n} \, \big( \widehat{\sigma}_{n}^{\,\init} - \sigma^* \big) = \Op{1}$ in Theorem \ref{theorem:adaptive:LASSO:second:moments} lead to
	\begin{align}
	0 \leq \frac{\sqrt{n} \, \lambda_{n}^{\sigma}}{\big|\widehat{\sigma}_{n,k}^{\,\init}\big|}  \leq \frac{\sqrt{n} \, \lambda_{n}^{\sigma}}{ \Big| \big|\sigma_{k}^*\big| - \big| \widehat{\sigma}_{n,k}^{\,\init} - \sigma_{k}^*\big| \Big| } \stackrel{\Prob} \to 0 \label{proof:theorem:adaptive:LASSO:second:moments:1}
	\end{align}
	for all $k \in S_{\sigma}$ since $| \sigma_k^* | >0$ for these $k$.  %Together with %\eqref{proof:theorem:LASSO:second:moments:2} and 
	%\eqref{proof:theorem:LASSO:second:moments:4} T
	This implies
	\begin{align}
		\sqrt{n} \, \Bigg[ \big(\X_{n,S_\sigma^c}^{\sigma}\big)^\top  &\XsigmaS \Big( \big(\XsigmaS\big)^\top  \XsigmaS \Big)^{-1} \Bigg( \lambda_n^{\sigma} \, \bigg( \frac{1}{ |\widehat{\sigma}_{n,S_{\sigma}}^{\,\init}| } \odot \sign\big(\sigma_{S_\sigma}^*\big) \bigg) \Bigg) + \frac{1}{n} \, \big(\X_{n,S_\sigma^c}^{\sigma}\big)^\top \proj{\XsigmaS} \, \epssigma \Bigg] \notag \\
		&= \Op{1} \,\op{1} +  \Op{1} =  \Op{1}\,. \label{proof:theorem:adaptive:LASSO:second:moments:2}
	\end{align} 
	Moreover, $\sqrt{n} \, \big( \widehat{\sigma}_{n}^{\,\init} - \sigma^* \big) = \Op{1}$ implies also $\sqrt{n} \, \widehat{\sigma}_{n,k}^{\,\init}= \Op{1} $ for all $k \in S_{\sigma}^c$ since $\sigma_k^* = 0$ for these $k$. Thus, by the second requirement $n \, \lambda_n^{\sigma} \to \infty$ on the regularization parameter it follows that
	\begin{align*}
	\frac{\sqrt{n} \, \lambda_{n}^{\sigma}}{\big|\widehat{\sigma}_{n,k}^{\,\init}\big|} = \frac{n \, \lambda_{n}^{\sigma}}{\sqrt{n} \,\big|\widehat{\sigma}_{n,k}^{\,\init}\big|}\stackrel{\Prob} \to \infty 
	\end{align*}
	for all $k \in S_{\sigma}^c$. Together with \eqref{proof:theorem:adaptive:LASSO:second:moments:2} this implies the first condition \eqref{primal:dual:witness:adaptive:lasso:1} of Lemma \ref{lemma:primal:dual:witness:adaptive:lasso} with high probability for a sufficient large number $n$ of observations. Furthermore, let
	\begin{align*}
	\widetilde{\sigma}_{n,S_\sigma} =  \sigma_{S_\sigma}^* + \bigg(\frac{1}{n} \, \big(\XsigmaS\big)^\top  \XsigmaS \bigg)^{-1} \Bigg( \frac{1}{n} \, \big(\XsigmaS\big)^\top \epssigma -  \lambda_n^{\sigma} \, \bigg( \frac{1}{ |\widehat{\sigma}_{n,S_{\sigma}}^{\,\init}| } \odot \sign\big(\sigma_{S_\sigma}^*\big) \bigg) \Bigg) \,.
	\end{align*}
	Then we obtain
	\begin{align*}
	\sqrt{n} \, \big( \widetilde{\sigma}_{n,S_\sigma} -  \sigma_{S_\sigma}^* \big) = \bigg(\frac{1}{n} \, \big(\XsigmaS\big)^\top  \XsigmaS \bigg)^{-1} \frac{1}{\sqrt{n}} \, \big(\XsigmaS\big)^\top \epssigma  + \op{1} %\label{proof:theorem:adaptive:LASSO:second:moments:5}
	\end{align*}
	by %\eqref{proof:theorem:LASSO:second:moments:1} and 
	\eqref{proof:theorem:adaptive:LASSO:second:moments:1}. Moreover, with Lemmas \ref{lemma:gradient:bounded:second:moments:2} and \ref{lemma:gradient:bounded:second:moments:1} it follows that
	\begin{align}
	\sqrt{n} \, \big( \widetilde{\sigma}_{n,S_\sigma} -  \sigma_{S_\sigma}^* \big) &= \bigg(\frac{1}{n} \, \big(\XsigmaS\big)^\top  \XsigmaS \bigg)^{-1} \frac{1}{\sqrt{n}} \, \big(\XsigmaS\big)^\top \delta_n + \op{1} \label{proof:theorem:adaptive:LASSO:second:moments:4}  \\
	&= \Op{1} + \op{1} = \Op{1} \,, \notag
	\end{align}
	which leads to
%	\begin{align*}
$	\widetilde{\sigma}_{n,S_\sigma} -  \sigma_{S_\sigma}^*   = \op{1} \,.$
%	\end{align*}
	Therefore the second condition, $\sign\big(\widetilde{\sigma}_{n,S_\sigma}\big)= \sign\big(\sigma_{S_\sigma}^*\big)$, of Lemma \ref{lemma:primal:dual:witness:adaptive:lasso} is also satisfied with high probability for large $n$. Sign-consistency of the adaptive LASSO and $\widehat{\sigma}_{n,S_\sigma}^{\,\AL} = \widetilde{\sigma}_{n,S_\sigma}$ is the consequence. 
	
	\medskip
	
	Note that for the asymptotic normality \eqref{asymptotic:normality:adaptive:LASSO:second:moments} of the rescaled estimation error only the first term in \eqref{proof:theorem:adaptive:LASSO:second:moments:4} is crucial. Hence we consider the random vectors
	\begin{align*}
	Z_n^{\sigma,1} = \frac{1}{\sqrt{n}} \, \big(\X_{n}^{\sigma}\big)^\top \delta_n = \frac{1}{\sqrt{n}} \sum_{i=1}^n \big( e_i^\top \delta_n \big) \, \ov (\f{X_i}) = \frac{1}{\sqrt{n}} \sum_{i=1}^n  \Big( \ov(\f{X_i})^\top \ovec \big(D_i-\Sigma^* \big) \Big) \, \ov (\f{X_i}) \,,
	\end{align*}
	where $D_i = \big(\f{A_i} - \mu^* \big) \big(\f{A_i}- \mu^*\big)^\top$ and $\delta_n$ is defined in \eqref{def:deltan}. Now we want to apply the Lindeberg-Feller central limit theorem for the array 
	\begin{align*}
	Q_{n,i} = \frac{1}{\sqrt{n}} \, \Big( \ov(\f{X_i})^\top \ovec \big(D_i-\Sigma^* \big) \Big) \, \ov (\f{X_i}) \,, \qquad i = 1,\dotsc,n\,,
	\end{align*}
	of random vectors. These are independent and identically distributed in each row (for fixed $n$) since $(\f{X_1^\top},\f{A_1^\top})^\top, \dotsc, (\f{X_n^\top}, \f{A_n^\top})^\top$ are independent and identically distributed. Furthermore, they are centered, 
	\begin{align*}
	\E\big[Q_{n,i}\big] = \frac{1}{\sqrt{n}} \, \E\bigg[ \E\Big[ \ov(\f{X_i})^\top \ovec \big(D_i-\Sigma^* \big) \, \Big| \, \Xsigma \Big] \, \ov (\f{X_i}) \bigg] = \frac{1}{\sqrt{n}} \, \E\big[  0 \cdot \ov (\f{X_i}) \big] = \f{0}_{p(p+1)/2} \, ,
	\end{align*}
	and for the sum of the covariance matrices 
	\begin{align*}
	\sum_{i=1}^{n} \cov \big( Q_{n,i} \big) = \cov \Bigg( \sum_{i=1}^{n}  Q_{n,i} \Bigg) = \cov\big( Z_n^{\sigma,1} \big)
	\end{align*}
	we get by Lemma \ref{lemma:gradient:bounded:second:moments:1}
	\begin{align*}
	\sum_{i=1}^{n} \cov \big( Q_{n,i} \big) = \mathrm{B}^{\sigma}  \,.
	\end{align*}
	Moreover, we obtain for arbitrary $\delta>0$ the equation
	\begin{align*}
	\sum_{i=1}^{n} \E \bigg[ \normzq{Q_{n,i}} ~ &\mathbbm{1} \big\{ \normz{Q_{n,i}} > \delta \big\} \bigg] = \E \bigg[ \ov \big( \f{X} \big)^\top \ovec \big(D-\Sigma^*\big) \, \ov \big( \f{X} \big)^\top \ovec \big(D-\Sigma^*\big) \, \ov \big( \f{X} \big)^\top \ov \big( \f{X} \big)  \\
	& \qquad \quad \cdot \mathbbm{1} \big\{ \ov ( \f{X} )^\top \ovec (D-\Sigma^*) \, \ov ( \f{X} )^\top \ovec (D-\Sigma^*) \, \ov ( \f{X} )^\top \ov ( \f{X} )  > \delta^2 n  \big\} \bigg] \, . 
	\end{align*}
	The expected mean $\E \big[ \ov ( \f{X} )^\top \ovec (D-\Sigma^*) \, \ov ( \f{X} )^\top \ovec (D-\Sigma^*) \, \ov ( \f{X} )^\top \ov ( \f{X} ) \big]$ exists because of Assumption \ref{ass:moments:second:moments} and the Cauchy Schwarz inequality. Thus we get
	\begin{align*}
	\lim_{n \to \infty} \sum_{i=1}^{n} \E \bigg[ \normzq{Q_{n,i}} ~ \mathbbm{1} \big\{ \normz{Q_{n,i}} > \delta \big\} \bigg] = 0 
	\end{align*}
	by Lebesgue's dominated convergence theorem, which coincides with Lindeberg's condition, see \citet[Proposition 2.27]{Vaart1998}. Hence the mentioned proposition implies the weak convergence
	\begin{align*}
	Z_n^{\sigma,1} = \frac{1}{\sqrt{n}} \, \big(\X_{n}^{\sigma}\big)^\top \delta_n = \sum_{i=1}^n Q_{n,i} ~ \stackrel{d} \longrightarrow ~ Q \sim \mathcal{N}_{{p(p+1)/2}} \Big( \f{0}_{p(p+1)/2} , \mathrm{B}^{\sigma} \Big)\,,
	\end{align*}
	respectively
	\begin{align*}
	\frac{1}{\sqrt{n}} \, \big(\XsigmaS\big)^\top \delta_n ~ \stackrel{d} \longrightarrow ~ Q_{S_\sigma} \sim \mathcal{N}_{s_\sigma} \big( \f{0}_{s_{\sigma}} , \mathrm{B}_{S_{\sigma} S_{\sigma}}^{\sigma} \big) \, .
	\end{align*}
	So all in all a multivariate version of Slutsky's theorem, see for example \citet[Theorem 2.7, Lemma 2.8]{Vaart1998}, together with equation \eqref{proof:theorem:adaptive:LASSO:second:moments:4} %and the almost sure convergence \eqref{proof:theorem:LASSO:second:moments:1} 
	leads to
	\begin{align*}
	\sqrt{n} \, \big( \widehat{\sigma}_{n,S_{\sigma}}^{\,\AL} -  \sigma_{S_{\sigma}}^* \big) ~ \stackrel{d} \longrightarrow ~  \big( \mathrm{C_{S_{\sigma} S_{\sigma} }^{\sigma}} \big)^{-1} \, Q_{S_\sigma}\,.
	\end{align*}
		In addition, it follows that
	\begin{align*}
		\big( \mathrm{C_{S_{\sigma} S_{\sigma} }^{\sigma}} \big)^{-1} \, Q_{S_\sigma} \sim \mathcal{N}_{s_\sigma} \Big( \f{0}_{s_{\sigma}} , \big( \mathrm{C_{S_{\sigma} S_{\sigma} }^{\sigma}} \big)^{-1} \mathrm{B}_{S_{\sigma} S_{\sigma}}^{\sigma} \big( \mathrm{C_{S_{\sigma} S_{\sigma} }^{\sigma}} \big)^{-1} \Big)
	\end{align*}
	by the symmetry of $\mathrm{C_{S_{\sigma} S_{\sigma} }^{\sigma}}$ and the properties of the multivariate normal distribution, and hence the asymptotic normality \eqref{asymptotic:normality:adaptive:LASSO:second:moments}.
\end{proof}

\begin{proof}[Proof of Lemma \ref{lemma:gradient:bounded:second:moments:2}]% \label{sec:proof:gradient:ls:second:moments}
We prove Lemma \ref{lemma:gradient:bounded:second:moments:2} in two steps. First we show that 
	\begin{align}\label{lemeq:helpone}
	\frac{1}{\sqrt{n}} \, \big(\Xsigma\big)^\top \zeta_n = \op{1} \, .
\end{align}
We obtain 
	\begin{align}
		\frac{1}{\sqrt{n}} \, \big(\Xsigma\big)^\top \zeta_n &= \frac{1}{\sqrt{n}} \sum_{i=1}^n \big( e_i^\top \zeta_n \big) \ov \big(\f{X_i}\big) = \frac{1}{\sqrt{n}} \sum_{i=1}^n \Big( \ov\big(\f{X_i}\big)^\top \ovec\big( E_n \big) \Big)  \ov \big(\f{X_i}\big) \notag\\
		&= \frac{1}{\sqrt{n}} \sum_{i=1}^n \bigg( \sum_{q=1}^{p(p+1)/2} \ov\big(\f{X_i}\big)_q \ovec\big( E_n \big)_q \bigg) \ov \big(\f{X_i}\big)  \notag\\
		&= \sum_{q=1}^{p(p+1)/2} \sqrt{n} \, \ovec\big( E_n \big)_q \bigg( \frac{1}{n} \sum_{i=1}^n \ov\big(\f{X_i}\big)_q \, \ov \big(\f{X_i}\big) \bigg)\, ,\label{proof:1:lemma:gradient:bounded:second:moments:3}
	\end{align}
	where 
	\begin{align*}
		E_n = \big( \mu^* - \widehat{\mu}_n \big) \big( \mu^* - \widehat{\mu}_n \big)^\top \,.
	\end{align*}
By the assumption on $\widehat{\mu}_{n}$ we get
%	\begin{align*}
$		e_k^\top E_n \, e_l = \big( \widehat{\mu}_{n,k} - \mu_k^* \big) \big( \widehat{\mu}_{n,l} - \mu_l^* \big) = \Op{1/n}$
%	\end{align*}
	for $k,l \in \{1\dotsc,p\}$, % under the conditions of Theorem \ref{theorem:adaptive:LASSO:first:moments} 
	 and hence also
	\begin{align}
		\sqrt{n} \, \ovec\big( E_n \big)_q = \Op{\frac{1}{\sqrt{n}}} \label{proof:3:lemma:gradient:bounded:second:moments:3}
	\end{align}
	for all $q \in \{1,\dotsc,p(p+1)/2\}$. Furthermore, the random vectors
	%\begin{align*}
		$Q_i^q = \ov\big(\f{X_i}\big)_q \, \ov \big(\f{X_i}\big)$ % \quad \in \R^{\frac{p(p+1)}{2}}
	%\end{align*}
	are independent and identically distributed with 
	\begin{align*}
		\E \Big[ \normz{Q_i^q} \Big] \leq \E \Big[\norme{Q_i^q} \Big] = \E \bigg[ \norme{ \ov\big(\f{X_i}\big)_q \, \ov \big(\f{X_i}\big) } \bigg] =  \sum_{r=1}^{p(p+1)/2} \E \bigg[ \Big| \ov \big(\f{X_i}\big)_r \, \ov\big(\f{X_i}\big)_q \Big| \bigg] < \infty  \,,
	\end{align*}   
%	since the fourth moments of the covariates exist by Assumption \ref{moments:second:moments:B2}, 
so that by the law of large numbers %and Lemma \ref{lemma:convergence:matrices} the convergence
%	\begin{align*}
%		\normi{\frac{1}{n} \sum_{i=1}^n \ov\big(\f{X_i}\big)_q \, \ov \big(\f{X_i}\big)  - \E\big[Q_1^q\big] } = \normi{ \frac{1}{n} \sum_{i=1}^n  Q_i^q - \E\big[Q_1^q\big] } ~ \stackrel{a.s.} \longrightarrow ~ 0
%	\end{align*}
%	and therefore
	\begin{align}
		\frac{1}{n} \sum_{i=1}^n \ov\big(\f{X_i}\big)_q \, \ov \big(\f{X_i}\big)  = \Op{1} \label{proof:4:lemma:gradient:bounded:second:moments:3}
	\end{align}
	for all $q \in \{1,\dotsc,p(p+1)/2\}$ follows. In summary,	 \eqref{proof:1:lemma:gradient:bounded:second:moments:3}, \eqref{proof:3:lemma:gradient:bounded:second:moments:3} and \eqref{proof:4:lemma:gradient:bounded:second:moments:3} lead to \eqref{lemeq:helpone}. 
	
	\medskip 
	
%\end{proof}

%\begin{lemma}\label{lemma:gradient:bounded:second:moments:4} %\hfill\\
%	Suppose that Assumption \ref{moments:second:moments:B1} and \ref{moments:second:moments:B2} hold. Then the random vectors 
%	\begin{align*}
%		Z_n^{\sigma,4} \defeq \,,
%	\end{align*}
%	where $\xi_n$ is defined in \eqref{def:xin} with $\widehat{\mu}_n = \widehat{\mu}_n^{\,\AL}$, converge under the conditions of Theorem \ref{theorem:adaptive:LASSO:first:moments} in probability zu zero,
%	\begin{align*}
%		Z_n^{\sigma,4} = \op{1} \, .
%	\end{align*}
%\end{lemma} 
%
%
%	We can rewrite the random vectors $Z_n^{\sigma,3}$ with the help of \eqref{def:xin} by	
%
In the second step, consider 
	\begin{align*}
		\frac{1}{\sqrt{n}} \, \big(\Xsigma\big)^\top \xi_n  &= \frac{1}{\sqrt{n}} \sum_{i=1}^n  \bigg( \sum_{q=1}^{p(p+1)/2} \ov\big(\f{X_i}\big)_q \ovec\big( F_{n,i} \big)_q \bigg) \ov \big(\f{X_i}\big)\,, 
%		&=  \frac{1}{\sqrt{n}} \sum_{i=1}^n  \bigg( \sum_{k,l=1}^p X_{i,k} X_{i,l} \big(F_{n,i}\big)_{kl }\bigg) \ov \big(\f{X_i}\big) \,,
	\end{align*}
	where
%	\begin{align*}
$		F_{n,i} = \big( \f{A_i} - \mu^* \big) \big( \mu^* - \widehat{\mu}_n \big)^\top + \big( \mu^* - \widehat{\mu}_n \big) \big( \f{A_i} - \mu^* \big)^\top \,.$
%	\end{align*}
	Then we obtain analogously 
	\begin{align*}
		\frac{1}{\sqrt{n}} \, \big(\Xsigma\big)^\top \xi_n %&= \frac{1}{\sqrt{n}} \sum_{i=1}^n \Bigg(  2 \sum_{k,l=1}^p   X_{i,k} X_{i,l} \, \big( \mu_{k}^* -\widehat{\mu}_{n,{k}}^{\,\AL} \big) \big( A_{i,{l}} -\mu_{l}^*\big)  \Bigg) \ov \big(\f{X_i}\big) \notag \\
		&= \sum_{k,l=1}^{p}  \sqrt{n}\,\big( \widehat{\mu}_{n,{k}} - \mu_{k}^* \big) \bigg( -\frac{2}{n} \sum_{i=1}^n  X_{i,k} X_{i,l} \, \big( A_{i,{l}} -\mu_{l}^*\big) \, \ov \big(\f{X_i}\big) \bigg) \\
		&= \Op{1}\, \op{1} = \op{1},
		%\, \label{proof:1:lemma:gradient:bounded:second:moments:4} 
	\end{align*}
since
\[ \E \Big[X_{1,k} X_{1,l} \, \big( A_{1,{l}} -\mu_{l}^*\big) \, \ov \big(\f{X_1}\big)\Big] = 0\]
by the independence of $\f{X_1}$ and $\f{A_1}$.
\end{proof}

\begin{proof}[Proof of Lemma \ref{lemma:gradient:bounded:second:moments:1}] %\hfill\\
	
	We obtain by simple calculation $\E\big[ \delta_n \, \big| \, \Xsigma \big] = \f{0}_{n} $ and
	%	\begin{align}
	$ \cov\big( \delta_n \, \big| \,\Xsigma \big) = \Omega_n^{\sigma}$, hence
	%= \diag \Big( \ov\big(\f{X_1}\big)^\top \Psi^* \, \ov\big(\f{X_1}\big),\dotsc,\ov\big(\f{X_n}\big)^\top \Psi^* \, \ov\big(\f{X_n}\big) \Big) \,, \label{cov:epsilon:second:moments}
	%	\end{align}
	%holds. The entries of $\delta_n$ are actually independent.
	%
	%	Consider the definition \eqref{def:parts:gradient:second:moments} of the random variables $Z_n^{\sigma,1}$. Under Assumption \ref{moments:second:moments:B1} we obtain 
	\begin{align*}
		\E\big[ Z_n^{\sigma,1} \, \big| \, \Xsigma \big] =  \frac{1}{\sqrt{n}} \,  \big(\Xsigma\big)^\top \E\big[ \delta_n \, \big| \, \Xsigma \big] = \f{0}_{p(p+1)/2}
	\end{align*}
	and 
	\begin{align*}
		\cov\big( Z_n^{\sigma,1} \, \big| \, \Xsigma \big) = \frac{1}{n} \, \big(\Xsigma\big)^\top  \cov\big(\delta_n \, \big| \, \Xsigma \big) \,  \Xsigma = \frac{1}{n} \, \big(\Xsigma\big)^\top \Omega_n^{\sigma} \, \Xsigma   \,.
	\end{align*}
	%by Lemma \ref{lemma:cov:epsilon:second:moments}. 
	For random variables $Q_1,Q_2$ and $Q_3$ the law of total covariance implies the decomposition
	\begin{align*}
		\cov\big(Q_1,Q_2\big) = \E \Big[\cov\big(Q_1,Q_2\,\big|\,Q_3\big)\Big] + \cov\Big(\E\big[Q_1\,\big|\,Q_3\big], \E\big[Q_2\,\big|\,Q_3\big]\Big) \, .
	\end{align*}
	This can be extended to random vectors and covariance matrices and hence we obtain
	\begin{align*}
		\cov\big(Z_n^{\sigma,1}\big) &=  \E\Big[\cov\big( Z_n^{\sigma,1} \, \big| \, \Xsigma \big)\Big] + \cov\Big(\E\big[ Z_n^{\sigma,1} \, \big| \, \Xsigma \big] \Big) \\
		&= \E\bigg[ \frac{1}{n} \, \big(\Xsigma\big)^\top \Omega_n^{\sigma} \, \Xsigma \bigg] = \mathrm{B}^\sigma\,.
		% + \cov\Big( \f{0}_{\frac{p(p+1)}{2}} \Big) \\
		%	&= \E\bigg[ \frac{1}{n} \, \big(\Xsigma\big)^\top \Omega_n^{\sigma} \, \Xsigma \bigg] \,.
	\end{align*} 
	%Moreover, Assumption \ref{moments:second:moments:B2} (existence of the eighth moments of $\f{X}$ is sufficient) and \eqref{mean:matrix:B:second:moments} in Remark \ref{remark:limiting:matrix:second:moments} lead to
	%	\begin{align*}
	%	\cov\big(Z_n^{\sigma,1}\big) = \E\bigg[ \frac{1}{n} \, \big(\Xsigma\big)^\top \Omega_n^{\sigma} \, \Xsigma \bigg] = \mathrm{B}^{\sigma}  \,.
	%	\end{align*}
	%
	Boundedness in probability follows since by the law of large numbers, 
	\begin{align*}
		\frac{1}{n} \, \big(\Xsigma\big)^\top \Omega_n^{\sigma} \, \Xsigma = \frac{1}{n} \sum_{i=1}^{n}  \Big( \ov\big( \f{X_i} \big)^\top \Psi^* \, \ov\big( \f{X_i} \big) \Big) \, \ov\big( \f{X_i} \big) \, \ov\big( \f{X_i} \big)^\top \to \mathrm{B}^{\sigma} \,.
	\end{align*}
	
\end{proof}

\section{Supplement: The adaptive LASSO} \label{sec:appendix:estimators}

We look for a fixed number $n \in \N$ of observations at the ordinary linear regression model
\begin{align*}
\Y_n = \X_n \, \beta^* + \varepsilon_n \,, %\label{general:linear:regression:model}
\end{align*}
where $\Y_n \in \R^n$ is the vector of the response variables, $\X_n \in \R^{n \times p}$ the deterministic design matrix, $\beta^* \in \R^p$ the unknown coefficient vector and $\varepsilon_n \in \R^n$ represents additive noise. Moreover, we allow the coefficients $\beta^*$ to be sparse, in other words it is $s \leq p $ for
\begin{align*}
S = \supp\big(\beta^*\big) = \Big\{ k \in \{1,\dotsc,p\} \, \big| \, \beta_k^* \neq 0 \Big\} \,, \quad \quad s = |S| \,.
\end{align*}  
In addition, let $S^c = \{1,\dotsc,p\} \setminus S$ be the relative complement of $S$. Because of the sparsity of the coefficients the linear regression model can also be expressed by 
\begin{align*}
\Y_n = \X_{n,S} \, \beta_S^* + \varepsilon_n \,.%\label{general:linear:regression:model:sparse}
\end{align*}

Consider the adaptive LASSO estimator with regularization parameter $\lambda_n > 0$, given by
\begin{align*}
\widehat{\beta}_n^{\,\AL} \in \rho_{n,\lambda_n}^{\,\AL} &\defeq \underset{\beta \in \R^p}{\arg\min} ~ \Bigg( \frac{1}{n} \, \normzq{\Y_n - \X_n \, \beta}  +  2 \lambda_n  \sum_{k=1}^p \frac{|\beta_k|}{\big|\widehat{\beta}_{n,k}^{\,\init}\big|} \Bigg) \,,
\end{align*}
where $\widehat{\beta}_n^{\,\init} \in \R^p$ is an initial estimator of $\beta^*$. If $\widehat{\beta}_{n,k}^{\,\init} = 0$, we require $\beta_k=0$ in the above definition.
%\begin{lemma}[Karush–Kuhn–Tucker (KKT) condition of the LASSO]\label{lemma:kkt:lasso} \hfill\\
%There exists a vector $\gamma \in \R^p$ with $\gamma_k \in [-1,1]$ for $k \in \{1,\dotsc,p\}$, so that the equivalence 
%\begin{align*}
%\widehat{\beta}_n^{\,\LA} \in \rho_{n,\lambda_n}^{\,\LA} \quad \Leftrightarrow \quad \frac{1}{n} \, \X_n^\top \big( \Y_n - \X_n \widehat{\beta}_n^{\,\LA} \big) = \lambda_n \gamma ~ \mathrm{and} ~ \gamma_k = \begin{cases}
%\sign\big(\widehat{\beta}_{n,k}^{\,\LA}\big) & \mathrm{for}~ \widehat{\beta}_{n,k}^{\,\LA} \neq 0 \, ,\\
%\in [-1,1] & \mathrm{for}~ \widehat{\beta}_{n,k}^{\,\LA} = 0 \, ,
%\end{cases} 
%\end{align*} 
%holds.
%\end{lemma}
%
%\begin{proof} %\hfill\\
%Cf. Lemma 12.1 in \citet{Zhou2009} with $\vec{w}=(1,\dotsc,1)^\top \in \R^p$.
%\end{proof}
%%\hfill
%
\begin{lemma}[Primal-dual witness characterization of the adaptive LASSO] \label{lemma:primal:dual:witness:adaptive:lasso} 
Assume $s \leq n$ and $\rank(\X_{n,S}) = s$. If
\begin{align}
\Bigg| \X_{n,S^c}^\top \X_{n,S} \Big( \X_{n,S}^\top \X_{n,S} \Big)^{-1} \lambda_n \Bigg( \frac{1}{\big|\widehat{\beta}_{n,S}^{\,\init}\big|} \odot \sign\big(\beta_S^*\big) \Bigg) + \frac{1}{n} \,  \X_{n,S^c}^\top \orthproj{\X_{n,S}} \, \varepsilon_n \Bigg| < \frac{\lambda_n}{\big|\widehat{\beta}_{n,S^c}^{\,\init}\big|} \label{primal:dual:witness:adaptive:lasso:1}
\end{align}
with
\begin{align*}
\orthproj{\X_{n,S}} \defeq \mathrm{I}_n - \X_{n,S} \Big( \X_{n,S}^\top \X_{n,S} \Big)^{-1} \X_{n,S}^\top
\end{align*}
holds, and 
\begin{align*}
\widetilde{\beta}_{n,S} =  \beta_S^* + \bigg( \frac{1}{n} \, \X_{n,S}^\top \X_{n,S} \bigg)^{-1} \Bigg( \frac{1}{n} \, \X_{n,S}^\top \, \varepsilon_n - \lambda_n \bigg( \frac{1}{\big|\widehat{\beta}_{n,S}^{\,\init}\big|} \odot \sign\big(\beta_S^*\big) \bigg) \Bigg) %\label{primal:dual:witness:adaptive:lasso:2}
\end{align*} 
satisfies $\sign\big(\widetilde{\beta}_{n,S}\big) = \sign\big(\beta_{S}^*\big)$, then the unique adaptive LASSO solution  $\rho_{n,\lambda_n}^{\,\AL} = \big\{ \widehat{\beta}_n^{\,\AL} \big\}$ satisfies
\begin{align*}
\sign\big(\widehat{\beta}_{n}^{\,\AL}\big) =  \sign\big(\beta^*\big)\, ,~ \widehat{\beta}_{n,S}^{\,\AL} = \widetilde{\beta}_{n,S} ~\mathrm{and}~ \widehat{\beta}_{n,S^c}^{\,\AL} = \f{0}_{|S^c|}\, .
\end{align*}
\end{lemma}

\begin{proof} %\hfill\\
Cf. Lemma 12.1 in \citet{Zhou2009} with $\vec{w}=\big(1/|\widehat{\beta}_{n,1}^{\, \init}|,\dotsc,1/|\widehat{\beta}_{n,p}^{\, \init}| \big)^\top \in \R^p$.
\end{proof}         

\section{Supplement: Estimating the means with diverging number $p$ of parameters} \label{sec:appendix:meandiv}

The model is given in vector-matrix form by
\begin{align*}
	\Ymu = \Xmu \, \mu^* + \epsmu \,, %\label{random:coefficient:model:second:moments:matrix}
\end{align*}
%
%
%\begin{align}
%Y_i^{\sigma} = \ov\big(\f{X_i}\big)^\top \sigma^* +  \ov\big(\f{X_i}\big)^\top \ovec\big( D_i - \Sigma^* + E_n + F_i \big)\,, \quad i=1,\dotsc,n \,,\label{random:coefficient:model:second:moments}
%\end{align} 
%where $\sigma^* = \ovec(\Sigma^*)$. Note that the deterministic coefficient vector $\sigma^*$ is $s_\sigma$-sparse. The errors are heteroscedastic and, moreover, they aren't independent since they depend all on the estimate $\widehat{\mu}_n$. Let
where
\begin{align*}
	\Ymu &\defeq \big( Y_1 , \dotsc, Y_n  \big)^\top ,\quad 
	\Xmu \defeq \big[ \f{X_1}, \dotsc, \f{X_n} \big]^\top , \quad
	\epsmu \defeq \Big( \f{X_1^\top} \big( \f{A_1} - \mu^* \big) ,\dotsc, 
	\f{X_n^\top} \big( \f{A_n} - \mu^* \big)  \Big)^\top. 
\end{align*}

Then the adaptive LASSO estimator with regularization parameter $\lambda_n^{\mu}>0$ is given by
\begin{align}
	\widehat{\mu}_n^{\,\AL} \in \rho_{\mu,n,\lambda_n^{\mu}}^{\,\AL} &\defeq \underset{\beta \in \R^{p} }{\arg\min} ~ \Bigg( \frac{1}{n} \, \normzq{\Ymu - \Xmu \, \beta} +  2 \lambda_n^{\mu}  \sum_{k=1}^{p} \frac{|\beta_k|}{\big|\widehat{\mu}_{n,k}^{\,\init}\big|} \Bigg) \,, \label{def:adaptive:LASSO:first:moments}
\end{align}
where $\widehat{\mu}_n^{\,\init} \in \R^p$ is an initial estimator of $\mu^*$. Note that if $\widehat{\mu}_{n,k}^{\,\init} = 0$, we require again $\beta_k=0$. Further, let
\begin{align*}
\mathrm{C}^{\mu} \defeq \E\big[ \f{X} \, \f{X^\top}\big]\,, \qquad \qquad \mathrm{B}^{\mu} \defeq \E\Big[ \big( \f{X^\top} \Sigma^* \, \f{X} \big) \,\f{X} \, \f{X^\top}\Big] \,,
\end{align*}
and we denote by
\begin{align*}
 S_{\mu} &\defeq \supp(\mu^*) = \Big\{ k \in \{1,\dotsc,p\} \, \big| \, \mu_k^* \neq 0 \Big\}\,, \qquad s_{\mu} \defeq | S_{\mu} |\,,
\end{align*}
the support of the mean vector $\mu^*$. $S_{\mu}^c \defeq \{1,\dotsc,p\} \setminus S_{\mu}$ is again the corresponding relative complement.

\begin{assumption}[Growing $p$] \label{ass:moments:first:moments:growing} %\hfill\\
	We assume that $(\f{X_i^\top}, \f{A_i^\top})^\top$, $i=1, \ldots, n$, are identically distributed, and that 
	\begin{enumerate}[label=\normalfont{(A\arabic*)},leftmargin=9.9mm]
		\setcounter{enumi}{\value{temp1}}
		\item the random coefficients $\f{A}$ have finite second moments,
		\item the covariate vector $\f{X}$ is sub-Gaussian, \label{ass:regressors:subgaussian:1}
		\item $\Cmul \leq \lambda_{\min} \big( \mathrm{C}^{\mu} \big) \leq \lambda_{\max} \big( \mathrm{C}^{\mu} \big) \leq \Cmuu
		$ for some positive constants $ 0 < \Cmul \leq \Cmuu < \infty$, where $\lambda_{\min}(A)$ and $\lambda_{\max}(A)$ denote the minimal and maximal eigenvalues of a symmetric matrix $A$, \label{ass:first:moments:C}
		\item  $\lambda_{\max} \big(\mathrm{B}^{\mu} \big) \leq \Bmuu$ for some positive constant $\Bmuu > 0$, \label{ass:first:moments:B}
		\item $\lim_{n\to \infty} p / n = 0$. \label{ass:first:moments:limit:pn} 
		\setcounter{temp1}{\value{enumi}}
	\end{enumerate}
\end{assumption}

\begin{theorem}[Variable selection for growing $p$] \label{theorem:adaptive:LASSO:first:moments:growing} %\hfill\\
%
%Suppose that the estimator $\widehat{\mu}_{n} $ of $\mu^*$ used in the residuals $\widetilde{Y}_i$ is $\sqrt{n}$-consistent, that is, $\sqrt{n} \, \big( \widehat{\mu}_{n} - \mu^* \big) = \Op{1}$. 
%
%	Consider the linear regression model \eqref{random:coefficient:model:second:moments} with $\widehat{\mu}_n = \widehat{\mu}_n^{\, \AL}$, where $\widehat{\mu}_n^{\, \AL}$ is a solution to \eqref{def:adaptive:LASSO:first:moments} with initial estimator $\widehat{\mu}_n^{\,\init}$ that satisfies  and regularization parameter $\lambda_n^{\mu}$ that satisfies $\lambda_n^{\mu} \to 0$, $\sqrt{n} \, \lambda_n^{\mu} \to 0$ and $n \, \lambda_n^{\mu} \to \infty$.
Let Assumption \ref{ass:moments:first:moments:growing} be satisfied, and assume that for the initial estimator $\widehat{\mu}_{n}^{\,\init}$ in the adaptive LASSO $\widehat{\mu}_n^{\,\AL}$ in \eqref{def:adaptive:LASSO:first:moments} we have \linebreak $\sqrt{n/p} \, \normz{ \widehat{\mu}_{n}^{\,\init} - \mu^* } = \Op{1}$. Moreover, if the regularization parameter is chosen as $\lambda_n^{\mu} \to 0$, 
%{\color{red} such that $\gamma > 2\nu/(1-\nu)$,  
%$\sqrt{n} \,\lambda_n^{\mu} \to 0$, $n^{\frac{(1-\nu)(1+\gamma)}{2}} \, \lambda_n^{\mu} \to \infty$ (for fixed $p$ same conditions as in Theorem \ref{theorem:adaptive:LASSO:second:moments}, $\nu=0$, $\gamma=1$)} and 
%{\color{red}	\min \Bigg(  \frac{1}{\sqrt{p} \, \lambda_n^{\,\init}}, \bigg(\frac{1}{\sqrt{p\,n} \, \lambda_n^{\mu}}\bigg)^\frac{1}{\gamma}\Bigg) \, \mu_{\min}^* \to \infty\,,} \\
\begin{align*}
	\sqrt{s_\mu \, n} \, \lambda_n^{\mu} \, / (\mu_{\min}^* \, \sqrt{p}) \to 0 \,, \quad \sqrt{p}/(\mu_{\min}^* \, \sqrt{n}) \to 0\,,\quad n\,\lambda_n^\mu/p \to \infty
    \end{align*} 
with $\mu_{\min}^* \defeq \min_{k \in S_{\mu}} |\mu_k^*|$, then it follows that $\widehat{\mu}_n^{\,\AL}$ is sign-consistent, 
\begin{align}
	\mathbb{P}\Big( \sign\big(\widehat{\mu}_n^{\,\AL}\big) = \sign\big(\mu^*\big)\Big) \to 1 \, . \label{sign:consistency:adaptive:LASSO:first:moments:growing}
\end{align}
\end{theorem}  

For the proof of Theorem \ref{theorem:adaptive:LASSO:first:moments:growing} we need the following auxiliary lemma.

\begin{lemma} \label{lemma:gradient:bounded:first:moments:growing} %\hfill\\
	%	Suppose that Assumption \ref{moments:second:moments:B1} holds. Then the random vectors $Z_n^{\sigma,1}$ in \eqref{def:parts:gradient:second:moments} satisfy 
	%
	Set $Z_n^{\mu} = \frac{1}{n} \, \big(\Xmu\big)^\top \epsmu$, then $\normz{Z_n^{\mu}} = \operatorname{\mathcal{O}_\mathbb{P}} \big(\sqrt{p/n} \big) $. 
	%In addition, it is clear that we obtain also $\big\| \frac{1}{n}\big(\XmuS\big)^\top \epsmu \big\|_2 = \Op{\sqrt{s_\mu/n}}$.
	%	\end{align*}
\end{lemma}

\begin{proof}[Proof of Lemma \ref{lemma:gradient:bounded:first:moments:growing}]
	It is
	\begin{align*}
		\E \Big[ \normzq{Z_n^{\mu}} \Big] &= \frac{1}{n^2} \, \E \Big[ (\epsmu)^\top \Xmu \, (\Xmu)^\top \epsmu \Big] = \frac{1}{n^2} \, \E \Big[ \tr \big( (\Xmu)^\top \epsmu \, (\epsmu)^\top \Xmu \big) \Big] \\
		&= \frac{1}{n^2} \, \E \Big[ \tr \big( (\Xmu)^\top \E\big[ \epsmu \, (\epsmu)^\top \, \big| \,\Xmu \big] \, \Xmu \big) \Big]\\
		&=   \frac{1}{n} \, \tr \Bigg( \E \bigg[ \frac{1}{n} \, (\Xmu)^\top \Omega_n^{\mu} \,\Xmu \bigg] \Bigg) \,,
	\end{align*}
	where $\Omega_n^{\mu}=\cov\big(\epsmu\, \big| \,\Xmu \big)$ is a diagonal matrix with entries $\f{X_1^\top} \Sigma^* \, \f{X_1^\top} ,\dotsc,\f{X_n^\top} \Sigma^* \, \f{X_n^\top}$. It is obvious that
	\begin{align*}
		\E \bigg[ \frac{1}{n} \, (\Xmu)^\top \Omega_n^{\mu} \,\Xmu \bigg] = \mathrm{B}^\mu\,,
	\end{align*}
	and hence we obtain by Assumption \ref{ass:first:moments:B} the estimate
	\begin{align*}
		\E \Big[ \normzq{Z_n^{\mu}} \Big] &= \frac{\tr \big( \mathrm{B}^\mu \big)}{n} \leq \frac{\lambda_{\max} \big(\mathrm{B}^{\mu} \big) \, p}{n} \leq \frac{\Bmuu \, p}{n} \,. 
	\end{align*}
	Markov's inequality implies the assertion.
\end{proof}

\begin{proof}[Proof of Theorem \ref{theorem:adaptive:LASSO:first:moments:growing}]
We shall use the primal-dual witness characterization of the adaptive LASSO in Lemma \ref{lemma:primal:dual:witness:adaptive:lasso} in Section \ref{sec:appendix:estimators} to prove the sign-consistency \eqref{sign:consistency:adaptive:LASSO:first:moments:growing}. We obtain by Assumption \ref{ass:regressors:subgaussian:1} and \citet[Theorem 6.5]{Wainwright2019} that
\begin{align*}
\normzM{\frac{1}{n} (\Xmu)^\top \Xmu - \mathrm{C}^\mu} = \Op{\sqrt{p/n}}\,,
\end{align*}
which implies together with the Assumptions \ref{ass:first:moments:C} and \ref{ass:first:moments:limit:pn} the invertibility of the Gram matrix for large $n$, and hence by \citet[Lemma 11]{Loh2017} we get also
\begin{align*}
	\normzM{ \bigg( \frac{1}{n} (\Xmu)^\top \Xmu \bigg)^{-1} - \big(\mathrm{C}^\mu\big)^{-1}} = \Op{\sqrt{p/n}}\,.
\end{align*}
Furthermore, basic properties of the $\ell_2$ operator norm lead to
\begin{align*}
\normzM{\big(\X_{n,S_\mu^c}^{\mu}\big)^\top  \XmuS \Big( \big(\XmuS\big)^\top  \XmuS \Big)^{-1} - \mathrm{C}_{S_\mu^c S_\mu}^\mu \big(\mathrm{C}_{S_\mu S_\mu}^\mu\big)^{-1}} = \Op{\sqrt{p/n}} \,.
\end{align*} 
In particular, this implies
\begin{align}
\normzM{ \bigg( \frac{1}{n} (\Xmu)^\top \Xmu \bigg)^{-1}} = \Op{1}\,, \qquad \normzM{\big(\X_{n,S_\mu^c}^{\mu}\big)^\top  \XmuS \Big( \big(\XmuS\big)^\top  \XmuS \Big)^{-1}} = \Op{1}\,. \label{proof:first:growing:1}
\end{align}

Moreover, let $\widehat{\mu}_{n,\min}^{\,\init} \defeq \min_{k \in S_{\mu}} |\widehat{\mu}_{n,k}^{\,\init}|$, then
\begin{align*}
\bigg| \frac{\widehat{\mu}_{n,\min}^{\,\init} - \mu_{\min}^*}{\mu_{\min}^*} \bigg| \leq \frac{1}{\mu_{\min}^*} \, \normz{\widehat{\mu}_{n}^{\,\init} - \mu^*} = \Op{\frac{\sqrt{p}}{\mu_{\min}^* \, \sqrt{n}}} = \op{1}
\end{align*}
since $\sqrt{n/p} \, \normz{ \widehat{\mu}_{n}^{\,\init} - \mu^* } = \Op{1}$ and $\sqrt{p}/(\mu_{\min}^* \, \sqrt{n}) \to 0$. This implies
\begin{align*}
\bigg( 1 + \frac{\widehat{\mu}_{n,\min}^{\,\init} - \mu_{\min}^*}{\mu_{\min}^*}\bigg)^{-1} = \Op{1}\,,
\end{align*}
and hence we obtain
\begin{align}
\sqrt{\frac{n}{p}} \, \normz{\lambda_n^{\mu} \, \bigg( \frac{1}{ |\widehat{\mu}_{n,S_{\mu}}^{\,\init}|} \odot \sign\big(\mu_{S_\mu}^*\big) \bigg)} &\leq \frac{\sqrt{n} \, \lambda_n^{\mu}}{\sqrt{p}} \, \normz{\frac{1}{ |\widehat{\mu}_{n,S_{\mu}}^{\,\init}|}} \leq \frac{\sqrt{s_\mu \, n} \, \lambda_n^{\mu}}{\sqrt{p}} \, \normi{\frac{1}{ |\widehat{\mu}_{n,S_{\mu}}^{\,\init}|}} \notag \\
&= \frac{\sqrt{s_\mu \, n} \, \lambda_n^{\mu}}{\sqrt{p}} \, \big(\widehat{\mu}_{n,\min}^{\,\init}\big)^{-1} \notag \\
&= \frac{\sqrt{s_\mu \, n} \, \lambda_n^{\mu}}{\sqrt{p}} \, \big(\mu_{\min}^*\big)^{-1} \, \bigg(1 + \frac{\widehat{\mu}_{n,\min}^{\,\init} - \mu_{\min}^*}{\mu_{\min}^*}\bigg)^{-1} \notag  \\
&= \frac{\sqrt{s_\mu \, n} \, \lambda_n^{\mu}}{\mu_{\min}^* \,\sqrt{p}} \, \Op{1} \notag \\
&= \op{1} \label{proof:first:growing:3}
\end{align}
since $\sqrt{s_\mu \, n} \, \lambda_n^{\mu} / (\mu_{\min}^* \,\sqrt{p}) \to 0 $ by assumption. It follows that
	\begin{align}
	&\sqrt{\frac{n}{p}} \, \normz{ \big(\X_{n,S_\mu^c}^{\mu}\big)^\top  \XmuS \Big( \big(\XmuS\big)^\top  \XmuS \Big)^{-1} \Bigg( \lambda_n^{\mu} \, \bigg( \frac{1}{|\widehat{\mu}_{n,S_{\mu}}^{\,\init}| } \odot \sign\big(\mu_{S_\mu}^*\big) \bigg) \Bigg) + \frac{1}{n} \, \big(\X_{n,S_\mu^c}^{\mu}\big)^\top \proj{\XmuS} \, \epsmu } \notag \\
	& \quad \leq \normzM{ \big(\X_{n,S_\mu^c}^{\mu}\big)^\top  \XmuS \Big( \big(\XmuS\big)^\top  \XmuS \Big)^{-1}} \, \sqrt{\frac{n}{p}} \, \normz{ \lambda_n^{\mu} \, \bigg( \frac{1}{|\widehat{\mu}_{n,S_{\mu}}^{\,\init}| } \odot \sign\big(\mu_{S_\mu}^*\big) \bigg)}\notag \\
	& \quad \quad  + \sqrt{\frac{n}{p}} \, \normz{\frac{1}{n} \, \big(\X_{n,S_\mu^c}^{\mu}\big)^\top  \epsmu} + \normzM{ \big(\X_{n,S_\mu^c}^{\mu}\big)^\top  \XmuS \Big( \big(\XmuS\big)^\top  \XmuS \Big)^{-1}} \, \sqrt{\frac{n}{p}} \, \normz{\frac{1}{n} \, \big(\X_{n,S_\mu}^{\mu}\big)^\top  \epsmu} \notag \\
	& \quad = \Op{1} \,\op{1} +  \Op{1} + \Op{1} \notag \\
	& \quad = \Op{1} \label{proof:first:growing:2}
\end{align}         
by Lemma \ref{lemma:gradient:bounded:first:moments:growing} and \eqref{proof:first:growing:1}, where 
\begin{align*}
	\proj{\XmuS} = \mathrm{I}_n - \XmuS \Big( \big(\XmuS\big)^\top \XmuS \Big)^{-1} \big(\XmuS\big)^\top \,.
\end{align*}
Furthermore, it is 
\begin{align*}
	\frac{\big|\widehat{\mu}_{n,k}^{\,\init}\big|}{\lambda_{n}^{\mu}} \leq \frac{\big\|\widehat{\mu}_{n,S_\mu^c}^{\,\init}\big\|_2 }{\lambda_{n}^{\mu}} = \frac{\big\|\widehat{\mu}_{n,S_\mu^c}^{\,\init} - \mu_{S_\mu^c}^*\big\|_2 }{\lambda_{n}^{\mu}} \leq \frac{\big\|\widehat{\mu}_{n}^{\,\init} - \mu^*\big\|_2 }{\lambda_{n}^{\mu}} = \frac{ \sqrt{n/p} \, \big\|\widehat{\mu}_{n}^{\,\init} - \mu^*\big\|_2 }{ \sqrt{n/p} \, \lambda_{n}^{\mu}}
\end{align*}
for all $k \in S^c$. The condition $\sqrt{n/p} \, \normz{ \widehat{\mu}_{n}^{\,\init} - \mu^* } = \Op{1}$ together with $n\,\lambda_n^\mu/p \to \infty$ implies the convergence
\begin{align*}
	\frac{\big|\widehat{\mu}_{n,k}^{\,\init}\big|}{ \sqrt{n/p} \, \lambda_{n}^{\mu}} = \frac{1}{n\,\lambda_n^\mu/p} \, \Op{1} = \op{1} \,. 
\end{align*}
Hence it follows by \eqref{proof:first:growing:2} that the first condition \eqref{primal:dual:witness:adaptive:lasso:1} of Lemma \ref{lemma:primal:dual:witness:adaptive:lasso} is satisfied with high probability for a sufficient large sample size $n$. Furthermore, let
\begin{align*}
	\widetilde{\mu}_{n,S_\mu} =  \mu_{S_\mu}^* + \bigg(\frac{1}{n} \, \big(\XmuS\big)^\top  \XmuS \bigg)^{-1} \Bigg( \frac{1}{n} \, \big(\XmuS\big)^\top \epsmu -  \lambda_n^{\mu} \, \bigg( \frac{1}{ |\widehat{\mu}_{n,S_{\mu}}^{\,\init}|} \odot \sign\big(\mu_{S_\mu}^*\big) \bigg) \Bigg) \,.
\end{align*}
Then we obtain
\begin{align*}
	\sqrt{\frac{n}{p}} \, \normz{ \widetilde{\mu}_{n,S_\mu} -  \mu_{S_\mu}^* } &\leq  \normzM{\bigg(\frac{1}{n} \, \big(\XmuS\big)^\top  \XmuS \bigg)^{-1}} \Bigg( \sqrt{\frac{n}{p}} \, \normz{\frac{1}{n} \, \big(\X_{n,S_\mu}^{\mu}\big)^\top  \epsmu}  \notag \\
	& \quad \quad \quad \quad \quad \quad + \sqrt{\frac{n}{p}} \, \normz{\lambda_n^{\mu} \, \bigg( \frac{1}{ |\widehat{\mu}_{n,S_{\mu}}^{\,\init}|} \odot \sign\big(\mu_{S_\mu}^*\big) \bigg) } \Bigg) \notag \\
	&= \Op{1} \big( \Op{1} + \op{1} \big) = \Op{1}
\end{align*}
by \eqref{proof:first:growing:1}, \eqref{proof:first:growing:3} and Lemma \ref{lemma:gradient:bounded:first:moments:growing}. In particular, this implies 
\begin{align*}
	\normz{ \widetilde{\mu}_{n,S_\mu} -  \mu_{S_\mu}^* } = \Op{\sqrt{p/n}}=\op{1} 
\end{align*}
by Assumption \ref{ass:first:moments:limit:pn}, and hence the second condition, $\sign\big(\widetilde{\mu}_{n,S_\mu}\big)= \sign\big(\mu_{S_\mu}^*\big)$, of Lemma \ref{lemma:primal:dual:witness:adaptive:lasso} is also satisfied with high probability for large sample sizes $n$. Sign-consistency of the adaptive LASSO and $\widehat{\mu}_{n,S_\mu}^{\,\AL} = \widetilde{\mu}_{n,S_\mu}$ is the consequence.
\end{proof}
	
\end{document}